\documentclass[12pt, a4paper]{article}
\usepackage[top=2cm, bottom=2cm, left=2cm, right=2cm]{geometry}
\usepackage{amsthm,amsmath}
\usepackage{harrymacros}
\usepackage{enumitem}
\usepackage{tikz-cd}

\DeclareMathOperator{\Id}{Id}
\DeclareMathOperator{\Var}{Var}
\DeclareMathOperator{\BV}{BV}
\DeclareMathOperator{\Leb}{Leb}
\DeclareMathOperator{\LY}{LY}

\DeclareMathOperator{\End}{End}
\DeclareMathOperator{\vspan}{span}
\DeclareMathOperator{\codim}{codim}
\DeclareMathOperator{\ess}{ess}
\DeclareMathOperator*{\esssup}{ess\,sup}
\DeclareMathOperator{\gap}{Gap}

\newcommand{\wnorm}[1]{\abs{#1}}
\newcommand{\into}{\hookrightarrow}

\theoremstyle{plain}
\newtheorem{innercustomthm}{Theorem}
\newenvironment{customthm}[1]
 {\renewcommand\theinnercustomthm{#1}\innercustomthm}
 {\endinnercustomthm}

\title{Stability of hyperbolic Oseledets splittings for quasi-compact operator cocycles}
\author{H. Crimmins\footnote{School of Mathematics and Statistics, UNSW Sydney, NSW, 2052 Australia;  email: harry.crimmins@unsw.edu.au}}
\date{\today}

\begin{document}

\maketitle

\begin{abstract}
  We consider the problem of stability and approximability of Oseledets splittings and Lyapunov exponents for Perron-Frobenius operator cocycles associated to random dynamical systems.
  By developing a random version of the perturbation theory of Gou{\"e}zel, Keller, and Liverani, we obtain a general framework for solving such stability problems, which is particularly well adapted to applications to random dynamical systems.
  We apply our theory to random dynamical systems consisting of $\mathcal{C}^k$ expanding maps on $S^1$ ($k \ge 2$) and provide conditions for the stability of Lyapunov exponents and Oseledets splitting of the associated Perron-Frobenius operator cocycle to (i) uniformly small fiber-wise $\mathcal{C}^{k-1}$-perturbations to the random dynamics, and (ii) numerical approximation via a Fej{\'e}r kernel method.
  A notable addition to our approach is the use of Saks spaces, which provide a unifying framework for many key concepts in the so-called `functional analytic' approach to studying dynamical systems, such as Lasota-Yorke inequalities and Gou{\"e}zel-Keller-Liverani perturbation theory.
\end{abstract}

\pagestyle{myheadings}

\section{Introduction}

Suppose that $M$ is a compact $\mathcal{C}^\infty$ Riemannian manifold, and that $T : M \to M$ is a sufficiently smooth, uniformly expanding (or uniformly hyperbolic) map.
One approach to characterising the statistical properties of $T$ is by finding a Banach space $(X, \norm{\cdot})$ such that\footnote{We use $\into$ to denote a continuous inclusion between topological vector spaces.} $\mathcal{C}^{\infty}(M, \C) \into X \into (\mathcal{C}^{\infty}(M, \C))'$ and so that the Perron-Frobenius operator associated to $T$, denoted $L_T$, induces a bounded, quasi-compact operator on $X$ (we recommend the reader consult \cite{galatolo2015statistical,liverani2004invariant, viviane2000positive,baladi2018dynamical} for an overview of this approach).
Quasi-compactness implies that the spectrum of $L_T$ outside of the essential spectral radius, which is strictly less than 1, consists solely of a number of isolated \emph{exceptional eigenvalues} of finite multiplicity.
The spectral data associated to these exceptional eigenvalues encodes information about the statistical properties of $T$, such as the existence of finitely many SRB measures for $T$, or the rate of exponential decay of correlations if the system is mixing \cite{blank2002ruelle,viviane2000positive}.
In addition, as the essential spectral radius of the Perron-Frobenius operator is typically determined by the exponential rate at which nearby trajectories separate, any eigenvector of $L_T$ associated to an exceptional eigenvalue with modulus less than 1 describes a feature of the dynamics that decays too slowly to be attributed to local hyperbolicity.
This is the Dellnitz-Froyland ansatz \cite{dellnitz2000isolated}, which asserts that these eigenvectors correspond to global dynamical structures that are responsible for slower than expected mixing, such as metastable states or almost-invariant sets.

It is a natural questions as to whether these exceptional eigenvalues and their associated eigendata are stable to perturbations of the dynamics, and whether they may be approximated numerically.
Unfortunately, only in very special settings\footnote{Such as when the dynamics are all analytic, in which case $L_T$ is frequently compact.} is the map $T \mapsto L_T$ continuous in the operator norm, and so classical perturbation theory for linear operators \cite{kato1966perturbation} is inapplicable.
An alternative theory, particularly well adapted to studying the stability of the spectrum of Perron-Frobenius operators, has been developed by Gou{\"e}zel, Keller, and Liverani \cite{gouezel2006banach,keller1982stochastic,keller1999stability}.
The basic observation is that $T \mapsto L_T$ is frequently continuous in a weaker sense: there is a norm $\wnorm{\cdot}$ on $X$ with $\wnorm{\cdot} \le \norm{\cdot}$ so that the `triple norm'
\begin{equation*}
  \tnorm{L_T - L_S} = \sup_{\norm{f} = 1} \wnorm{(L_T - L_S)f}
\end{equation*}
is small when $T$ and $S$ are close. Then, provided that all operators under consideration satisfy a uniform Lasota-Yorke inequality, in addition to some other requirements, one recovers stability of the exceptional eigenvalues of $L_T$ as well as of the associated eigenprojections and eigenvectors in $\tnorm{\cdot}$ and $\wnorm{\cdot}$, respectively.
This approach has found wide applicability to dynamical systems:
diverse systems and perturbations fit into this paradigm, notably including numerical methods.

The `functional analytic' approach outlined in the previous paragraphs has been partially generalised to random dynamical systems by applying multiplicative ergodic theory to random Perron-Frobenius operator cocycles \cite{froyland2013semi, GTQuas1, gonzalez2018multiplicative}.
Let $(\Omega, \mathcal{F}, \mathbb{P})$ be a Lebesgue space and $\sigma : \Omega \to \Omega$ be a $\mathbb{P}$-ergodic, invertible transformation. A measurable map $\mathcal{T} : \Omega \to \mathcal{C}^k(M)$ induces a random dynamical system over $\sigma$, whose trajectories are of the form
\begin{equation*}
  x, \mathcal{T}_\omega(x), \mathcal{T}_{\omega}^2(x), \dots , \mathcal{T}_{\omega}^n(x), \dots
\end{equation*}
where $\mathcal{T}_\omega := \mathcal{T}(\omega)$ and $\mathcal{T}_\omega^n := \mathcal{T}_{\sigma^{n}(\omega)} \circ \cdots \circ \mathcal{T}_\omega$.
If $(X, \norm{\cdot})$ is a Banach space on which the Perron-Frobenius operators $L_\omega$ associated to each $\mathcal{T}_\omega$ are bounded, then the map\footnote{When $(X,\norm{\cdot})$ is a Banach space we denote the set of $\norm{\cdot}$-bounded linear operators on $X$ by $\LL(X)$.}
\begin{equation*}
  (n, \omega) \in \mathbb{N} \times \Omega \mapsto L_{\sigma^{n}(\omega)} \circ \cdots \circ L_\omega \in \LL(X)
\end{equation*}
is the associated Perron-Frobenius cocycle, and we call $(\Omega, \mathcal{F}, \mathbb{P}, \sigma, X, \omega \mapsto L_\omega)$ a \emph{linear random dynamical system}.
If the cocycle is quasi-compact \cite{thieullen1987fibres} on an appropriate Banach space then by applying a semi-invertible multiplicative ergodic theorem one obtains an Oseledets splitting and Lyapunov exponents for the cocycle \cite{froyland2010coherentpf, froyland2013semi,GTQuas1}.
In this setting the Lyapunov exponents play the role of the exceptional eigenvalues, and their corresponding Oseledets spaces generalise the eigenspaces associated to exceptional eigenvalues.
These objects yield information about the statistical properties of a given random dynamical system.
For instance, vectors in the intersection of the top Oseledets space and positive cone are exactly the systems' random equivariant measures \cite{froyland2013semi}, and the Oseledets spaces corresponding to negative Lyapunov exponents describe global random dynamical structures, known as \emph{coherent structures}, that decay slower than the long term rate of local trajectory separation \cite{froyland2010coherentpf}.
However, there is yet to be a stability theory for Oseledets splittings and Lyapunov exponents that is comparable to Gou{\"e}zel--Keller--Liverani (GKL) perturbation theory.
Developing such a theory is the goal of this paper.

The question of stability and approximability of Lyapunov exponents and Oseledets splittings for Perron-Frobenius cocycles has been considered on a few prior occasions.
The stability of the random equivariant measure (i.e. the top Oseledets space) for i.i.d. random systems consisting of uniformly expanding maps nearby a fixed map as the random dynamics is perturbed fiber-wise was treated in \cite{baladi1996random}, and then extended in \cite{baladi1997correlation} to the stability of random equilibrium states.
The first paper comparable to ours in scope is \cite{bogenschutz2000stochastic}, wherein the stability of Lyapunov exponents above a critical index, and their corresponding Oseledets spaces, is proven for `asymptotically small random perturbations' of cocycles whose Oseledets splitting satisfies certain hyperbolicity conditions.
In Section \ref{sec:stability_lyapunov} we contrast the results of \cite{bogenschutz2000stochastic} with our main theorems.
More recently, in \cite{froyland2014stability} the stability of the random equivariant measures (i.e. the top Oseledets space) was proven for general Markov perturbations following a triple-norm approach as in GKL perturbation theory.
Lastly, in \cite{gonzalez2018stability} the stability, and lack thereof, of Lyapunov exponents was studied for the Perron-Frobenius operator cocycle associated to a random dynamical system consisting of expanding Blashcke products when subjected to a variety of perturbations. This last setting is quite special, since the dynamics are all analytic and each Perron-Frobenius operator is compact; in such a setting one usually has stability in the operator norm, and therefore does not need to resort to GKL perturbation theory.

Our first two main theorems (Theorems \ref{thm:stability_cocycle} and \ref{thm:stability_lyapunov}) are abstract stability results, which are modelled on GKL perturbation theory, for the Oseledets splitting and Lyapunov exponents of certain random linear cocycles. The following statement summarises these two results; for precise formulations see Sections \ref{sec:saks_space_stability} and \ref{sec:stability_lyapunov}.

\begin{customthm}{A}\label{thmA}
  Suppose that $(X, \norm{\cdot})$ is a separable Banach space with a compatible weak norm $\wnorm{\cdot}$ (Definition \ref{def:saks_space}), that $(\Omega, \mathcal{F}, \mathbb{P})$ is a Lebesgue space, and that $\sigma : \Omega \to \Omega$ is a $\mathbb{P}$-ergodic invertible transformation.
  If $\{(\Omega, \mathcal{F}, \mathbb{P}, \sigma, X, \mathcal{Q}_\epsilon) \}_{\epsilon \ge 0}$ is a sequence of strongly measurable linear random dynamical systems (Definition \ref{def:random_linear_system}) such that
  \begin{enumerate}
    \item $(\Omega, \mathcal{F}, \mathbb{P}, \sigma, X, \mathcal{Q}_0)$ has a hyperbolic Oseledets splitting (Definition \ref{def:hyperbolic_oseledets}) and is $\wnorm{\cdot}$-bounded;
    \item the set $\{(\Omega, \mathcal{F}, \mathbb{P}, \sigma, X, \mathcal{Q}_\epsilon) \}_{\epsilon \ge 0}$ satisfies a uniform Lasota-Yorke inequality (Definition \ref{def:ly}); and
    \item $\lim_{\epsilon \to 0} \esssup_\omega  \tnorm{\mathcal{Q}_\epsilon(\omega) -\mathcal{Q}_0(\omega)} = 0$;
  \end{enumerate}
  then $\mathcal{Q}_\epsilon$ has an Oseledets splitting for sufficiently small $\epsilon$, and the Lyapunov exponents and Oseledets projections of $\mathcal{Q}_\epsilon$ converge to those of $\mathcal{Q}_0$ as $\epsilon \to 0$.
\end{customthm}

The proofs of GKL perturbation theory do not generalise to prove Theorems \ref{thm:stability_cocycle} and \ref{thm:stability_lyapunov} since the Lyapunov exponents and Oseledets splittings are not related to the spectral data of any operator in consideration.
A new method is therefore required, and the one we pursue is reminiscent of the proof that the class of Anosov maps is open \cite[Corollary 5.5.2]{brin2002introduction}, and therefore bypasses the use of spectral theory.
Specifically, we construct invariant cone fields of subspaces for the unperturbed cocycle and show that these cone fields are also invariant under cocycles that satisfy an appropriate Lasota-Yorke inequality and are uniformly $\tnorm{\cdot}$-close to the unperturbed cocycle.
For this cone argument to work we require an extra hypothesis that is automatically satisfied in the non-random case: the Oseledets splitting must be \emph{hyperbolic}, which essentially means that the exponential separation of Oseledets spaces under the action of the unperturbed cocycle must be realised uniformly in time across all realisations of the random cocycle.
Cone-based arguments of a different flavour have been applied to prove stability of Lyapunov exponents and Oseledets splittings before: in the previously mentioned \cite{baladi1996random}, in \cite{novel2017p} to prove continuity of Lyapunov exponents for operator norm perturbations, and in \cite{sedro2018etude} to prove differentiability of the top Lyapunov exponent and Oseledets spaces for Perron-Frobenius cocycles.
In these works it is required that both the perturbed and unperturbed cocycles preserve a constant (non-random) cone field, which are defined as subsets of the relevant Banach space, and then a Birkhoff cone contraction argument is used.
This should be contrasted with our approach: we use hyperbolicity of the Oseledets splitting to construct various `random' cone fields of subspaces, which are defined as subsets of the Grassmannian. Moreover, by working with the graph-representation of the Grassmannian we mainly deal with projections rather than subspaces, reducing much of the proof to various operator norm estimates.

As a simple application of Theorems \ref{thm:stability_cocycle} and \ref{thm:stability_lyapunov}, we obtain the following result on the stability of Oseledets splitting and Lyapunov exponents for certain random dynamical systems consisting of $\mathcal{C}^k$ expanding maps on the circle (see Theorems \ref{thm:random_deterministic_perturbation} and \ref{thm:random_fejer_approx} for precise, formal statements).
This result should be compared with those in \cite{baladi1996random,bogenschutz2000stochastic} (in the one-dimensional case): when fiber-wise perturbing a $C^k$ map we only require that the maps are fiberwise $\mathcal{C}^{k-1}$-close with bounded $k$th derivative, rather being than $\mathcal{C}^k$-close, although we obtain slightly weaker conclusions as a consequence.
To our knowledge this is the first result on the numerical approximation of Oseledets splittings and Lyapunov exponents for Perron-Frobenius cocycles, other than the approximation of the top Oseledets space in \cite{froyland2014stability}.

\begin{customthm}{B}\label{thmB}
  For $k \ge 2$, $\alpha \in (0,1)$ and $K > 0$ we set
  \begin{equation*}
    \LY_k(\alpha, K) = \left\{ T \in \mathcal{C}^k(S^1, S^1) : \inf \abs{T'} \ge \alpha^{-1} \text{ and } d_{\mathcal{C}^k}(T,0) \le K \right\}.
  \end{equation*}
  Suppose that $(\Omega, \mathcal{F}, \mathbb{P})$ is a Lebesgue space, $\sigma : \Omega \to \Omega$ is $\mathbb{P}$-ergodic and invertible, and $\mathcal{T} : \Omega \to \LY_k(\alpha, K)$ is measurable.
  If the Perron-Frobenius cocycle induced by the random dynamical system associated to $\mathcal{T}$ and $\sigma$ has a hyperbolic Oseledets splitting on the Sobolev space $W^{k-1,1}(S^1)$ then the Lyapunov exponents and projections onto Oseledets subspaces are stable (with $\wnorm{\cdot}$ taken to be the $W^{k-2,1}(S^1)$ norm) to
  \begin{enumerate}
    \item Deterministic fiber-wise perturbations: for measurable $\mathcal{S} : \Omega \to \LY_k(\alpha, K)$ we have convergence of Lyapunov exponents and Oseledets projections as $\esssup d_{\mathcal{C}^{k-1}}(\mathcal{T}_\omega, \mathcal{S}_\omega) \to 0$.
    \item Fej{\'e}r kernel methods: the Lyapunov exponents and Oseledets projections of the finite dimensional matrix cocycle associated to the $n$th Fej{\'e}r kernel approximation of the Perron-Frobenius operator cocycle converge to those of the unperturbed cocycle as $n \to \infty$.
  \end{enumerate}
\end{customthm}

Having sketched our main results let us comment on a key conceptual difference between our approach and those prior, which is the framing of our results in terms of \emph{Saks spaces}. Concretely, a Saks space is a Banach space $(X, \norm{\cdot})$ equipped with a second, weaker topology satisfying certain compatibility conditions. In applications to dynamics, this weaker topology is the weak norm $\wnorm{\cdot}$ mentioned in the previous paragraphs. The natural `weak norm' on the set of linear, $\norm{\cdot}$-bounded operators on $X$ happens to be the triple norm $\tnorm{\cdot}$, and so GKL perturbation theory may be considered as a theory of spectral stability for Saks space perturbations. We believe that Saks spaces provide a natural framework for many key concepts in the so-called `functional analytic' approach to studying dynamical systems, and by studying these spaces on their own we may better understand the potential for, and limitations of, this approach.
In a concrete demonstration of the relevancy Saks spaces let us remark that (i) using Saks space theory one can precisely characterise the set of norms $\wnorm{\cdot}$ such that the closed unit $\norm{\cdot}$-ball is $\wnorm{\cdot}$-compact (see Theorems \ref{thm:compact_saks_unique} and \ref{thm:saks_space_norm_formula}), which has applications to the construction of good anisotropic Banach spaces for hyperbolic dynamical systems, and
(ii) by embracing the Saks space setting we weaken the hypotheses of GKL stability results in the non-random case (see Remark \ref{remark:comparison_to_kl}).

The remainder of the paper is structured as follows. In Section \ref{sec:prelims} we recall material about the Grassmannian on a Banach space, and provide a primer on the theory of Saks spaces.
We prove our first main result of this paper in Section \ref{sec:saks_space_stability}, which concerns the Saks space stability of hyperbolic splittings for bounded linear endomorphisms on vector bundles.
In Section \ref{sec:stability_lyapunov} we leverage the results of Section \ref{sec:saks_space_stability} to prove the Saks space stability of Oseledets splittings and Lyapunov exponents for random linear dynamical systems on separable Banach spaces when the unperturbed system possesses a hyperbolic Oseledets splitting.
Finally, in Section \ref{sec:app_random_dynamics} we apply the abstract results of the previous sections to deduce the stability of the Lyapunov exponents and Oseledets splittings for Perron-Frobenius operator cocycles associated to random dynamical systems of uniformly expanding $\mathcal{C}^k$ maps on the circle with respect to (i) fiber-wise perturbations to the dynamics and (ii) numerical approximation of the Perron-Frobenius operator using a Fej{\'e}r kernel method.
The two appendices contain some of the more technical proofs of the paper, which would otherwise obscure the main ideas of their respective sections.

\section{Preliminaries}\label{sec:prelims}

In this section we recall some preliminary theory that we will later use to state and prove our main results. Section \ref{sec:grassmannian} concerns the Grassmannian on a Banach space, while Section \ref{sec:ss_primer} serves as a primer on the theory of Saks spaces. To simplify the exposition we have deferred all proofs to Appendix \ref{app:preliminaries_proofs}.

First we fix some notation. Firstly, we shall always consider Banach spaces over $\C$. When $X_1$ and $X_2$ are Banach spaces we let $\LL(X_1, X_2)$ denote the set of bounded operators from $X_1$ to $X_2$.
When $\norm{\cdot}$ is a norm on a vector space $X$ we denote the associated closed-unit ball by either $B_{\norm{\cdot}}$ or, when the relevant norm is clear, by $B_{X}$.
If $X_1$ and $X_2$ are topological vector spaces such that $X_1$ is continuously included into $X_2$ then we will write $X_1 X_2$.
If $A \in \LL(X)$ then we denote the spectrum of $A$ by $\sigma(A)$, the spectral radius of $A$ by $\rho(A)$, and the essential spectral radius of $A$ by $\rho_{\ess}(A)$.
When $(Y, d)$ is a metric space we denote the Borel $\sigma$-algebra on $Y$ by $\mathcal{B}_Y$.

\subsection{Graphs and the Grassmannian}\label{sec:grassmannian}

This section summarises some old, but not particularly well-known, material for the readers convenience, and has been collated from \cite[Chapter IV, \S 2 and \S 4]{kato1966perturbation}, \cite[Section 2.1]{blumenthal2016volume}, \cite[Section 2]{froyland2013semi}, and \cite[Appendix A.2]{quas2019explicit}.
If $(X, \norm{\cdot})$ is a Banach space then the set of closed subspaces of $X$ is called the Grassmannian of $X$, and is denoted by $\mathcal{G}(X)$. It is a complete metric space when equipped with the metric
\begin{equation*}
  d_H(E, F) = \max\left\{\sup_{\substack{e \in E \\ \norm{e} = 1}} \inf_{\substack{f \in F \\ \norm{f} = 1}} \norm{e - f}, \sup_{\substack{f \in F \\ \norm{f} = 1}} \inf_{\substack{e \in E \\ \norm{e} = 1}} \norm{e - f} \right\}.
\end{equation*}
The metric $d_H$ is rather hard to bound directly. Instead, it is convenient to work with the \emph{gap} between two subspaces:
\begin{equation*}
  \gap(E, F) = \sup_{\substack{ e \in E \\ \norm{e} = 1}} \inf_{v \in F} \norm{e - f}.
\end{equation*}
We can work with the gap in place of $d_H$ due to the following inequality
\begin{equation}\label{eq:gap_equiv}
  \max\{ \gap(E, F), \gap(F, E)\} \le d_H(E,F) \le 2\max\{ \gap(E, F), \gap(F, E)\}.
\end{equation}

We say that $E, F \in \mathcal{G}(X)$ are topologically complementary subspaces if $E + F = X$ and $E \cap F = \emptyset$, in which case we will write $E \oplus F = X$.
Denote by $\Pi_{E || F}$ the projection  onto $E$ and parallel to $F$ i.e the image of $\Pi_{E || F}$ is $E$ and $\ker(\Pi_{E || F}) = F$. By the closed graph theorem $\Pi_{E || F} \in \LL(X)$.
For every $d \in \Z^+$ we denote by $\mathcal{G}_d(X)$ and $\mathcal{G}^d(X)$ the sets of $d$-dimensional and $d$-codimensional subspaces, respectively. The sets $\mathcal{G}_d(X)$ and $\mathcal{G}^d(X)$ are closed for every $d \in \Z^+$.
For each $F \in \mathcal{G}(X)$ the set
\begin{equation*}
  \mathcal{N}(F) = \{ E \in \mathcal{G}(X) : E \oplus F = X \}
\end{equation*}
is open in $\mathcal{G}(X)$, and has a convenient representation in terms of certain charts. Specifically, for any $E \in \mathcal{N}(F)$ we define $\Phi_{E \oplus F} : \mathcal{N}(F) \to \mathcal{L}(E,F)$
by
\begin{equation*}
  \Phi_{E \oplus F}(E') = \left(\restr{\Pi_{E || F}}{E'}\right)^{-1} - \Id.
\end{equation*}
We call $\Phi_{E \oplus F}$ the graph representation of $\mathcal{N}(F)$ induced by $E \oplus F$. That the graph representation of $\mathcal{N}(F)$ induced by $E \oplus F$ is well-defined follows from the following lemma.
\begin{lemma}\label{lemma:inv_proj}
  If $F \in \mathcal{G}(X)$ and $E_1,E_2 \in \mathcal{N}(F)$ then $\Pi_{E_1 || F} : E_2 \to E_1$ is invertible.
\end{lemma}

We summarise the properties of the graph representation in the next proposition.
\begin{proposition}\label{prop:graph_chart}
  If $E \oplus F = X$ then the associated graph representation $\Phi_{E \oplus F}$ is a homeomorphism. Moreover, for every $E' \in \mathcal{N}(F)$ we have
  \begin{equation}\label{eq:graph_chart_0}
    \Pi_{E' || F} = (\Id + \Phi_{E \oplus F}(E'))\Pi_{E || F},
  \end{equation}
  and for every $L \in \LL(E,F)$ we have
  \begin{equation}\label{eq:graph_chart_00}
    \Phi_{E \oplus F}^{-1}(L) = (\Id + L)(E).
  \end{equation}
\end{proposition}
The identities \eqref{eq:graph_chart_0} and \eqref{eq:graph_chart_00} follows from Lemmas \ref{lemma:proj_graph} and \ref{lemma:graph_rep_inv}, respectively.
That $\Phi_{E \oplus F}$ is a homeomorphism is a consequence of the following two Lemmas.

\begin{lemma}\label{lemma:graph_rep_continuity}
  If $E \oplus F = X$ and $L_1, L_2 \in \mathcal{L}(E,F)$ then
  \begin{equation*}
    d_{H}(\Phi_{E \oplus F}^{-1}(L_1), \Phi_{E \oplus F}^{-1}(L_2)) \le 2\norm{\Pi_{E || F}} \norm{L_1 - L_2}.
  \end{equation*}
\end{lemma}

\begin{lemma}\label{lemma:graph_rep_inv_continuity}
  If $E \oplus F = X$ and $E_1, E_2 \in \mathcal{N}(F)$ then
  \begin{equation*}
    \norm{\Phi_{E \oplus F}(E_1) - \Phi_{E \oplus F}(E_2)} \le \left(\max \left\{\norm{\Pi_{F || E_1}}\norm{\Pi_{E_2 || F}}, \norm{\Pi_{F || E_2}}\norm{\Pi_{E_1 || F}}\right\}\right)^{-1}d_H(E_1, E_2).
  \end{equation*}
\end{lemma}

Suppose that $X_1, X_2$ are Banach spaces with $E_i \oplus F_i = X_i$ for $i=1,2$, and let $S \in \LL(X_1, X_2)$. Then $S$ induces a natural action on both $\mathcal{G}(X_1)$ and $\mathcal{G}(X_2)$.
Namely, $V_1 \in \mathcal{G}(X_1)$ is mapped to $S(V_1) \in \mathcal{G}(X_2)$, and $V_2 \in \mathcal{G}(X_2)$ is mapped to $S^{-1}(V_2) \in \mathcal{G}(X_1)$.
We will now describe how these actions may induce an action between the graph representations of $\mathcal{N}(F_1)$ and $\mathcal{N}(E_2)$.
If $U \in \LL(E_1, F_1)$ is such that
$\restr{\Pi_{E_{2} || F_{2}} S(\Id + U)}{E_1} : E_1 \to E_{2}$ is invertible then we may define the forward graph-transform of $U$ by $S$ to be
\begin{equation*}
  S^*U = \Pi_{F_{2} || E_{2}} S(\Id + U) \left(\restr{\Pi_{E_{2} || F_{2}} S(\Id + U)}{E_1}\right)^{-1},
\end{equation*}
in which case $S^*U \in \mathcal{L}(E_{2}, F_{2})$.
On the other hand, if $U \in \LL(F_2, E_2)$ is such that $\Pi_{E_2 || F_2}(\Id - U\Pi_{F_2 || E_2})S : E_1 \to E_2$ is invertible then we define the backward graph-transform of $U$ by $S$ to be
\begin{equation*}
  S_* U = \left(\restr{\Pi_{E_2 || F_2}(\Id - U\Pi_{F_2 || E_2})S}{E_1}\right)^{-1} (U\Pi_{F_2 || E_2} - \Pi_{E_2 || F_2})S.
\end{equation*}
Using Proposition \ref{prop:graph_chart}, a quick calculation confirms that $S^*$ and $S_*$ agree with the usual action of an operator on a subspace.
\begin{proposition}\label{prop:graph_tranform}
  Fix $S \in \LL(X_1, X_2)$ and suppose that $E_i \oplus F_i = X_i$ for $i=1,2$.
  \begin{enumerate}
    \item For any $E' \in \mathcal{N}(F_1)$ such that $\Pi_{E_{2} || F_{2}} S \Pi_{E' || F_1} : E_1 \to E_2$ is invertible we have
    \begin{equation*}
      S(E') = \Phi_{E_2 \oplus F_2}^{-1}\left(S^*(\Phi_{E_1 \oplus F_1}(E'))\right).
    \end{equation*}
    \item For any $F' \in \mathcal{N}(E_2)$ such that $\Pi_{E_2 || F_2}\Pi_{E_2 || F'}S : E_1 \to E_2$ is invertible we have
    \begin{equation*}
      S^{-1}(F') = \Phi_{F_1 \oplus E_1}^{-1}(S_* (\Phi_{F_2 \oplus E_2}(F'))).
    \end{equation*}
  \end{enumerate}
\end{proposition}

\subsection{A Saks space primer}\label{sec:ss_primer}

In this section we reproduce from \cite[Chapter 1]{cooper2011saks} the definition and basic properties of Saks spaces, in addition to proving some new results.
We refer the reader to \cite[Chapter 1]{cooper2011saks} for a comprehensive overview of the theory and history of Saks spaces, as we only include the theory needed for our applications.
Throughout this section $X$ will denote a vector space.

\begin{lemma}[{\cite[Lemma 3.1]{cooper2011saks}}]\label{lemma:equiv}
  Let $X$ be a vector space, $\tau$ be a locally convex topology on $X$, and $\norm{\cdot}$ be a norm on $X$.
  Then the following are equivalent:
  \begin{enumerate}
    \item $B_{\norm{\cdot}}$ is $\tau$-closed;
    \item $\norm{\cdot}$ is lower semicontinuous for $\tau$;
    \item $\norm{\cdot} = \sup\{ \varphi : \varphi \text{ is a  $\tau$-continuous seminorm with } \varphi \le \norm{\cdot}\}$.
  \end{enumerate}
\end{lemma}

\begin{definition}[Saks space]\label{def:saks_space}
  Let $(X, \norm{\cdot})$ be a normed space and $\tau$ be a Hausdorff locally convex topology on $X$ such that $B_{\norm{\cdot}}$ is $\tau$-bounded and any of the conditions from Lemma \ref{lemma:equiv} is satisfied.
  Denote by $\gamma[\norm{\cdot}, \tau]$ the finest linear topology on $X$ that coincides with $\tau$ on $B_{\norm{\cdot}}$.
  The tuple $(X,\norm{\cdot},\tau)$ equipped with the topology $\gamma[\norm{\cdot}, \tau]$ is called a Saks space; when clear we simply denote this space by $X$ and the topology by $\gamma$.
  We say that $X$ is complete (resp. compact, pre-compact) if $B_{\norm{\cdot}}$ is $\tau$-complete (resp. $\tau$-compact, $\tau$-pre-compact).
\end{definition}

\begin{remark}
  If $X$ is complete as a Saks space then $(X,\norm{\cdot})$ is a Banach space. The converse is false.
\end{remark}

\begin{remark}\label{remark:saks_space_wo_closure}
  Sometimes one produces a tuple $(X,\norm{\cdot},\tau)$ satisfying the definition of a Saks space except for the conditions in Lemma \ref{lemma:equiv}.
  In such a case we could instead consider the Saks spaces $(X, \norm{\cdot}', \tau)$, where $\norm{\cdot}'$ denotes the Minkowski functional\footnote{If $K$ is a balanced, convex body in a vector space $X$ then the Minkowski functional of $K$ is the map $\rho_{K} : V \to [0,\infty)$ defined by $\rho_K(x) = \inf\{ \lambda \in [0, \infty) : \lambda x \in K\}$. The Minkowski functional of a balanced, convex body is always a seminorm, and if $K$ has non-empty interior then it is also a norm.} associated to the $\tau$-closure of $\norm{\cdot}$.
  We do not lose any continuous linear maps by performing this procedure \cite[Lemma 3.3]{cooper2011saks}.
\end{remark}

\begin{remark}
  As outlined in \cite[Chapter 1, Section 3.6]{cooper2011saks}, there is a canonical completion of a non-complete Saks space $(X, \norm{\cdot}, \tau)$. Let $\cl{X}_\tau$ denote the $\tau$-completion of $X$, and define $\norm{\cdot}_\tau$ to be the Minkowski functional of the $\tau$-completion of $B_{\norm{\cdot}}$ in $\cl{X}_\tau$.
  If $\cl{X}$ denotes the linear span of $B_{\norm{\cdot}_\tau}$ then $(\cl{X}, \norm{\cdot}_\tau, \tau)$ is a complete Saks space: the Saks space completion of $(X, \norm{\cdot}, \tau)$.
  We refer the reader to the discussion at the end of \cite[Chapter 1, Section 3.6]{cooper2011saks} for further properties of the Saks space completion, and to \cite[Proposition 3.8]{cooper2011saks} for an interesting characterisation of complete Saks spaces.
\end{remark}

\begin{example}
  Let $(X, \norm{\cdot})$ be a Banach space, and let $\tau$ denote the weak-star topology on $X^*$. Then $(X^*, \norm{\cdot}^*, \tau)$ is a compact Saks space by the Banach-Alaoglu Theorem.
  In Proposition \ref{prop:saks_space_finite} we shall see that every compact Saks space has this form.
\end{example}

\begin{example}\label{example:incl}
  Suppose that $(X_i, \norm{\cdot}_i)$, $i = 1,2$ are Banach spaces with $X_2 \into X_1$.
  Hence $B_{\norm{\cdot}_2}$ is $\norm{\cdot}_1$-bounded, however it may not be the case that $B_{\norm{\cdot}_2}$ is $\norm{\cdot}_1$-closed.
  If we let $\norm{\cdot}_2'$ denote the Minkowski functional of the $\norm{\cdot}_1$-completion of $B_{\norm{\cdot}_2}$, then $(X, \norm{\cdot}_2', \norm{\cdot}_1)$ is a Saks space per Remark \ref{remark:saks_space_wo_closure}.
  A formula for $\norm{v}_2'$ was recognised in \cite[Remark 2.3]{conze2007limit}:
  \begin{equation}\label{eq:incl_1}
    \norm{v}_2' = \lim_{\delta \to 0} \inf \{ \norm{w}_2 : \norm{w-v}_1 \le \delta \}.
  \end{equation}
  To see \eqref{eq:incl_1}, let us fix $v$ and suppose that $\norm{v}_2' = 1$. Let
  \begin{equation*}
    P_{v} = \left\{ (v_n )_{n \in \Z^+} \subseteq B_{\norm{\cdot}_2} : \lim_{n \to \infty} \norm{v_n -v}_1 = 0 \right\}.
  \end{equation*}
  By definition, $P_v$ must be non-empty and for every $(v_n )_{n \in \Z^+} \in P_v$ we must have $\lim_{n \to \infty} \norm{v_n}_2 = 1$, else $\norm{v}_2' < 1$ by the definition of the Minkowski functional.
  Since $P_v$ is non-empty and closed under taking subsequences, there exists $(v_n )_{n \in \Z^+} \in P_v$ such that $\norm{v_n -v}_1 \le 1/n$ and $\abs{\norm{v_n}_2  - 1} \le 1/n$ for every $n \in \Z^+$.
  Thus
  \begin{equation*}
    \inf \{ \norm{w}_2 : \norm{w-v}_1 \le 1/n \} \le \norm{v_n}_2 \le 1 + 1/n
  \end{equation*}
  and so $\lim_{\delta \to 0} \inf \{ \norm{w}_2 : \norm{w-v}_1 \le \delta \} \le 1$.
  If the limit is not 1, then there exists a sequence $(u_n)_{n \in \Z^+}$ such that $\lim_{n \to \infty} \norm{u_n -v}_1 = 0$ and $\lim_{n \to \infty} \norm{u_n}_2 < 1$, which implies that $(u_n)_{n \in \Z^+} \in P_v$.
  However, as discussed earlier, if $(u_n)_{n \in \Z^+} \in P_v$ then we must have $\lim_{n \to \infty} \norm{u_n}_2 = 1$, which is a contradiction. Thus $\lim_{\delta \to 0} \inf \{ \norm{w}_2 : \norm{w-v}_1 \le \delta \} = 1$.
\end{example}

\begin{example}
  We may generalise Example \ref{example:incl} as follows.
  Suppose that $(X_i, \norm{\cdot}_i)$, $i \in \{1,\dots,n\}$, are Banach spaces with $X_n \into  X_{n-1} \into \dots \into X_1$.
  Let $\norm{\cdot}_{i}'$ denote the Minkowski functional induced by the $\norm{\cdot}_{i-1}$-completion of $B_{\norm{\cdot}_i}$.
  As in Example \ref{example:incl}, each of the tuples $(X_i, \norm{\cdot}_{i}', \norm{\cdot}_{i-1})$ is a Saks space.
  Moreover, we have the following chain of continuous inclusions:
  \begin{equation*}
    (X_n, \norm{\cdot}_n) \into (X_n, \norm{\cdot}_{n}', \norm{\cdot}_{n-1}) \into (X_{n-1}, \norm{\cdot}_{n-1}) \into \dots \into (X_2, \norm{\cdot}_{2}', \norm{\cdot}_{1}) \into  (X_1, \norm{\cdot}_{1}).
  \end{equation*}
\end{example}

For a Saks space $(X, \norm{\cdot}, \tau)$ it is possible to give an explicit description of the open sets in $\gamma[\norm{\cdot}, \tau]$. If $(U_n)_{n \in \Z^+}$ denotes a family of absolutely convex $\tau$-open neighbourhoods of 0, then all the sets of the form
\begin{equation}\label{eq:ss_neighbourhood_basis}
  \bigcup_{n=1}^\infty (U_1 \cap B_{\norm{\cdot}} + \dots + U_n \cap 2^{n-1} B_{\norm{\cdot}})
\end{equation}
form a neighbourhood basis about $0$ for a locally convex topology on $X$. By \cite[Proposition 1.5]{cooper2011saks}, this locally convex topology is the $\gamma[\norm{\cdot}, \tau]$ topology.

\begin{remark}
  Any Banach space $(X, \norm{\cdot})$ induces a Saks space with the structure $(X, \norm{\cdot}, \norm{\cdot})$. From the definition of the neighbourhood basis for $\gamma[\norm{\cdot}, \norm{\cdot}]$, it is clear that the $\norm{\cdot}$-topology is equivalent to $\gamma[\norm{\cdot}, \norm{\cdot}]$.
\end{remark}

Let $(X,\norm{\cdot},\tau)$ be a Saks space. Despite $X$ being endowed with the structure of a locally convex vector space, for conceptual purposes it is better to forget this characterisation and adopt the following philosophy: provided that one works on $\norm{\cdot}$-bounded sets, the topological properties of $\gamma$ are the same as $\tau$.
The following three propositions demonstrate this principle.

\begin{proposition}[{\cite[Proposition 1.10]{cooper2011saks}}]\label{prop:saks_space_conv}
  A sequence $(x_n)_{n \in \Z^+} \subseteq X$ is $\gamma$-convergent to $x$ if and only if $(x_n)_{n \in \Z^+}$ is $\norm{\cdot}$-bounded and $\tau$-convergent to $x$.
\end{proposition}

\begin{proposition}[{\cite[Proposition 1.11, 1.12]{cooper2011saks}}]\label{prop:ss_bounded}
  If $V \subseteq X$ then:
  \begin{enumerate}
    \item $V$ is $\gamma$-bounded if and only if it is $\norm{\cdot}$-bounded.
    \item $V$ is $\gamma$-compact (resp. $\gamma$-pre-compact) if and only if it is $\norm{\cdot}$-bounded and $\tau$-compact (resp. $\tau$-pre-compact).
  \end{enumerate}
\end{proposition}

\begin{proposition}[{\cite[Corollary 1.6]{cooper2011saks}}]\label{prop:saks_space_weak_top}
  If $(X, \norm{\cdot}, \tau)$ and $(X, \norm{\cdot}, \tau')$ are Saks spaces then $\gamma[\norm{\cdot}, \tau]$ and  $\gamma[\norm{\cdot}, \tau']$ are equivalent if and only if $\tau$ and $\tau'$ are equivalent on $B_{\norm{\cdot}}$.
\end{proposition}

Having described the basic theory of Saks spaces, we mention a few more concrete examples.

\begin{example}\label{example:bv}
  Let $P$ denote the set of strictly increasing finite sequences in $S^1$.
  Fix $p \in [1, \infty)$. For $f \in L^p(S^1)$ the $p$-variation of $f$ to be
  \begin{equation*}
    \Var_p(f) = \inf \left \{ \left(\sup_{\{x_i\}_{i=0}^n \in P} \sum_{i=1}^n \abs{g(x_{i}) - g(x_{i-1})}^{p} \right)^{1/p} : g = f \text{ a.e.}\right\}.
  \end{equation*}
  The set of functions of bounded $p$-variation on $S^1$ is
  \begin{equation*}
    \BV_p(S^1) = \{ f \in L^1(S^1) : \Var_p(f) < \infty \}.
  \end{equation*}
  Functions bounded $p$-variation have been used to study the statistical properties of piecwise expanding dynamical systems: for $p = 1$ see \cite[Chapter 3]{baladi1997correlation} or \cite{lawsOfChaos}, while for general $p$ see \cite{keller1985generalized}.
  On $\BV_p(S^1)$ the map $f \mapsto \Var_p(f)$ is a seminorm and lower semicontinuous with respect to $\norm{\cdot}_{L^p}$.
  It follows that $\BV_p(S^1)$ is a Banach space when endowed with the norm $\norm{\cdot}_{\BV_p} = \norm{\cdot}_{L^p} + \Var_p(\cdot)$, and is also a Saks space with structure $(\BV_p(S^1), \norm{\cdot}_{\BV_p}, \norm{\cdot}_{L^p})$.
  Let $\E_n \in \LL(L^p(S_1))$ denote the conditional expectation operator associated to the uniform partition of $S^1$ into $n$ intervals. It is clear that $\E_n$ is a contraction on $L^p(S^1)$, and a straightforward calculation shows that the same is true for $\BV_p(S^1)$. One verifies that for every $f \in \BV_p(S^1)$ we have
  \begin{equation*}
    \norm{\E_n(f) - f}_{L^p} \le n^{-1/p}\Var_p(f),
  \end{equation*}
  and so $\E_n \to \Id$ in $\LL(\BV_p(S^1), L^p(S^1))$.
  Since each $\E_n$ has finite rank, it follows from the previous remarks that $B_{\norm{\cdot}_{\BV_p}}$ is $\norm{\cdot}_{L^p}$-compact.
  Hence $(\BV_p(S^1), \norm{\cdot}_{\BV_p}, \norm{\cdot}_{L^p})$ is a compact Saks space.
\end{example}

\begin{example}\label{example:sobolev}
  Fix $p \in [1, \infty)$. The Sobolev space $W^{1,p}(S^1)$ is defined by
  \begin{equation*}
    W^{1,p}(S^1) = \{ f \in L^p(S^1) : f' \text{ exists in the weak sense and }\norm{f'}_{L^p} < \infty \}.
  \end{equation*}
  Each $W^{1,p}(S^1)$ becomes a Banach space when equipped with the norm
  \begin{equation*}
    \norm{f}_{W^{1,p}} = \norm{f'}_{L^p} + \norm{f}_{L^p}.
  \end{equation*}
  It is well-known that $W^{1,1}(S^1)$ coincides with the set of absolutely continuous functions on $S^1$, and that $W^{1,p}(S^1) \subseteq W^{1,1}(S^1)$. Hence every $f \in W^{1,p}(S^1)$ is Riemann integrable. A short calculation then shows that
  \begin{equation*}
    \Var_p(f) = \norm{f'}_{L^p}
  \end{equation*}
  and so $\norm{f}_{\BV} = \norm{f}_{W^{1,p}}$ for every $f \in W^{1,p}(S^1)$. In view of Example \ref{example:bv} we may conclude that $(W^{1,p}(S^1), \norm{\cdot}_{W^{1,p}}, \norm{\cdot}_{L^p})$ is a pre-compact Saks space.
  Moreover, we claim that the Saks space completion of $(W^{1,p}(S^1), \norm{\cdot}_{W^{1,p}}, \norm{\cdot}_{L^p})$ is equal to $(\BV_p(S^1), \norm{\cdot}_{\BV_p}, \norm{\cdot}_{L^p})$.
  The continuous inclusion of the Saks space completion of $W^{1,p}(S^1)$ into $\BV_p(S^1)$ is clear.
  For the other direction, one fixes a sequence $\{q_\epsilon\}_{\epsilon > 0} \subseteq \mathcal{C}^\infty(S^1)$ that approximates the identity. For each $f \in \BV_p(S^1)$ the sequence $\{q_\epsilon \ast f\}_{\epsilon > 0}$ lies in $W^{1,p}(S^1)$, satisfies $\norm{(q_\epsilon \ast f) - f}_{L^p} \to 0$, and by Young's inequality we have $\sup_{\epsilon > 0} \norm{q_\epsilon \ast f}_{W^{1,p}} \le \norm{f}_{W^{1,p}}$.
  Thus $f$ is in the Saks space completion of $W^{1,p}(S^1)$, which confirms our earlier claim.
\end{example}

An obvious question at this stage is whether $\gamma[\norm{\cdot}, \tau]$ is metrisable, since a positive answer would reduce the study of Saks spaces to that of classically studied objects. For interesting examples, however, this is never the case.
\begin{proposition}[{\cite[Proposition 1.14]{cooper2011saks}}]\label{prop:metrisable_saks}
  If $\gamma[\norm{\cdot}, \tau]$ is metrisable then $\tau$ and $\norm{\cdot}$ are equivalent, in which case $\gamma[\norm{\cdot}, \tau]$ and $\norm{\cdot}$ are equivalent too.ß
\end{proposition}

Despite much of the theory holding when $\tau$ is a general locally convex topology, we now specialise to the case where $\tau$ is induced by a norm $\wnorm{\cdot}$. In this case we call $(X, \norm{\cdot}, \wnorm{\cdot})$ a \emph{normed} Saks space\footnote{This terminology should cause no confusion in view of Proposition \ref{prop:metrisable_saks}}.
Normed Saks spaces are also known as two-norm spaces, due to the pioneering papers by Alexiewicz and Semadeni (see e.g. \cite{alexiewicz1953two, alexiewicz1958linear,alexiewicz1959two}).
If Saks space $(X, \norm{\cdot}, \wnorm{\cdot})$ is a normed Saks space such that $\wnorm{\cdot} \le \norm{\cdot}$ then we say $(X, \norm{\cdot}, \wnorm{\cdot})$ is \emph{normal}.
Since $B_{\norm{\cdot}}$ is $\wnorm{\cdot}$-bounded we can make any normed Saks space normal after possibly rescaling either $\wnorm{\cdot}$ or $\norm{\cdot}$.

We now turn our attention to continuous linear maps between normed Saks spaces\footnote{Our approach here is based on \cite[Chapter 1, Section 3.11]{cooper1987saks} rather than on \cite[Chapter 1, Section 3.16]{cooper2011saks}, which is the equivalent section in the second edition.}.
Let $(X_{i},\norm{\cdot}_i,\wnorm{\cdot}_i)$, $i \in \{1,2\}$, be Saks spaces.
We denote by the space of continuous linear operators from $X_1$ to $X_2$ by $\LL_S(X_1,X_2)$. Let $\norm{\cdot}$ denote the strong operator norm:
\begin{equation*}
  \norm{A} = \sup_{\norm{f}_1 = 1} \norm{A f}_2,
\end{equation*}
and let $\tnorm{\cdot}$ be the triple norm:
\begin{equation*}
  \tnorm{A} = \sup_{\norm{f}_1 = 1} \wnorm{Af}_2.
\end{equation*}

\begin{proposition}\label{prop:linear_ops_saks}
  $\LL_S(X_1,X_2)$ is a Saks space with the structure $(\LL_S(X_1,X_2), \norm{\cdot}, \tnorm{\cdot})$.
\end{proposition}

The following two result are used to prove Proposition \ref{prop:linear_ops_saks}.

\begin{proposition}[{\cite[Proposition 1.9]{cooper2011saks}}]
  Suppose that $B_{\norm{\cdot}}$ is $\tau$-metrisable, then a linear mapping from $(X_1, \norm{\cdot}, \tau)$ into a topological vector space $X_2$ is continuous if and only if it is sequentially continuous.
\end{proposition}

\begin{lemma}\label{lemma:saks_boundedness}
  If $(X_i, \norm{\cdot}_i, \wnorm{\cdot}_i)$, $i \in \{ 1,2\}$, are Saks spaces then $\LL_S(X_1,X_2) \subseteq \LL(X_1,X_2)$.
\end{lemma}

We note that $\LL_S(X_1,X_2)$ is not necessarily equal to $\LL(X_1, X_2)$, although it is interesting to observe that the proof of Proposition \ref{prop:linear_ops_saks} also implies that $(\LL(X_1, X_2), \norm{\cdot}, \tnorm{\cdot})$ is a Saks space.
While it is clear that $\LL(X_1, X_2) \cap \LL((X_1, \wnorm{\cdot}_1), (X_2, \wnorm{\cdot}_2)) \subseteq \LL_S(X_1,X_2)$,
it is desirable to have a more quantitative characterisation of $\LL_S(X_1,X_2)$, similar to the characterisation of continuous operators between two Banach spaces as bounded operators.
The following proposition gives such a characterisation, and may even be used to characterise the equicontinuous families of operators in $\LL_S(X_1,X_2)$.

\begin{proposition}\label{prop:saks_equicont}
  Suppose $(X_i, \norm{\cdot}_i, \wnorm{\cdot}_i)$, $i \in \{1,2\}$ are Saks spaces, that $\mathcal{A}$ is an index set, and that for each $\alpha \in \mathcal{A}$ there exists a linear map $A_\alpha : X_1 \to X_2$.
  Then $\{A_\alpha\}_{\alpha \in \mathcal{A}}$ is an equicontinuous subset of $\LL_{S}(X_1, X_2)$ if and only if $\{A_\alpha\}_{\alpha \in \mathcal{A}}$ is an equicontinuous subset of $\LL(X_1, X_2)$ and for every $\eta > 0$ there exists $C_\eta > 0$ such that for every $\alpha \in \mathcal{A}$ and $f \in X_1$ we have
  \begin{equation}\label{eq:saks_equicont_0}
    \wnorm{A_\alpha f}_{2} \le \max\{\eta \norm{f}_1, C_\eta \wnorm{f}_1\}.
  \end{equation}
\end{proposition}

Proposition \ref{prop:saks_equicont} allows one to work with the inequality \eqref{eq:saks_equicont_0} in place of open sets of the form in \eqref{eq:ss_neighbourhood_basis}, which often leads to conceptully simpler proofs, such as that of the following proposition.

\begin{proposition}\label{prop:operator_saks_space_complete}
  Suppose $(X_i, \norm{\cdot}_i, \wnorm{\cdot}_i)$, $i \in \{1,2\}$ are Saks spaces.
  If $(X_2, \norm{\cdot}_2, \wnorm{\cdot}_2)$ is complete then $\LL_S(X_1,X_2)$ is complete.
\end{proposition}

We finish this section with some results on compact Saks spaces.
Compact Saks spaces frequently appear in dynamical systems literature due to their use in the Ionescu-Tulcea--Marinescu Theorem, which is also known as Hennion's Theorem due to a later strengthening by Hennion.

\begin{theorem}[{A Saks space version of the Ionescu-Tulcea--Marinescu Theorem \cite{hennion1993theoreme}}]\label{thm:ITMSS}
  Suppose that $(X, \norm{\cdot}, \wnorm{\cdot})$ is a (pre-)compact Saks space, and that $A \in \LL(X)$. If there exist sequences of real numbers $\{r_n\}_{n \in \Z^+}$ and $\{R_n\}_{n \in \Z^+}$ such that for each $n \in \Z^+$ and $f \in X$ we have
  \begin{equation}\label{eq:ITMSS_1}
    \norm{A^n f} \le r_n \norm{f} + R_n \wnorm{f},
  \end{equation}
  then $\rho_{\ess}(A) \le \liminf_{n \to \infty} r_n^{1/n}$.
\end{theorem}

Recall that an operator $A \in \LL(X)$ is said to be quasi-compact if $\rho_{\ess}(A) < \rho(A)$.
As outlined in the introduction, a modern approach to studying the statistical properties of a dynamical system $T : M \to M$, where $M$ is some Riemannian manifold, is by attempting to find a Banach space $X$ on which the Perron-Frobenius operator is quasi-compact and such that $\mathcal{C}^\infty(M) \into X \into (\mathcal{C}^\infty(M))'$.
The typical route for proving quasi-compactness is via Theorem \ref{thm:ITMSS} i.e. by endowing $X$ with the structure of a (pre-)compact Saks space and obtaining an appropriate Lasota-Yorke inequality, as in \eqref{eq:ITMSS_1}.
This connection prompts some questions about Saks spaces with high relevancy to dynamical systems:
\begin{enumerate}[label=(Q\arabic*)]
  \item \label{en:q1} What Banach spaces permit the structure of a compact Saks space?
  \item \label{en:q2} Given a Banach space which may be endowed with the structure of a compact Saks space, is this structure unique in any sense?
  \item \label{en:q3} If $(X, \norm{\cdot}, \wnorm{\cdot})$ is a compact Saks space, to what extent does $\norm{\cdot}$ determine $\wnorm{\cdot}$?
\end{enumerate}
The first question has a very satisfactory answer: a Banach space may be made into a compact Saks space if and only if it has a predual.
We state the result for the case where the Banach space permits the structure of a normed compact Saks space and refrain from giving all the relevant definitions. For the full result and definitions we refer the reader to \cite[Proposition 2.9]{cooper2011saks} and \cite[Chapter 1]{cooper2011saks}.
\begin{proposition}[{\cite[Proposition 2.9]{cooper2011saks}}]\label{prop:saks_space_finite}
  Let $(X,\norm{\cdot}, \tau)$ be a Saks space. Then the following are equivalent:
  \begin{enumerate}
    \item $B_{\norm{\cdot}}$ is compact and metrisable with respect to $\tau$.
    \item $X$ is the Saks space projective limit of a sequence of finite dimensional Banach spaces.
    \item $X$ has the form $(F^*, \norm{\cdot}, \sigma(F^*,F))$ for some separable Banach space $F$, where $\sigma(F^*,F)$ denotes the weak-$*$ topology on $F^*$.
    \item $B_{\norm{\cdot}}$ is compact and normable with respect to $\tau$ i.e. there is a norm $\wnorm{\cdot}$ on $X$ such that $\wnorm{\cdot}$ and $\tau$ are equivalent on $B_{\norm{\cdot}}$.
  \end{enumerate}
\end{proposition}

We answer \ref{en:q2} and \ref{en:q3} in the following two theorems.
\begin{theorem}\label{thm:compact_saks_unique}
  Suppose that $(X,\norm{\cdot})$ is a Banach space, $\tau$ is a locally convex topology on $X$ such that $(X,\norm{\cdot}, \tau)$ is a compact Saks space and that $D$ is a Hausdorff topological vector space $D$ such that $(X,\tau) \into D$.
  Then $\gamma[\norm{\cdot},\tau]$ is unique up to $D$ i.e. if $\tau'$ is a locally convex topology on $X$ such that $(X,\norm{\cdot}, \tau')$ is a compact Saks space with $(X,\tau') \into D$ then $\gamma[\norm{\cdot}, \tau] = \gamma[\norm{\cdot}, \tau']$.
\end{theorem}

\begin{theorem}\label{thm:saks_space_norm_formula}
  For a bounded countable family of functionals $\Phi = \{\varphi_n \}_{n \in \Z^+} \in \LL(X, \C)$ we set
  \begin{equation*}
    \norm{f}_{\Phi} = \sup_{\varphi \in \Phi} \abs{\varphi(f)}
  \end{equation*}
  and
  \begin{equation*}
    \wnorm{f}_{\Phi} = \sum_{n \in \Z^+} 2^{-n} \abs{\varphi_n(f)}.
  \end{equation*}
  If $(X, \norm{\cdot}, \wnorm{\cdot})$ is a compact Saks space then there exists $\Phi =\{\varphi_n \}_{n \in \Z^+} \in \LL(X, \C)$ such that $\norm{\varphi_n} = 1$ for every $n \in \Z^+$ then $\norm{\cdot}$ is equivalent to $\norm{\cdot}_\Phi$, and $\wnorm{\cdot}$ is equivalent to $\wnorm{\cdot}_{\Phi}$ on $\norm{\cdot}$-bounded sets.
  In particular $\gamma[\norm{\cdot}, \wnorm{\cdot}]$ and $\gamma[\norm{\cdot}_\Phi, \wnorm{\cdot}_\Phi]$ are equivalent.
\end{theorem}

\section{Saks space stability of hyperbolic splittings for Lasota-Yorke cocycles}\label{sec:saks_space_stability}

Our main result for this section concerns the stability of hyperbolic splittings (Definition \ref{def:hyperbolic_splitting}) for operator cocycles satisfying a Lasota-Yorke inequality (Definition \ref{def:ly}) and certain Saks space equicontinuity conditions (Definition \ref{def:equicont_seq}).

Let us fix some notation.
Let $\Omega$ be a set, and $\sigma : \Omega \to \Omega$ be an invertible map. For each $\omega \in \Omega$ let $(X_\omega, \norm{\cdot}_\omega, \wnorm{\cdot}_\omega)$ be a normal Saks space, with each $(X_{\omega}, \norm{\cdot}_\omega)$ being a Banach space\footnote{It will be important later that $\LL(X_\omega)$ is complete.}.
We will consider the vector space bundle\footnote{Note that we do not endow $\mathbb{X}$ with a topology.} $\mathbb{X} = \bigsqcup_{\omega \in \Omega} \{\omega\} \times X_\omega$.
Let $\pi : \mathbb{X} \to \Omega$ denote the projection onto $\Omega$, and for each $\omega \in \Omega$ let $\tau_\omega : \pi^{-1}(\omega) \to X_\omega$ be defined by $\tau_\omega(\omega, f) = f$.
We say that $L : \mathbb{X} \to \mathbb{X}$ is a \emph{bounded linear endomorphism of $\mathbb{X}$ covering $\sigma$} if $\pi \circ L = \sigma \circ \pi$ and if $f \mapsto \tau_{\sigma(\omega)}(L(\omega, f))$ is in $\LL(X_\omega, X_{\sigma(\omega)})$ for every $\omega \in \Omega$.
We denote the set of all bounded linear endomorphisms of $\mathbb{X}$ covering $\sigma$ by $\End(\mathbb{X},\sigma)$.
When $n \in \N$, $\omega \in \Omega$, and $L \in \End(\mathbb{X},\sigma)$ we denote the map $f \mapsto \tau_{\sigma(\omega)}(L(\omega, f))$ by $L_\omega$ and set
\begin{equation*}
  L_\omega^n =
  \begin{cases}
    L_{\sigma^{n-1}(\omega)} \circ \cdots \circ L_{\omega} & \text{if } n \ge 1,\\
    \Id & \text{if } n = 0.
  \end{cases}
\end{equation*}
Clearly $L_\omega^n \in \LL(X_{\omega}, X_{\sigma^n(\omega)})$ for every $n \in \N$ and $\omega \in \Omega$.
Unless required we will frequently drop the subscript $\omega$ from $\norm{\cdot}_\omega$ and $\wnorm{\cdot}_\omega$.
We denote the norm on $\LL(X_{\omega},X_{\sigma(\omega)})$ by $\norm{\cdot}$, the norm on $\LL((X_{\omega}, \wnorm{\cdot}_{\omega}),(X_{\sigma(\omega)}, \wnorm{\cdot}_{\sigma(\omega)}))$ by $\wnorm{\cdot}$, and the norm on $\LL((X_{\omega}, \norm{\cdot}_{\omega}),(X_{\sigma(\omega)}, \wnorm{\cdot}_{\sigma(\omega)}))$ by $\tnorm{\cdot}$.

\begin{definition}\label{def:hyperbolic_splitting}
  Suppose that $L \in \End(\mathbb{X},\sigma)$, $d \in \Z^+$, $0 \le \mu < \lambda$, $(E_\omega)_{\omega \in \Omega} \in \prod_{\omega \in \Omega} \mathcal{G}_d(X_\omega)$ and $(F_\omega)_{\omega \in \Omega} \in \prod_{\omega \in \Omega} \mathcal{G}^d(X_\omega)$.
  We say that $(E_\omega)_{\omega \in \Omega}$ and $(F_\omega)_{\omega \in \Omega}$ form a $(\mu, \lambda,d)$-hyperbolic splitting for $L$, and that $L$ has a hyperbolic splitting of index $d$, if there exists constants $C_\lambda, C_\mu, \Theta > 0$ such that:
  \begin{enumerate}[label=\textnormal{(H\arabic*)}]
    \item \label{en:h1} For every $\omega \in \Omega$ we have $E_\omega \oplus F_\omega = X_\omega$ and
    \begin{equation}\label{eq:h1_angle}
      \max\{\norm{\Pi_{F_\omega || E_\omega}}, \norm{\Pi_{E_\omega || F_\omega}}\} \le \Theta.
    \end{equation}
    \item \label{en:h2} For each $\omega \in \Omega$ we have $L_\omega E_\omega = E_{\sigma(\omega)}$. Moreover, for every $n \in \Z^+$ and $f \in E_\omega$ we have
    \begin{equation}\label{eq:h2_expansion}
      \norm{L_\omega^n f } \ge C_\lambda \lambda^{n} \norm{f}.
    \end{equation}
    \item \label{en:h3} For each $\omega \in \Omega$ we have $L_\omega F_\omega \subseteq F_{\sigma(\omega)}$ and for every $n \in \Z^+$ we have
    \begin{equation}\label{eq:h3_decay}
      \norm{\restr{L_\omega^n}{F_\omega}} \le C_\mu \mu^{n}.
    \end{equation}
  \end{enumerate}
  We call $(E_\omega)_{\omega \in \Omega}$ the equivariant fast spaces for $L$, and $(F_\omega)_{\omega \in \Omega}$ the equivariant slow spaces for $L$.
\end{definition}

\begin{remark}
  Since $L_\omega E_\omega = E_{\sigma(\omega)}$ and $L_\omega F_\omega \subseteq F_{\sigma(\omega)}$ we have \begin{equation*}
    L_{\sigma(\omega)} \Pi_{E_\omega || F_\omega} = \Pi_{E_{\sigma(\omega)} || F_{\sigma(\omega)}} L_\omega \text{ and } L_{\sigma(\omega)} \Pi_{F_\omega || E_\omega} = \Pi_{F_{\sigma(\omega)} || E_{\sigma(\omega)}} L_\omega.
  \end{equation*}
\end{remark}

We will now describe the elements of $\End(\mathbb{X}, \sigma)$ which are `equicontinuous in the Saks space sense'.

\begin{definition}[Saks space continuous endomorphisms]\label{def:equicont_seq}
  We say that $L \in \End(\mathbb{X},\sigma)$ is a \emph{Saks space equicontinuous endomorphism} if $\sup_{\omega \in \Omega} \norm{L_\omega} < \infty$ and if for each $\eta > 0$ there exists $C_\eta > 0$ such that for every $\omega \in \Omega$ and $f \in X_\omega$ we have
  \begin{equation}\label{eq:equicont_seq_1}
     \wnorm{L_{\omega} f} \le \eta \norm{f} + C_\eta \wnorm{f}.
  \end{equation}
  We denote the set of all Saks space equicontinuous endomorphisms in $\End(\mathbb{X}, \sigma)$ by $\End_S(\mathbb{X},\sigma)$.
\end{definition}

\begin{remark}\label{remark:ss_equi}
  Proposition \ref{prop:saks_equicont} justifies the characterisation of the condition in Definition \ref{def:equicont_seq} as an equicontinuity condition.
  Indeed, when all the spaces $(X_\omega, \norm{\cdot}_\omega, \wnorm{\cdot}_\omega)$ are to a fixed space $(X, \norm{\cdot}, \wnorm{\cdot})$, then $L \in \End_S(\mathbb{X}, \sigma)$ if and only if the set $\{L_\omega\}_{\omega \in \Omega}$ is equicontinuous in $\LL_S(X)$.
\end{remark}

\begin{remark}
  $\End_S(\mathbb{X},\sigma)$ admits an interesting alternative characterisation.
  For $(f_\omega)_{\omega \in \Omega} \in \prod_{\omega \in \Omega} X_\omega$ let $\norm{(f_\omega)_{\omega \in \Omega}}_\infty = \sup_{\omega \in \Omega} \norm{f_\omega}_\omega$ and $\wnorm{(f_\omega)_{\omega \in \Omega}}_\infty = \sup_{\omega \in \Omega} \wnorm{f_\omega}_\omega$.
  The set
  \begin{equation*}
    X_\infty = \left\{ (f_\omega)_{\omega \in \Omega} \in \prod_{\omega \in \Omega} X_\omega : \norm{(f_\omega)_{\omega \in \Omega}}_\infty < \infty\right\}
  \end{equation*}
  is a Banach space when equipped with $\norm{\cdot}_\infty$, and a normal Saks space when given the structure $(X_\infty, \norm{\cdot}_\infty, \wnorm{\cdot}_\infty)$.
  For $L \in \End(\mathbb{X}, \sigma)$, one can show that $L \in \End_S(\mathbb{X}, \sigma)$ if and only if the map $(f_\omega)_{\omega \in \Omega} \mapsto (L_\omega f_\omega)_{\omega \in \Omega}$ is in $\LL_S(X_\infty)$.
\end{remark}

\begin{remark}
  If $L \in \End(\mathbb{X}, \sigma)$ satisfies $\sup_{\omega \in \Omega} \norm{L_\omega} < \infty$ and $\sup_{\omega \in \Omega} \wnorm{L_\omega} < \infty$ then $L \in \End_S(\mathbb{X}, \sigma)$.
  In many applications it is easier to bound $\sup_{\omega \in \Omega} \wnorm{L_\omega}$ than it is to obtain the inequality \eqref{eq:equicont_seq_1}.
\end{remark}

We will only consider endomorphisms that satisfy a uniform Lasota-Yorke inequality.
\begin{definition}[Lasota-Yorke class]\label{def:ly}
  For $C_1, C_2, r, R \ge 0$ we denote by $\mathcal{LY}(C_1, C_2, r, R)$ the set of $L \in \End(\mathbb{X}, \sigma)$ such that for every $\omega \in \Omega$, $f \in X_\omega$ and $n \in \Z^+$ we have
  \begin{equation}\label{eq:ly}
    \norm{L_{\omega}^n f} \le C_1 r^n \norm{f} + C_2 R^n \wnorm{f}.
  \end{equation}
\end{definition}

\begin{remark}
  If $L \in \mathcal{LY}(C_1, C_2, r, R)$ then
  \begin{equation*}
    \norm{L_{\omega}^n} \le C_1 r^n + C_2 R^n \le C_3\max\{r, R\}^n,
  \end{equation*}
  where $C_3 = C_1 + C_2$.
  We will only deal with the case where $r \le R$, which we note is the case when $L$ admits a $(\mu, \lambda,d)$-hyperbolic splitting with $\mu > r$, as then $r \le \mu < \lambda \le R$.
  Hence for every $\omega \in \Omega$ and $n \in \Z^+$ we have
  \begin{equation}\label{eq:power_bound}
    \norm{L_\omega^n} \le C_3R^n.
  \end{equation}
\end{remark}

Finally, if $L \in \End_S(\mathbb{X}, \sigma)$ then for $\epsilon > 0$ we set
\begin{equation*}
  \mathcal{O}_{\epsilon}(L) = \left\{ S \in \End(\mathbb{X}, \sigma) : \sup_{\omega \in \Omega} \tnorm{L_\omega - S_\omega} < \epsilon \right\}.
\end{equation*}

Our main result is the following.
\begin{theorem}\label{thm:stability_cocycle}
  Fix $\mu, \lambda, C_1, C_2, R \ge 0$, with $0 \le r < \mu < \lambda$, $d \in \Z^+$, $(E_\omega)_{\omega \in \Omega} \in \prod_{\omega \in \Omega} \mathcal{G}_d(X_\omega)$ and $(F_\omega)_{\omega \in \Omega} \in \prod_{\omega \in \Omega} \mathcal{G}^d(X_\omega)$.
  Suppose that $L \in \End_S(\mathbb{X}, \sigma) \cap \mathcal{LY}(C_1, C_2, r, R)$ has a $(\mu,\lambda,d)$-hyperbolic splitting composed of fast spaces $(E_\omega)_{\omega \in \Omega}$ and slow spaces $(F_\omega)_{\omega \in \Omega}$.
  There exists $\epsilon' > 0$ so that
  \begin{enumerate}
    \item If $S \in \mathcal{LY}(C_1, C_2, r, R) \cap \mathcal{O}_{\epsilon'}(L)$
    then $S$ has a hyperbolic splitting of index $d$.
    \item If $(E_\omega^S)_{\omega \in \Omega} \in \prod_{\omega \in \Omega} \mathcal{G}_d(X_\omega)$ and $(F_\omega^S)_{\omega \in \Omega} \in \prod_{\omega \in \Omega} \mathcal{G}^d(X_\omega)$ denote the equivariant fast and slow spaces for $S$ then
    \begin{equation}\label{eq:stability_cocycle_1}
      \sup \left\{ \norm{\Pi_{E^S_\omega || F^S_\omega}}  : \omega \in \Omega, S \in \mathcal{LY}(C_1, C_2, r, R) \cap \mathcal{O}_{\epsilon'}(L) \right\} < \infty.
    \end{equation}
  \end{enumerate}
  Moreover, for every $\beta \in (0,(\lambda - \mu)/2)$ and $\delta >0$ there exists $\epsilon_{\beta, \delta} \in (0,\epsilon')$ and $C_\beta > 0$ so that if $S \in \mathcal{LY}(C_1, C_2, r, R) \cap \mathcal{O}_{\epsilon_{\beta, \delta}}(L)$
  then
  \begin{enumerate}
    \item We have the estimates
    \begin{equation}\label{eq:stability_cocycle_2}
      \sup_{\omega \in \Omega} \tnorm{\Pi_{E_{\omega}^S || F_{\omega}^S} - \Pi_{E_{\omega} || F_{\omega}} } \le \delta,
    \end{equation}
    and
    \begin{equation}\label{eq:stability_cocycle_3}
      \sup_{\omega \in \Omega} d_H(F_{\omega}^S, F_\omega) \le \delta.
    \end{equation}
    \item The spaces $(E_{\omega}^S)_{\omega \in \Omega}$ and $(F_{\omega}^S)_{\omega \in \Omega}$ form a $(\mu + \beta, \lambda - \beta, d)$-hyperbolic splitting for $S$. More specifically, for every $\omega \in \Omega$ and $n \in \Z^+$ we have
    \begin{equation}\label{eq:stability_cocycle_4}
      \norm{\restr{S_{\omega}^n}{F_{\omega}^S} } \le C_\beta (\mu + \beta)^n,
    \end{equation}
    and, for every $v \in E_{\omega}^S$, that
    \begin{equation}\label{eq:stability_cocycle_5}
      \norm{S_{\omega}^n v} \ge C_\beta^{-1} (\lambda - \beta)^n \norm{v}.
    \end{equation}
  \end{enumerate}
\end{theorem}

\begin{remark}
  In principle one could compute explicit bounds on the various quantities in the statement of Theorem \ref{thm:stability_cocycle}, such as $\epsilon_{\beta, \delta}$ or the supremum in \eqref{eq:stability_cocycle_1}.
  We opted not to pursue such bounds for the sake of simplicity.
\end{remark}

\begin{remark}
  It is possible to obtain an estimate on the distance between $E^S_\omega$ and $E_\omega$ in the Grassmannian distance on $(X_\omega, \wnorm{\cdot}_\omega)$ from \eqref{eq:stability_cocycle_2} by using \cite[Proposition 2.4]{froyland2014detecting}.
\end{remark}

The strategy behind the proof of Theorem \ref{thm:stability_cocycle} is reminiscent of the usual proof that the class of Anosov maps is open \cite[Corollary 5.5.2]{brin2002introduction}, and is quite similar to the overall strategy in \cite{bogenschutz2000stochastic}.
We start by collecting some preliminary estimates and results in Section \ref{sec:estimates}.
In Section \ref{sec:fast_space} we construct invariant `fast' cones of $d$-dimensional subspaces, defined in terms of the graph representation of the hyperbolic splitting of $L$, and show that the forward graph transform induced by an iterate of the perturbed cocycle is a contraction mapping on these cones.
We then prove that perturbed fast spaces approximate, in a Saks space sense, the unperturbed fast spaces.
Once the fast spaces have been constructed, it is quite easy to construct the slow spaces, which is the subject of Section \ref{sec:slow_space}.
The primary difference is that we consider a different graph representation of the Grassmannian -- the one associated to the splitting of each $X_\omega$ into the perturbed fast space and the unperturbed slow space.
The backwards graph transform of an iterate of the perturbed cocycle is a contraction mapping on the set of all such graphs, and we recover the perturbed slow spaces as the transform's fixed point.
We may then prove that the perturbed slow spaces approximate the unperturbed slows spaces.
In Section \ref{sec:end_of_proof} we bring together the results of the previous sections to complete the proof of Theorem \ref{thm:stability_cocycle}.
An advantage of working with the graph representation of the Grassmannian is that we mainly work with projections rather than with subspaces. This is nice both conceptually and technically, particularly when proving the aforementioned stability results.

\begin{remark}
  Until the proof of Theorem \ref{thm:stability_cocycle} is concluded we shall use $L$ to refer to an element of $\End_S(\mathbb{X}, \sigma) \cap \mathcal{LY}(C_1, C_2, r, R)$ which satisfies all the hypotheses of Theorem \ref{thm:stability_cocycle}.
\end{remark}

\subsection{Preliminary estimates and lemmata}\label{sec:estimates}

The following estimate forms the backbone of the proof of Theorem \ref{thm:stability_cocycle}.

\begin{proposition}\label{prop:perturbed_spaces_bound}
  For every $\beta \in (0, (\lambda - \mu)/2)$ there exists $N_\beta$ and for each $n  > N_\beta$ an $\epsilon_{n,\beta}>0$ so that if $S \in \mathcal{LY}(C_1, C_2, r, R) \cap \mathcal{O}_{\epsilon_{n,\beta}}(L)$ and $\omega \in \Omega$ then
  \begin{equation}\label{eq:perturbed_spaces_bound_slow}
    \norm{\restr{S_{\omega}^n}{F_\omega}} \le (\mu+\beta)^n,
  \end{equation}
  and if $v \in E_\omega$ then
  \begin{equation}\label{eq:perturbed_spaces_bound_fast}
    \norm{\Pi_{E_{\sigma^n(\omega)} || F_{\sigma^n(\omega)}}S_{\omega}^n v} \ge (\lambda-\beta)^n \norm{v}.
  \end{equation}
\end{proposition}

The proof of Proposition \ref{prop:perturbed_spaces_bound} is split over the following lemmas, all of which are of independent interest.

\begin{lemma}\label{lemma:fast_space_eccentricity}
  There exists $K$ such that for every $\omega \in \Omega$ and $v \in X_\omega$ we have
  \begin{equation*}
    \norm{\Pi_{E_\omega || F_\omega} v} \le K \wnorm{\Pi_{E_\omega || F_\omega} v}.
  \end{equation*}
  \begin{proof}
    From \eqref{eq:ly} we have
    \begin{equation*}
      \norm{L_\omega^n \Pi_{E_{\omega} || F_\omega} v} \le C_1 r^n \norm{\Pi_{E_{\omega} || F_\omega} v} + C_2 R^n \wnorm{\Pi_{E_{\omega} || F_\omega} v},
    \end{equation*}
    while, on the other hand, by \eqref{eq:h2_expansion} we have $\norm{L_\omega^n \Pi_{E_{\omega} || F_\omega} v} \ge C_\lambda \lambda^n \norm{\Pi_{E_{\omega} || F_\omega} v}.$
    Since $\lambda > r$ there exists $N$ such that $C_\lambda \lambda^N - C_1 r^N > 0$, and so
    \begin{equation*}
      \norm{\Pi_{E_{\omega} || F_\omega} v} \le \frac{C_2 R^N}{C_\lambda \lambda^N - C_1 r^N} \wnorm{\Pi_{E_{\omega} || F_\omega} v}.
    \end{equation*}
  \end{proof}
\end{lemma}

\begin{lemma}\label{lemma:fast_space_ss}
  For every $\eta > 0$ there exists $C_\eta > 0$ such that for every $\omega \in \Z$ and $v \in X_\omega$ we have
  \begin{equation*}
    \norm{\Pi_{E_\omega || F_\omega} v} \le \eta \norm{v} + C_\eta \wnorm{v}.
  \end{equation*}
  \begin{proof}
    By \eqref{eq:h1_angle} and \eqref{eq:ly} we have
    \begin{equation*}
      \norm{L_\omega^n \Pi_{E_{\omega} || F_\omega} v} = \norm{\Pi_{E_{\sigma^n(\omega)} || F_{\sigma^n(\omega)}} L_\omega^n v} \le \Theta C_1 r^n \norm{ v} + \Theta C_2 R^n \wnorm{v}.
    \end{equation*}
    On the other hand, by \eqref{eq:h2_expansion} we have
    \begin{equation*}
      \norm{L_\omega^n \Pi_{E_{\omega} || F_\omega} v} \ge C_\lambda \lambda^n \norm{ \Pi_{E_{\omega} || F_\omega}v}.
    \end{equation*}
    Hence
    \begin{equation*}
      \norm{ \Pi_{E_{\omega} || F_\omega}v} \le \Theta C_1 \left(\frac{r}{\lambda}\right)^n \norm{ v} + \Theta C_2 \left(\frac{R}{\lambda}\right)^n \wnorm{ v}.
    \end{equation*}
    Since $\lambda > r$, by taking $n$ large enough we may ensure that $\Theta C_1 \left(\frac{r}{\lambda}\right)^n < \eta$, from which the result immediately follows.
  \end{proof}
\end{lemma}

\begin{remark}\label{remark:equicont_proj}
    Since $\wnorm{\cdot} \le \norm{\cdot}$, by Lemma \ref{lemma:fast_space_ss} and Proposition \ref{prop:saks_equicont} we have $\Pi_{E_\omega || F_\omega}, \Pi_{F_\omega || E_\omega} \in \LL_S(X_\omega)$ for each $\omega$.
    It is interesting to note that we did not use the fact that $L \in \End_S(\mathbb{X}, \sigma)$ in the proof of Lemma \ref{lemma:fast_space_ss}.
\end{remark}

\begin{lemma}\label{lemma:equicont_power}
  If $R, S \in \End_S(\mathbb{X}, \sigma)$ then $R_1S \in \End_S(\mathbb{X}, \sigma^2)$.
  Hence if $R \in \End_S(\mathbb{X}, \sigma)$ then $R^n \in \End_S(\mathbb{X}, \sigma^n)$ for every $n \in \Z^+$.
  \begin{proof}
    We of course have
    \begin{equation*}
      \sup_{\omega \in \Omega} \norm{R_{\sigma(\omega)} S_\omega} \le \left(\sup_{\omega \in \Omega} \norm{R_{\omega}}\right)\left( \sup_{\omega \in \Omega} \norm{S_\omega} \right)< \infty.
    \end{equation*}
    Let $\eta > 0$.
    For every $\kappa_1, \kappa_2 > 0$ there exists $C_{\kappa_1}, D_{\kappa_2}$ such that for every $\omega \in \Omega$, $f \in X_\omega$ and $g \in X_{\sigma(\omega)}$ we have $\wnorm{S_{\omega} f} \le \max\{\kappa_1 \norm{f}, C_{\kappa_1} \wnorm{f}\}$ and $\wnorm{R_{\sigma(\omega)} g} \le \max\{\kappa_2 \norm{g}, D_{\kappa_2} \wnorm{g}\}$.
    Hence for every $\omega \in \Omega$ and $f \in X_\omega$ we have
    \begin{equation*}
      \wnorm{R_{\sigma(\omega)}S_{\omega} f} \le \max\{\kappa_2 \norm{S_\omega f}, D_{\kappa_2} \wnorm{S_\omega f}_{\omega} \} \le \max\left\{\max\{\kappa_2 \norm{S_\omega}, \kappa_1 D_{\kappa_2}\} \norm{f}, C_{\kappa_1}D_{\kappa_2} \wnorm{f} \right\}.
    \end{equation*}
    Setting $\kappa_2 = (\sup_{\omega \in \Omega}\norm{S_\omega})^{-1} \eta$, $\kappa_1 = D_{\kappa_2}^{-1}\eta$ and $C_\eta =  C_{\kappa_1}D_{\kappa_2}$ yields $\wnorm{R_{\sigma(\omega)}S_\omega f} \le \max\left\{ \eta \norm{f}, C_\eta \wnorm{f} \right\}$.
    Thus $R S \in \End_S(\mathbb{X}, \sigma^2)$.
  \end{proof}
\end{lemma}

\begin{lemma}\label{lemma:tnorm_vanish}
  For every $n \in \Z^+$ and $\epsilon > 0$ there exists $\kappa > 0$ so that if $S \in \mathcal{LY}(C_1, C_2, r, R) \cap \mathcal{O}_{\kappa}(L)$ then
  \begin{equation}\label{eq:tnorm_vanish_0}
    \sup_{\omega \in \Omega} \tnorm{L_\omega^n - S_\omega^n} \le \epsilon.
  \end{equation}
  \begin{proof}
    For $S \in \mathcal{LY}(C_1, C_2, r, R)$ we may write
    \begin{equation*}
      L_\omega^n - S_\omega^n = \sum_{k=0}^{n-1} L_{\sigma^{k+1}(\omega)}^{n- k - 1}(L_{\sigma^k(\omega)} - S_{\sigma^k(\omega)})S_{\omega}^k.
    \end{equation*}
    Since $L \in \End_S(\mathbb{X}, \sigma)$, Lemma \ref{lemma:equicont_power} and \eqref{eq:power_bound} imply that for every $\eta > 0$ there exists $C_\eta$ such that for every $\omega \in \Omega$ and $S \in \mathcal{LY}(C_1, C_2, r, R) \cap \mathcal{O}_{\kappa}(L)$ we have
    \begin{equation*}\begin{split}
      \tnorm{L_\omega^n - S_\omega^n} &\le \sum_{k=0}^{n-1} \left(\eta \norm{L_{\sigma^k(\omega)} - S_{\sigma^k(\omega)}} + C_\eta \tnorm{L_{\sigma^k(\omega)} - S_{\sigma^k(\omega)} }\right) \norm{S_{\omega}^k} \\
      &\le C_3 \sum_{k=0}^{n-1} \left(2C_3 R \eta + \kappa C_\eta  \right) R^k \le nC_3\left(2C_3 R \eta + \kappa C_\eta  \right) \max\{1, R^n\}.
    \end{split}\end{equation*}
    We obtain \eqref{eq:tnorm_vanish_0} by choosing $\eta$ so that $n2C_3^2 R \eta\max\{1, R^n\} < \epsilon / 2$, and then taking $\kappa$ small enough so that $nC_3 \kappa C_\eta \max\{1, R^n\} < \epsilon / 2$
  \end{proof}
\end{lemma}

\begin{proof}[{The proof of Proposition \ref{prop:perturbed_spaces_bound}}]
  We prove \eqref{eq:perturbed_spaces_bound_slow} and \eqref{eq:perturbed_spaces_bound_fast} separately.
  ~\paragraph{The proof of \eqref{eq:perturbed_spaces_bound_slow}.}
  By telescoping we have
  \begin{equation}\label{eq:perturbed_spaces_bound_1}
    \norm{\restr{S_{\omega}^n}{F_\omega}} \le \norm{\restr{L_{\omega}^n}{F_\omega}} + \norm{\restr{(L_{\omega}^n - S_{\omega}^n)}{F_{\omega}}}\le \norm{\restr{L_{\omega}^n}{F_{\omega}}} + \sum_{k=0}^{n-1}
    \norm{S_{\sigma^{k+1}(\omega)}^{n-k-1}  (L_{\sigma^{k}(\omega)} - S_{\sigma^{k}(\omega)}) \restr{L_{\omega}^k}{F_{\omega}}}.
  \end{equation}
  Since $S, L \in \mathcal{LY}(C_1, C_2, r, R)$, by \eqref{eq:ly} and \eqref{eq:power_bound} we get
  \begin{equation}\begin{split}\label{eq:perturbed_spaces_bound_2}
    \sum_{k=0}^{n-1}
    &\norm{S_{\sigma^{k+1}(\omega)}^{n-k-1}  (L_{\sigma^{k}(\omega)} - S_{\sigma^{k}(\omega)}) \restr{L_{\omega}^k}{F_{\omega}}} \\
    &\le \sum_{k=0}^{n-1}
    C_1 r^{n-k-1}\norm{(L_{\sigma^{k}(\omega)} - S_{\sigma^{k}(\omega)}) \restr{L_{\omega}^k}{F_{\omega}}} + C_2 R^{n-k-1} \tnorm{(L_{\sigma^{k}(\omega)} - S_{\sigma^{k}(\omega)}) \restr{L_{\omega}^k}{F_{\omega}}} \\
    &\le \sum_{k=0}^{n-1} \left(
    2C_1C_3R r^{n-k-1}  + C_2 R^{n-k-1} \tnorm{L_{\sigma^{k}(\omega)} - S_{\sigma^{k}(\omega)}} \right)\norm{\restr{L_{\omega}^k}{F_{\omega}}}.
  \end{split}\end{equation}
  Combining \eqref{eq:perturbed_spaces_bound_1} and \eqref{eq:perturbed_spaces_bound_2}, and then Applying \eqref{eq:h3_decay} yields
  \begin{equation*}
    \norm{\restr{S_{\omega}^n}{F_{\omega}}} \le C_\mu \mu^n + C_\mu\mu^n \sum_{k=0}^{n-1} \left(
    2C_1C_3R\mu^{-1} \left(\frac{r}{\mu}\right)^{n-k-1} + C_2\mu^{-1}\left( \frac{R}{\mu}\right)^{n-k-1} \tnorm{L_{\sigma^{k}(\omega)} - S_{\sigma^{k}(\omega)}} \right).
  \end{equation*}
  By Lemma \ref{lemma:tnorm_vanish} and as $\mu > r$, for any $n \in \Z^+$ there exists $\epsilon_n > 0$ so that if $S \in \mathcal{O}_{\epsilon_n}(L) \cap\mathcal{LY}(C_1, C_2, r, R)$ then
  \begin{equation*}\begin{split}
    \sum_{k=0}^{n-1} \left(
    2C_1C_3R\mu^{-1} \left(\frac{r}{\mu}\right)^{n-k-1} + C_2\mu^{-1}\left( \frac{R}{\mu}\right)^{n-k-1} \tnorm{L_{\sigma^{k}(\omega)} - S_{\sigma^{k}(\omega)}} \right) &\le 3 C_1C_3R\mu^{-1}\sum_{k=0}^{n-1}
     \left(\frac{r}{\mu}\right)^{k}\\
    &\le \frac{3C_1C_3R}{\mu - r} := C'.
  \end{split}\end{equation*}
  Thus if $S \in \mathcal{O}_{\epsilon_n}(L) \cap\mathcal{LY}(C_1, C_2, r, R)$ and $\omega \in \Omega$ then $\norm{\restr{S_{\omega}^n}{F_{\omega}}} \le C_\mu C' \mu^n$.
  Setting $ N_\beta = \log(C_\mu C') / \log(1 + \beta/ \mu)$, we therefore get \eqref{eq:perturbed_spaces_bound_slow} whenever $n > N_\beta$, $S \in \mathcal{O}_{\epsilon_n}(L) \cap\mathcal{LY}(C_1, C_2, r, R)$ and $\omega \in \Omega$.

  \paragraph{The proof of \eqref{eq:perturbed_spaces_bound_fast}.}
  As $\wnorm{\cdot} \le \norm{\cdot}$, for each $v \in E_{\omega}$ we have
  \begin{equation}\begin{split}\label{eq:fast_space_expansion_1}
    \norm{\Pi_{E_{\sigma^n(\omega)} || F_{\sigma^n(\omega)}} S_{\omega}^n v}
    &\ge \wnorm{\Pi_{E_{\sigma^n(\omega)} || F_{\sigma^n(\omega)}} L_{\omega}^n v} - \wnorm{\Pi_{E_{\sigma^n(\omega)} || F_{\sigma^n(\omega)}}(L_{\omega}^n - S_{\omega}^n) v}.
  \end{split}\end{equation}
  Using Lemma \ref{lemma:fast_space_eccentricity}, \eqref{eq:h2_expansion} and the fact that $L_{\omega}^n v \in E_{\sigma^n(\omega)}$, we get
  \begin{equation}\label{eq:fast_space_expansion_2}
    \wnorm{\Pi_{E_{\sigma^n(\omega)} || F_{\sigma^n(\omega)}} L_{\omega}^n v} \ge K^{-1} \norm{L_{\omega}^n v} \ge K^{-1} C_\lambda \lambda^n \norm{v}.
  \end{equation}
  On the other hand, by Lemma \ref{lemma:fast_space_ss} we have for every $\eta > 0$ that
  \begin{equation}\begin{split}\label{eq:fast_space_expansion_3}
    \wnorm{\Pi_{E_{\sigma^n(\omega)} || F_{\sigma^n(\omega)}}(L_{\omega}^n - S_{\omega}^n) v} &\le \eta \norm{(L_{\omega}^n - S_{\omega}^n) v}+ C_\eta \wnorm{(L_{\omega}^n - S_{\omega}^n) v}\\
    &\le \left(2\eta C_3 R^n + C_\eta \tnorm{L_{\omega}^n - S_{\omega}^n} \right)\norm{v}.
  \end{split}\end{equation}
  Applying \eqref{eq:fast_space_expansion_2} and \eqref{eq:fast_space_expansion_3} to \eqref{eq:fast_space_expansion_1} yields
  \begin{equation*}
    \norm{\Pi_{E_{\sigma^n(\omega)} || F_{\sigma^n(\omega)}} S_{\omega}^n v} \ge \left(K^{-1} C_\lambda \lambda^n - 2\eta C_3 R^n - C_\eta \tnorm{L_{\omega}^n - S_{\omega}^n} \right) \norm{v}.
  \end{equation*}
  For each $n$ let $\eta$ be small enough so that $2\eta C_3 K C_\lambda^{-1} R^n < (\lambda)^n/4$.
  By Lemma \ref{lemma:tnorm_vanish} there exists $\epsilon_{n,\beta}$ so that if $S \in \mathcal{O}_{\epsilon_{n,\beta}}(L) \cap\mathcal{LY}(C_1, C_2, r, R)$ then $K C_{\eta}C_\lambda^{-1} \tnorm{L_{\omega}^n - S_{\omega}^n} < (\lambda - \beta)^n/4$.
  Thus if $S \in \mathcal{O}_{\epsilon_{n,\beta}}(L) \cap\mathcal{LY}(C_1, C_2, r, R)$ then $\norm{\Pi_{E_{\sigma^n(\omega)} || F_{\sigma^n(\omega)}} S_{\omega}^n v} \ge (2K)^{-1}C_\lambda \lambda^n \norm{v}$.
  Setting $N_\beta = \log(2^{-1} K^{-1} C_\lambda)/\log(1 - \beta / \lambda)$, we observe that if $n > N_\beta$ and $S \in \mathcal{O}_{\epsilon_{n,\beta}}(L) \cap\mathcal{LY}(C_1, C_2, r, R)$ then \eqref{eq:perturbed_spaces_bound_fast} holds.
\end{proof}

\subsection{Stability of the fast spaces}\label{sec:fast_space}

In this section we will construct perturbed fast spaces $(E_{\omega}^S)_{\omega \in \Omega} \in \prod_{\omega \in \Omega} \mathcal{G}^d(X_\omega)$ when $S \in \mathcal{LY}(C_1, C_2, r, R) \cap \mathcal{O}_\epsilon(L)$ for some small $\epsilon$, and then show that these spaces approximate $(E_{\omega})_{\omega \in \Omega}$ in a Saks space sense.
We will construct these spaces as the fixed point of the forward graph transform of an iterate of $S$, which we will prove is a contraction mapping on a certain cone of subspaces.
Specifically, for $U = (U_\omega)_{\omega \in \Omega} \in \prod_{\omega \in \Omega} \LL(F_\omega, E_\omega)$ such that $\restr{\Pi_{E_{\sigma^n(\omega)} || F_{\sigma^n(\omega)}} S_{\omega}^n (\Id + U_\omega )}{E_\omega }$ is invertible for every $\omega \in \Omega$ we define $(S^n)^* U$ by
\begin{equation*}
  ((S^n)^* U)_\omega = (S_{\sigma^{-n}(\omega)}^n)^* U_{\sigma^{-n}(\omega)},
\end{equation*}
where the forward graph transform has domain $\LL(E_{\sigma^{-n}(\omega)},F_{\sigma^{-n}(\omega)})$ and codomain $\LL(E_\omega, F_\omega)$.
For each $\omega \in \Omega$ and $a > 0$ we define
\begin{equation*}
  \mathcal{C}_{\omega,a} = \{ U \in \mathcal{L}(E_\omega , F_\omega ) : \norm{U} \le a \},
\end{equation*}
and set the fast cone field to be $\mathcal{C}_{a} = \prod_{\omega \in \Omega}  \mathcal{C}_{\omega,a}$.
For each $a > 0$ the fast cone field $\mathcal{C}_{a}$ is a complete metric space with the metric inherited from $\prod_{\omega \in \Omega} \LL(E_\omega ,F_\omega )$.
We may now state our first main result for this section.

\begin{proposition}\label{prop:fast_contraction_mapping}
  There exists $a_0, \epsilon_0 > 0, n_0 \in \Z^+$ so that if $S \in \mathcal{LY}(C_1, C_2, r, R) \cap \mathcal{O}_{\epsilon_0}(L)$ then $(S^{n_0})^* \mathcal{C}_{a_0} \subseteq \mathcal{C}_{a_0}$.
  Moreover, there exists $c_0 \in [0,1)$ such that for every $U, V \in \mathcal{C}_{a_0}$ and $\omega \in \Omega$ we have
  \begin{equation}\label{eq:fast_contraction_mapping_0}
    \norm{(S_{\omega}^{n_0})^*(U_\omega) -(S_{\omega}^{n_0})^*(V_\omega)} \le c_0\norm{U_\omega - V_\omega}
  \end{equation}
   i.e. $(S_{\omega}^{n_0})^*$ is a contraction mapping on $\mathcal{C}_{a_0}$.
\end{proposition}

If $S$ satisfies the hypotheses of Proposition \ref{prop:fast_contraction_mapping} then we let
$U^S \in \mathcal{C}_{a_0}$ denote the unique fixed point of $(S^{n_0})^*$ and define $E_{\omega}^S = \Phi_{E_\omega  \oplus F_\omega }^{-1}(U_{\omega}^S) = (\Id + U_{\omega}^S)(E_\omega )$.
By Proposition \ref{prop:graph_chart} the sequence $(E_{\omega}^S)_{\omega \in \Omega} \in \prod_{\omega \in \Omega} \mathcal{N}(F_\omega)$ is fixed by $S^{n_0}$ i.e. $S_\omega^{n_0} E_{\omega}^S = E_{\sigma^{n_0}(\omega)}^S$ for every $\omega \in \Omega$.
Our second main result for this section confirms that if $\epsilon$ is sufficiently small then $(E_{\omega}^S)_{\omega \in \Omega}$ satisfy the estimate \eqref{eq:stability_cocycle_5} and that $(E_{\omega}^S)_{\omega \in \Omega}$ and $(E_{\omega})_{\omega \in \Omega}$ are close in a Saks space sense.

\begin{proposition}\label{prop:fast_space_props}
  We have
  \begin{equation}\label{eq:fast_space_props_1}
    \sup \left\{ \norm{\Pi_{E^S_\omega || F_\omega}}  : \omega \in \Omega, S \in \mathcal{LY}(C_1, C_2, r, R) \cap \mathcal{O}_{\epsilon_0}(L) \right\} < \infty.
  \end{equation}
  Moreover, for every $\beta \in (0, (\lambda - \mu)/2)$ and $\delta > 0$ there is $\epsilon_{\beta, \delta} \in (0, \epsilon_0)$ and $C_\beta > 0$ such that if $S \in \mathcal{LY}(C_1, C_2, r, R) \cap \mathcal{O}_{\epsilon_{\beta, \delta}}(L)$ then
  \begin{equation}\label{eq:fast_space_props_2}
    \sup_{\omega \in \Omega} \tnorm{\Pi_{E_{\omega}^S || F_\omega} - \Pi_{E_\omega || F_\omega}} \le \delta,
  \end{equation}
  and if, in addition, we have $\omega \in \Omega$, $v \in E_{\omega}^S$ and $n \in \Z^+$ then
  \begin{equation}\label{eq:fast_space_props_3}
    \norm{S_{\omega}^n v} \ge C_\beta^{-1} (\lambda - \beta)^n \norm{v}.
  \end{equation}
\end{proposition}

We will focus on proving Proposition \ref{prop:fast_contraction_mapping} first.

\begin{lemma}\label{lemma:graph_transform_inv}
  Fix $\beta \in (0, (\lambda - \mu)/2)$ and $a > 0$.
  There exists constants $M_{\beta}$ and, for each $n > M_\beta$, $\epsilon_{n,\beta,a} > 0$ such that if $S \in \mathcal{LY}(C_1, C_2, r, R) \cap \mathcal{O}_{\epsilon_{n,\beta,a}}(L)$, $\omega \in \Omega$ and $U \in \mathcal{C}_{\omega,a}$ then $\Pi_{E_{\sigma^{n}(\omega)} || F_{\sigma^{n}(\omega)}} S_\omega^{n}(\Id + U) : E_\omega \to E_{\sigma^{n}(\omega)}$ is invertible with
  \begin{equation}\label{eq:graph_transform_inv_0}
    \norm{\left(\restr{\Pi_{E_{\sigma^{n}(\omega)} || F_{\sigma^{n}(\omega)}} S_{\omega}^{n}(\Id + U)}{E_\omega}\right)^{-1}} \le (\lambda - \beta)^{-n}.
  \end{equation}
  \begin{proof}
    By Proposition \ref{prop:perturbed_spaces_bound} there exists $M_\beta$ and, for each $n > M_\beta$, $\epsilon_{n,\beta} > 0$ such that for all $\omega \in \Omega$ and $S \in \mathcal{LY}(C_1, C_2, r, R) \cap \mathcal{O}_{\epsilon_{n,\beta}}(L)$ we have
    \begin{equation}\label{eq:graph_transform_inv_1}
      \norm{\left(\restr{\Pi_{E_{\sigma^{n}(\omega)} || F_{\sigma^{n}(\omega)}}S_{\omega}^n}{E_\omega}\right)^{-1}} \le 2(\lambda - \beta)^{-n}.
    \end{equation}
    On the other hand, since $\Pi_{E_{\sigma^n(\omega)} || F_{\sigma^n(\omega)}}L_{\omega}^nU = 0$ and by Lemma \ref{lemma:fast_space_ss} we have for every $\eta > 0$ that
    \begin{equation}\begin{split}\label{eq:graph_transform_inv_2}
      \norm{\restr{\Pi_{E_{\sigma^n(\omega)} || F_{\sigma^n(\omega)}} S_{\omega}^n U}{E_\omega}} &= \norm{\restr{\Pi_{E_{\sigma^n(\omega)} || F_{\sigma^n(\omega)}} (S_{\omega}^n - L_\omega^n) U}{E_\omega}} \\
      &\le \eta \norm{\restr{(S_{\omega}^n - L_\omega^n) U}{E_\omega}} + C_\eta \tnorm{\restr{(S_{\omega}^n - L_\omega^n) U}{E_\omega}} \\
      &\le 2a \eta C_3 R^n  + a C_\eta \tnorm{S_{\omega}^n - L_\omega^n}.
    \end{split}\end{equation}
    By fixing $\eta = \frac{(\lambda - \beta)^{n}}{4a C_3 R^n }$ and applying Lemma \ref{lemma:tnorm_vanish} we find $\epsilon_{n,\beta, a} \in (0,\epsilon_{n,\beta})$ such that if $S \in \mathcal{LY}(C_1, C_2, r, R) \cap \mathcal{O}_{\epsilon_{n,\beta, a}}(L)$ then $a C_\eta \tnorm{S_{\omega}^n - L_\omega^n} \le (\lambda - \beta)^{n}/2$.
    Applying these bounds to \eqref{eq:graph_transform_inv_2} implies that if $S \in \mathcal{LY}(C_1, C_2, r, R) \cap \mathcal{O}_{\epsilon_{n,\beta, a}}(L)$, $\omega \in \Omega$ and $U \in \mathcal{C}_{\omega,a}$ then
    \begin{equation}\label{eq:graph_transform_inv_3}
      \norm{\restr{\Pi_{E_{\sigma^n(\omega)} || F_{\sigma^n(\omega)}} S_{\omega}^n U}{E_\omega}} \le (\lambda - \beta)^{n}.
    \end{equation}
    By combining \eqref{eq:graph_transform_inv_1} and \eqref{eq:graph_transform_inv_3} we confirm that $\restr{\Pi_{E_{\sigma^n(\omega)} || F_{\sigma^n(\omega)}} S_{\omega}^n(\Id + U)}{E_\omega}$ is invertible, and that the estimate \eqref{eq:graph_transform_inv_0} holds.
  \end{proof}
\end{lemma}

The main consequence of Lemma \ref{lemma:graph_transform_inv} is this: for each $a > 0$ and $n$ sufficiently large there exists $\epsilon_{a,n} > 0$ such that if $S \in \mathcal{LY}(C_1, C_2, r, R) \cap \mathcal{O}_{\epsilon_{a,n}}(L)$, $\omega \in \Omega$ and $U \in \mathcal{C}_{\omega,a}$ then $(S^n_\omega)^*U$ is well defined.

\begin{lemma}\label{lemma:perturbed_fast_cone_preservation}
  For sufficiently large $n$ there exists $a_n, \epsilon_n > 0$ such that if $S \in \mathcal{LY}(C_1, C_2, r, R) \cap \mathcal{O}_{\epsilon_{n}}(L)$ and $\omega \in \Omega$ then $(S_\omega^{n})^* \mathcal{C}_{\omega,a_n} \subseteq \mathcal{C}_{\sigma^{n}(\omega),a_n}$.
  \begin{proof}
    Fix $\beta \in (0, (\lambda - \mu)/2)$.
    For $a > 0$ let $M_\beta$ and $\epsilon_{n,\beta, a}$ denote the constants produced by Lemma \ref{lemma:graph_transform_inv}.
    By Lemma \ref{lemma:graph_transform_inv}, for every $S \in \mathcal{LY}(C_1, C_2, r, R) \cap \mathcal{O}_{\epsilon_{n,\beta, a}}(L)$, $\omega \in \Omega$ and $U \in \mathcal{C}_{\omega,a}$ we have $(S_\omega^n)^*U \in \LL(E_{\sigma^{n}(\omega)}, F_{\sigma^{n}(\omega)})$.
    By the estimate \eqref{eq:graph_transform_inv_0} and the definition of the forward graph transform,
    \begin{equation*}\begin{split}
      \norm{(S_{\omega}^n)^* U}
      &\le \norm{\Pi_{F_{\sigma^n(\omega)} || E_{\sigma^n(\omega)}} S_{\omega}^n(\Id + U)} (\lambda - \beta)^{-n}.
    \end{split}\end{equation*}
    Let $N_\beta$ and $\epsilon_{n,\beta}$ denote the constants produced by Proposition \ref{prop:perturbed_spaces_bound} and set $\epsilon_n = \min\{\epsilon_{n,\beta}, \epsilon_{n,\beta,a} \}$.
    Then for $n > \max\{N_\beta, M_\beta\}$, $S \in \mathcal{LY}(C_1, C_2, r, R) \cap \mathcal{O}_{\epsilon_{n}}(L)$ and $U \in \mathcal{C}_{\omega,a}$ we have
    \begin{equation*}
      \norm{\Pi_{F_{\sigma^n(\omega)} || E_{\sigma^n(\omega)}} S_{\omega}^n(\Id + U)} \le \Theta \left(\norm{S_{\omega}^n} + a\norm{\restr{S_{\omega}^n}{F_{\omega}}} \right) \le \Theta \left(C_3 R^n + a(\mu + \beta)^n\right),
    \end{equation*}
    and so
    \begin{equation}\label{eq:perturbed_fast_cone_preservation_2}
      \norm{(S_{\omega}^n)^* U} \le \Theta \left(C_3 \left(\frac{R}{\lambda - \beta}\right)^n + a\left(\frac{\mu + \beta}{\lambda - \beta}\right)^n \right).
    \end{equation}
    Since $\beta \in (0, (\lambda - \mu)/2)$, it follows from \eqref{eq:perturbed_fast_cone_preservation_2} that if $n$ is large enough so that $\Theta (\mu +\beta)^n < (\lambda - \beta)^n$ and we set
    \begin{equation*}
      a_n = \frac{\Theta C_3 R^n}{(\lambda - \beta)^n - \Theta(\mu + \beta)^n}
    \end{equation*}
    then $\norm{(S_{\omega}^n)^* U} \le a_n$ for every $U \in \mathcal{C}_{\omega,a_n}$ and $S \in \mathcal{LY}(C_1, C_2, r, R) \cap \mathcal{O}_{\epsilon_{n}}(L)$.
  \end{proof}
\end{lemma}

\begin{lemma}\label{lemma:fast_cone_contraction}
  Suppose that $n$ is large enough so that Lemma \ref{lemma:perturbed_fast_cone_preservation} may be applied, and let $a_n$ and $\epsilon_n$ denote the produced constants.
  For any such $n$ there exists $\epsilon' \in (0, \epsilon_n]$, $k \in \Z^+$ and $c \in [0,1)$ such that for every $\omega \in \Omega$, $S \in \mathcal{LY}(C_1, C_2, r, R) \cap\mathcal{O}_{\epsilon'}(L)$ and $U_1, U_2 \in \mathcal{C}_{\omega,a_n}$ we have
  \begin{equation*}
    \norm{(S_{\omega}^{nk})^*(U_1) - (S_{\omega}^{nk})^*(U_2)} \le c\norm{U_1 - U_2}.
  \end{equation*}
  \begin{proof}
    For brevity we set $\Xi_\omega = \Pi_{F_{\omega} || E_{\omega}}$ and $\Gamma_\omega = \Pi_{E_{\omega} || F_{\omega}}$.
    By the definition of $(S_{\omega}^{nk})^*$ we have
    \begin{equation}\begin{split}\label{eq:lemma:fast_cone_contraction_1}
      (S_{\omega}^{nk})^*(U_1) &- (S_{\omega}^{nk})^*(U_2)
       = \Xi_{\sigma^{nk}(\omega)} S_{\omega}^{nk}(U_1 - U_2) \left(\restr{\Gamma_{\sigma^{nk}(\omega)} S_{\omega}^{nk}(\Id + U_1)}{E_\omega}\right)^{-1} \\
      &+ ((S_{\omega}^{nk})^*U_2) \left(\Gamma_{\sigma^{nk}(\omega)} S_{\omega}^{nk}(U_2 - U_1) \right)\left(\restr{\Gamma_{\sigma^{nk}(\omega)} S_{\omega}^{nk}(\Id + U_1)}{E_\omega}\right)^{-1}.
    \end{split}\end{equation}
    We now fix $n$ large enough so that Lemma \ref{lemma:perturbed_fast_cone_preservation} may be applied.
    If $S \in \in \mathcal{LY}(C_1, C_2, r, R) \cap\mathcal{O}_{\epsilon_n}(L)$ then for every $k \in \Z^+$ we have $(S_{\omega}^{nk})^*\mathcal{C}_{\omega,a_n} \subseteq \mathcal{C}_{\sigma^{nk}(\omega),a_n}$, and so $((S_{\omega}^{nk})^*U_2) \le a_n$.
    Thus, \eqref{eq:lemma:fast_cone_contraction_1} becomes
    \begin{equation*}\begin{split}
      \norm{(S_{\omega}^{nk})^*(U_1) - (S_{\omega}^{nk})^*(U_2)}
      &\le (1+a_n) \Theta \norm{\restr{S_{\omega}^{nk}}{F_{\omega}}}  \norm{\left(\restr{\Gamma_{\sigma^{nk}(\omega)} S_{\omega}^{nk}(\Id + U_1)}{E_\omega}\right)^{-1}} \norm{U_1 - U_2}.
    \end{split}\end{equation*}
    Fix $\beta \in (0, (\lambda - \mu)/2)$.
    By Proposition \ref{prop:perturbed_spaces_bound} and Lemma \ref{lemma:graph_transform_inv} for every $k$ sufficiently large there exists $\epsilon_k > 0$ such that if $S \in \mathcal{LY}(C_1, C_2, r, R) \cap\mathcal{O}_{\epsilon_{nk, \beta, a_n}}(L)$ then
    \begin{equation}\label{eq:lemma:fast_cone_contraction_2}
      \norm{(S_{\omega}^{nk})^*(U_1) - (S_{\omega}^{nk})^*(U_2)} \le \Theta (1+a_n)\left(\frac{\mu + \beta}{\lambda - \beta}\right)^{nk} \norm{U_1 - U_2}.
    \end{equation}
    By taking $k$ large enough we may ensure that $c := \Theta (1+a_n)(\mu + \beta^{nk}/(\lambda - \beta)^{nk} < 1$,
    and so we obtain the required inequality from \eqref{eq:lemma:fast_cone_contraction_2} upon setting $\epsilon' = \epsilon_k$.
  \end{proof}
\end{lemma}

\begin{proof}[{The proof of Proposition \ref{prop:fast_contraction_mapping}}]
  Suppose that $n$ is large enough so Lemmas \ref{lemma:perturbed_fast_cone_preservation} and \ref{lemma:fast_cone_contraction} may be applied, and let $a_n$, $\epsilon'$, $k$, and $c$ denote the produced constants.
  Set $a_0 := a_n$, $n_0 := nk$, $\epsilon_0 := \epsilon'$ and $c_0 := c$.
  By Lemma \ref{lemma:perturbed_fast_cone_preservation} we have $(S^{n_0})^* \mathcal{C}_{a_0} \subseteq \mathcal{C}_{a_0}$ for every $S \in \mathcal{LY}(C_1, C_2, r, R) \cap\mathcal{O}_{\epsilon_0}(L)$.
  The estimate \eqref{eq:fast_contraction_mapping_0} is exactly the content of Lemma \ref{lemma:fast_cone_contraction}.
\end{proof}

We turn to the proof of Proposition \ref{prop:fast_space_props}. Recall that $U^S\in \mathcal{C}_{a_0}$ denotes the unique fixed point of $(S^{n_0})^*$, and that $E_{\omega}^S = \Phi_{E_\omega \oplus F_\omega}^{-1}(U^S_\omega) = (\Id + U^S_\omega)(E_\omega)$.

\begin{lemma}\label{lemma:fast_space_angle}
  We have
  \begin{equation*}
    \sup \left\{ \norm{\Pi_{E^S_\omega || F_\omega}}  : \omega \in \Omega, S \in \mathcal{LY}(C_1, C_2, r, R) \cap \mathcal{O}_{\epsilon_0}(L) \right\} < \infty.
  \end{equation*}
  \begin{proof}
    By Proposition \ref{prop:graph_chart} we have $\Pi_{E_{\omega}^S || F_\omega} = (\Id + U_{\omega}^S)\Pi_{E_{\omega} || F_\omega}$.
    Hence, as $U_{\omega}^S \in \mathcal{C}_{a_0}$, it follows that $\norm{\Pi_{E_\omega^S || F_\omega}} \le \norm{\Id + U_{\omega}^S} \norm{\Pi_{E_{\omega} || F_\omega}} \le (1 + a_0)\Theta$.
  \end{proof}
\end{lemma}

\begin{lemma}\label{lemma:tnorm_conv_fast_space_proj}
  For every $\delta > 0$ there exists $\epsilon_\delta \in (0, \epsilon_0]$ so that for every $S \in \mathcal{LY}(C_1, C_2, r, R) \cap \mathcal{O}_{\epsilon_\delta}(L)$ we have
  \begin{equation*}
    \sup_{\omega \in \Omega} \tnorm{\Pi_{E_{\omega}^S || F_\omega} - \Pi_{E_{\omega} || F_\omega}} \le \delta.
  \end{equation*}
  \begin{proof}
    By Proposition \ref{prop:graph_chart} we have $\Pi_{E_{\omega}^S || F_{\omega}} = (\Id + U_{\omega}^S) \Pi_{E_{\omega} || F_{\omega}}$, and so
    \begin{equation}\label{eq:tnorm_conv_fast_space_proj_0}
      \tnorm{\Pi_{E_{\omega}^S || F_{\omega}} - \Pi_{E_{\omega} || F_{\omega}}} \le \tnorm{U_{\omega}^S}\norm{\Pi_{E_{\omega} || F_{\omega}}} \le \Theta \tnorm{U_{\omega}^S}.
    \end{equation}
    For any $k \in \Z^+$ we have
    \begin{equation}\begin{split}\label{eq:tnorm_conv_fast_space_proj_1}
      \tnorm{U_{\omega}^S}
      &\le \tnorm{(S_{\sigma^{-n_0 k}(\omega)}^{n_0 k})^* 0}  + \norm{(S_{\sigma^{-n_0 k}(\omega)}^{n_0 k})^* U_{\sigma^{-n_0 k}(\omega)}^S - (S_{\sigma^{-n_0 k}(\omega)}^{n_0 k})^* 0}.
    \end{split}\end{equation}
    By Proposition \ref{prop:fast_contraction_mapping} we have
    \begin{equation}\label{eq:tnorm_conv_fast_space_proj_2}
      \norm{(S_{\sigma^{-n_0 k}(\omega)}^{n_0 k})^* U_{\sigma^{-n_0 k}(\omega)}^S - (S_{\sigma^{-n_0 k}(\omega)}^{n_0 k})^* 0} \le c_0^k a_0.
    \end{equation}
    Hence, by fixing $k$ large enough we may make the left hand side of \eqref{eq:tnorm_conv_fast_space_proj_2} strictly smaller than $\delta/(3\Theta)$.
    On the other hand, since $(L_{\sigma^{-n_0 k}(\omega)}^{n_0 k})^* 0 = 0$ and $\Pi_{F_{\omega} || E_{\omega}} L_{\sigma^{-n_0 k}(\omega)}^{n_0 k} \Pi_{E_{\sigma^{-n_0 k}(\omega)} || F_{\sigma^{-n_0 k}(\omega)}} = 0$, after a short calculation we find that
    \begin{equation*}\begin{split}
      (S_{\sigma^{-n_0 k}(\omega)}^{n_0 k})^* 0 &= (S_{\sigma^{-n_0 k}(\omega)}^{n_0 k})^* 0 - (L_{\sigma^{-n_0 k}(\omega)}^{n_0 k})^* 0 \\
      &= \Pi_{F_{\omega} || E_{\omega}} \left(S_{\sigma^{-n_0 k}(\omega)}^{n_0 k} - L_{\sigma^{-n_0 k}(\omega)}^{n_0 k}\right) \left(\restr{\Pi_{E_{\omega} || F_{\omega}} L_{\sigma^{-n_0 k}(\omega)}^{n_0 k} }{ E_{\sigma^{-n_0 k}(\omega)} } \right)^{-1}.
    \end{split}\end{equation*}
    Hence, by Lemma \ref{lemma:fast_space_ss}, Remark \ref{remark:equicont_proj} and \eqref{eq:h2_expansion} for every $\eta > 0$ there exists $C_\eta > 0$ such that
    \begin{equation*}\begin{split}
      \tnorm{(S_{\sigma^{-n_0 k}(\omega)}^{nk})^* 0} &= \tnorm{\Pi_{F_{\omega} || E_{\omega}} \left(S_{\sigma^{-n_0 k}(\omega)}^{nk} - L_{\sigma^{-n_0 k}(\omega)}^{nk}\right) \left(\restr{\Pi_{E_{\omega} || F_{\omega}} L_{\sigma^{-n_0 k}(\omega)}^{nk} }{E_{\sigma^{-n_0 k}(\omega)}}\right)^{-1}} \\
      &\le \eta \norm{\left(S_{\sigma^{-n_0 k}(\omega)}^{nk} - L_{\sigma^{-n_0 k}(\omega)}^{nk}\right) \left(\restr{\Pi_{E_{\omega} || F_{\omega}} L_{\sigma^{-n_0 k}(\omega)}^{nk} }{E_{\sigma^{-n_0 k}(\omega)}}\right)^{-1}}
      \\
      &+ C_\eta\tnorm{\left(S_{\sigma^{-n_0 k}(\omega)}^{nk} - L_{\sigma^{-n_0 k}(\omega)}^{nk}\right) \left(\restr{\Pi_{E_{\omega} || F_{\omega}} L_{\sigma^{-n_0 k}(\omega)}^{nk} }{E_{\sigma^{-n_0 k}(\omega)}}\right)^{-1}} \\
      &\le C_\lambda^{-1} \lambda^{-nk}\left(2\eta C_3 R^{nk} +C_\eta\tnorm{S_{\sigma^{-n_0 k}(\omega)}^{nk} - L_{\sigma^{-n_0 k}(\omega)}^{nk}} \right).
    \end{split}\end{equation*}
    Since $k$ is fixed there exists $\eta$ such that $2\eta C_3 C_\lambda^{-1} R^{nk} \lambda^{-nk} < \delta/(3\Theta)$.
    Then, by Lemma \ref{lemma:tnorm_vanish}, there exists $\epsilon_\delta \in (0, \epsilon_0]$ such that if $S \in \mathcal{LY}(C_1, C_2, r, R) \cap \mathcal{O}_{\epsilon_\delta}(L)$ then
    \begin{equation*}
      C_\eta C_\lambda^{-1} \lambda^{-nk}\tnorm{S_{\sigma^{-n_0k}(\omega)}^{nk} - L_{\sigma^{-n_0k}(\omega)}^{nk}} \le \frac{\delta}{3\Theta}.
    \end{equation*}
    Thus, if $S \in \mathcal{LY}(C_1, C_2, r, R) \cap \mathcal{O}_{\epsilon_\delta}(L)$ then $\tnorm{(S_{\sigma^{-n_0 k}(\omega)}^{nk})^* 0} \le 2\delta/ (3\Theta)$,
    and so $\tnorm{U_{\omega}^S} \le \delta/\Theta$ by \eqref{eq:tnorm_conv_fast_space_proj_1}.
    We obtain the required inequality upon recalling \eqref{eq:tnorm_conv_fast_space_proj_0}.
  \end{proof}
\end{lemma}

\begin{lemma}\label{lemma:lower_lyapunov}
  For each $\beta \in (0, (\lambda - \mu)/2)$ there exists $k_\beta \in \Z^+$ and $\epsilon_\beta > 0$ such that for every $S \in \mathcal{LY}(C_1, C_2, r, R) \cap \mathcal{O}_{\epsilon_\beta}(L)$, $\omega \in \Omega$, $U \in \mathcal{C}_{\omega,a_0}$ and $v \in E_\omega^S$ we have
  \begin{equation*}
    \norm{S_{\omega}^{k_\beta n_0} v} \ge (\lambda - \beta)^{k_\beta n_0} \norm{v}.
  \end{equation*}
  \begin{proof}
    Since $\wnorm{\cdot} \le \norm{\cdot}$ we have
    \begin{equation}\label{eq:lower_lyapunov_1}
      \norm{S_{\omega}^{k n_0} v} \ge \wnorm{S_{\omega}^{k n_0} v} \ge \wnorm{L_\omega^{k n_0} \Pi_{E_{\omega} || F_{\omega}} v} - \wnorm{\left(S_{\omega}^{k n_0} - L_\omega^{k n_0}\right) \Pi_{E_{\omega} || F_{\omega}} v} - \norm{S_{\omega}^{k n_0} \Pi_{F_{\omega} || E_{\omega}} v}.
    \end{equation}
    Using Lemma \ref{lemma:fast_space_eccentricity} and \eqref{eq:h2_expansion} we find that
    \begin{equation*}
      \wnorm{L_\omega^{k n_0} \Pi_{E_{\omega} || F_{\omega}} v} \ge K^{-1} C_\lambda \lambda^{kn_0} \norm{\Pi_{E_{\omega} || F_{\omega}} v}.
    \end{equation*}
    Since $v \in E^S_\omega = (\Id + U_\omega)(E_{\omega})$ and $U_\omega \in \mathcal{C}_{\omega, \epsilon_0}$ we have $\norm{\Pi_{F_{\omega} || E_{\omega}} v} \le a_0 \norm{\Pi_{E_{\omega} || F_{\omega}} v}$ and so $(1+a_0)^{-1} \norm{ v} \le \norm{\Pi_{E_{\omega} || F_{\omega}} v}$.
    Hence \eqref{eq:lower_lyapunov_1} becomes
    \begin{equation}\begin{split}\label{eq:lower_lyapunov_2}
      \norm{S_{\omega}^{k n_0} v}
      &\ge \left((1+a_0)^{-1}K^{-1} C_\lambda \lambda^{kn_0}  - \Theta \tnorm{S_{\omega}^{k n_0} - L_\omega^{k n_0}} -  \Theta \norm{\restr{S_{\omega}^{k n_0}}{F_{\omega}}}\right)\norm{v}.
    \end{split}\end{equation}
    Let $k := k_\beta$ be sufficiently large so that
    \begin{equation}\label{eq:lower_lyapunov_3}
      (1+a_0)^{-1}K^{-1} C_\lambda \lambda^{k_\beta n_0} \ge 2(\lambda - \beta)^{k_\beta n_0},
    \end{equation}
    and so that Proposition \ref{prop:perturbed_spaces_bound} and Lemma \ref{lemma:tnorm_vanish} may be applied with $n = k_\beta n_0$ to produce $\epsilon_\beta$ so that if $S \in \mathcal{LY}(C_1, C_2, r, R) \cap \mathcal{O}_{\epsilon_{\beta}}(L)$ and $\omega \in \Omega$ then
    \begin{equation}\label{eq:lower_lyapunov_4}
      \norm{\restr{S_{\omega}^{k_\beta n_0}}{F_{\omega}}} \le (\mu + \beta)^{k_\beta n_0} \le \frac{(\lambda - \beta)^{k_\beta n_0}}{2\Theta},
    \end{equation}
    and
    \begin{equation}\label{eq:lower_lyapunov_5}
      \tnorm{L_\omega^{k_\beta n_0 } - S_{\omega}^{k_\beta n_0}} \le \frac{(\lambda - \beta)^{ k_\beta n_0}}{2\Theta}.
    \end{equation}
    Applying \eqref{eq:lower_lyapunov_3}, \eqref{eq:lower_lyapunov_4} and \eqref{eq:lower_lyapunov_5} to \eqref{eq:lower_lyapunov_2} yields the required inequality.
  \end{proof}
\end{lemma}

\begin{proof}[{The proof of Proposition \ref{prop:fast_space_props}}]
  The estimates \eqref{eq:fast_space_props_1} and \eqref{eq:fast_space_props_2} are proven in Lemmas \ref{lemma:fast_space_angle} and \ref{lemma:tnorm_conv_fast_space_proj}, respectively.
  Thus to finish the proof it suffices to demonstrate \eqref{eq:fast_space_props_3}.

  For $\beta \in (0, (\lambda - \mu)/2)$ let $k_\beta$ and $\epsilon_\beta$ be the constants produced by Lemma \ref{lemma:lower_lyapunov}. For $n \in \Z^+$ write $n = m n_0 k_\beta + j$ where $m \in \N$ and $j \in \{0, \dots, n_0 k_\beta - 1\}$.
  For any $\omega \in \Omega$, $S \in \mathcal{LY}(C_1, C_2, r, R) \cap \mathcal{O}_{\epsilon_{\beta}}(L)$ and $v \in E_{\omega }^S$ we have
  \begin{equation*}
    S_{\omega}^{mn_0 k_\beta} v \in E_{\sigma^{mn_0 k_\beta}(\omega)}^S = \left(\Id + U_{\sigma^{mn_0 k_\beta}(\omega)}^S\right)E_{\sigma^{mn_0 k_\beta}(\omega)}.
  \end{equation*}
  Hence, as $U_{\sigma^{mn_0 k_\beta}(\omega)}^S \in \mathcal{C}_{\sigma^{mn_0 k_\beta}(\omega), a_0}$, by Lemma \ref{lemma:lower_lyapunov} we have
  \begin{equation*}
    \norm{S_{\omega}^{(m+1)n_0 k_\beta} v} = \norm{S_{\sigma^{mn_0 k_\beta}(\omega)}^{n_0 k_\beta}S_{\omega}^{mn_0 k_\beta} v} \ge (\lambda - \beta)^{n_0 k_\beta}\norm{S_{\omega}^{mn_0 k_\beta} v}.
  \end{equation*}
  By repeating this argument we deduce that $\norm{S_{\omega}^{(m+1) n_0 k_\beta} v} \ge (\lambda - \beta)^{(m+1)n_0 k_\beta} \norm{v}$. Therefore, as $R > \lambda - \beta$,
  \begin{equation*}\begin{split}
    \norm{S_{\omega}^{n } v} \ge \norm{S_{\sigma^{m n_0 k_\beta+j}(\omega)}^{n_0 k_\beta - j }}^{-1}\norm{S_{\omega}^{(m+1) n_0 k_\beta} v} &\ge C_3^{-1} \left(\frac{\lambda - \beta}{R}\right)^{n_0 k_\beta -j }  (\lambda - \beta)^{mn_0 k_\beta + j} \norm{v}\\
    &\ge C_3^{-1} \left(\frac{\lambda - \beta}{R}\right)^{n_0 k_\beta}  (\lambda - \beta)^{n} \norm{v},
  \end{split}\end{equation*}
  and so we obtain the required claim by setting $C_\beta = C_3 \left(\frac{R}{\lambda - \beta}\right)^{n_0 k_\beta}$.
\end{proof}

\subsection{Stability of the slow spaces}\label{sec:slow_space}

In this section we will construct and characterise the perturbed slow spaces for $S \in \mathcal{LY}(C_1, C_2, r, R) \cap \mathcal{O}_{\epsilon}(L)$ when $\epsilon$ is sufficiently small.
These perturbed slow spaces will be the fixed point of a backwards graph transform associated to $S$, although our approach is slightly different to that of the previous section since we may capitalise on the existence of fast spaces for $S$. Once constructed, we show that the slow spaces are stable in the Grassmannian, and verify the estimate \eqref{eq:stability_cocycle_4}.
Let $n_0$ and $\epsilon_0$ be the constants produced by Proposition \ref{prop:fast_contraction_mapping}, and suppose that $S \in \mathcal{LY}(C_1, C_2, r, R) \cap \mathcal{O}_{\epsilon_0}(L)$.
For $V \in \LL(F_\omega, E_{\omega}^S)$ recall that $(S_{\sigma^{-n_0}(\omega)}^{n_0})_* V$ is well-defined if $(\Pi_{E_{\omega}^S || F_{\omega}} - V\Pi_{F_{\omega} || E_{\omega}^S})S_{\sigma^{-n_0}(\omega)}^{n_0} : E_{\sigma^{-n_0}(\omega)}^S \to E_{\omega}^S$ is invertible.
Since $S_{\sigma^{-n_0}(\omega)}^{n_0}E_{\sigma^{-n_0 }(\omega)}^S = E_{\omega}^S$ it follows that
\begin{equation*}
  \restr{(\Pi_{E_{\omega}^S || F_{\omega}} - V\Pi_{F_{\omega} || E_{\omega}^S})S_{\sigma^{-n_0}(\omega)}^{n_0}}{E_{\sigma^{-n_0}(\omega)}^S} = \restr{S_{\sigma^{-n_0}(\omega)}^{n_0}}{E_{\sigma^{-n_0}(\omega)}^S},
\end{equation*}
which is always invertible.
Hence the map $(S_{\sigma^{-n_0 }(\omega)}^{n_0})_* : \LL(F_{\omega}, E_{\omega}^S) \to \LL(F_{\sigma^{-n_0 }(\omega)}, E_{\sigma^{-n_0 }(\omega)}^S)$ is well defined and satisfies
\begin{equation*}
(S_{\sigma^{-n_0 }(\omega)}^{n_0 })_* V = \left(\restr{S_{\sigma^{-n_0 }(\omega)}^{n_0}}{E_{\sigma^{-n_0}(\omega)}^S}\right)^{-1} \left(V \Pi_{F_\omega || E_{\omega}^S} - \Pi_{E_{\omega}^S || F_\omega} \right)S_{\sigma^{-n_0 }(\omega)}^{n_0 }.
\end{equation*}
Finally, let $S^{n_0}_* : \prod_{\omega \in \Omega} \LL(F_{\omega}, E_{\omega}^S) \to \prod_{\omega \in \Omega} \LL(F_{\omega}, E_{\omega}^S)$ be defined by
\begin{equation*}
  (S^{n_0}_* V)_\omega =  (S^{n_0}_\omega)_*V_{\sigma^{n_0}(\omega)}.
\end{equation*}

\begin{proposition}\label{prop:slow_contraction_mapping}
  There exists $k \in \Z^+$, $c \in[0,1)$ and $\epsilon_1 \in [0, \epsilon_0)$ such that for any $S \in \mathcal{LY}(C_1, C_2, r, R) \cap \mathcal{O}_{\epsilon_1}(L)$, $\omega \in \Omega$ and $V_1,V_2 \in \LL(F_\omega, E_{\omega}^S)$ we have
  \begin{equation*}
    \norm{(S_{\sigma^{-n_0 k}(\omega)}^{n_0 k})_*(V_1) - (S_{\sigma^{-n_0 k}(\omega)}^{n_0 k})_*(V_2)} \le c \norm{V_1 - V_2}.
  \end{equation*}
  Hence $S^{n_0 k}_*$ is a contraction mapping on $\prod_{\omega \in \Omega} \LL(F_{\omega}, E_{\omega}^S)$.
\end{proposition}

If $S$ satisfies the hypotheses of Proposition \ref{prop:slow_contraction_mapping} then we let $V^S \in \prod_{\omega \in \Omega} \LL(F_\omega, E_{\omega}^S)$ denote the unique fixed point of $S^{n_0k}_*$, and set $F_{\omega}^S = \Phi_{F_{\omega} \oplus E_{\omega}^S}^{-1}(V_{\omega}^S) = (\Id + V_{\omega}^S)(F_\omega)$.
Note that, since $S^{n_0}_*$ preserves $\prod_{\omega \in \Omega} \LL(F_\omega, E_{\omega}^S)$, we must have $S^{n_0}_* V^S = V^S$.
Moreover, by the definition of the graph representation we have $F^S_\omega \in \mathcal{N}(E_\omega^S)$, so that $X = F^S_\omega \oplus E_\omega^S$.
Our second main result for this section confirms that the spaces $(F_{\omega}^S)_{\omega \in \Omega}$ are equivariant slow spaces for $S^{n_0k}$, and that $(F_{\omega}^S)_{\omega \in \Omega}$ approximates $(F_{\omega})_{\omega \in \Omega}$ in the Grassmannian.

\begin{proposition}\label{prop:slow_space_props}
  We have
  \begin{equation}\label{eq:slow_space_props_1}
    \sup \left\{ \norm{\Pi_{E^S_\omega || F_\omega^S}}  : \omega \in \Omega, S \in \mathcal{LY}(C_1, C_2, r, R) \cap \mathcal{O}_{\epsilon_1}(L) \right\} < \infty,
  \end{equation}
  and $S_{\omega}^{n_0}F_{\omega}^S \subseteq F^{S}_{\sigma^{n_0}(\omega)}$ for every $\omega \in \Omega$.
  Moreover, for every $\beta \in (0, (\lambda - \mu)/2)$ and $\delta > 0$ there is $\epsilon_{\beta, \delta} \in (0, \epsilon_1]$ and $C_\beta > 0$ such that if $S \in \mathcal{LY}(C_1, C_2, r, R) \cap \mathcal{O}_{\epsilon_{\beta, \delta}}(L)$ then
  \begin{equation}\label{eq:slow_space_props_2}
    \sup_{\omega \in \Omega} d_H( F_{\omega}^S, F_\omega) \le \delta,
  \end{equation}
  \begin{equation}\label{eq:slow_space_props_3}
    \sup_{\omega \in \Omega} \tnorm{\Pi_{F_{\omega}^S || E_{\omega}^S } - \Pi_{ F_{\omega} || E_{\omega}}} \le \delta,
  \end{equation}
  and if, in addition, we have $\omega \in \Omega$ and $n \in \Z^+$ then
  \begin{equation}\label{eq:slow_space_props_4}
    \norm{\restr{S_\omega^n}{F_{\omega}^S}} \le C_\beta (\mu + \beta)^n.
  \end{equation}
\end{proposition}

To fix some notation, we let
\begin{equation}\label{eq:norm_bound}
  M := \sup \left\{ \norm{\Pi_{E^S_\omega || F_\omega}}  : \omega \in \Omega, S \in \mathcal{LY}(C_1, C_2, r, R) \cap \mathcal{O}_{\epsilon_0}(L) \right\},
\end{equation}
which is finite by Proposition \ref{prop:fast_space_props}.

\begin{proof}[{The proof of Proposition \ref{prop:slow_contraction_mapping}}]
  By Proposition \ref{prop:graph_tranform} we have for every $k \in \Z^+$ that
  \begin{equation*}
    (S_{\sigma^{-n_0 k}(\omega)}^{n_0 k})_*(V_1) - (S_{\sigma^{-n_0 k}(\omega)}^{n_0 k})_*(V_2) = \left(\restr{S_{\sigma^{-n_0 k}(\omega)}^{n_0 k}}{E_{\sigma^{-n_0 k}(\omega)}^S}\right)^{-1} (V_1 - V_2) \Pi_{F_\omega || E_{\omega}^S}S_{\sigma^{-n_0 k}(\omega)}^{n_0 k},
  \end{equation*}
  and so
  \begin{equation}\begin{split}\label{eq:slow_contraction_mapping_1}
    \bigg\lVert(S_{\sigma^{-n_0 k}(\omega)}^{n_0 k})_*(V_1) &- (S_{\sigma^{-n_0 k}(\omega)}^{n_0 k})_*(V_1)\bigg\rVert \\
    &\le \norm{ \left(\restr{S_{\sigma^{-n_0 k}(\omega)}^{n_0 k}}{E_{\sigma^{-n_0 k}(\omega)}^S}\right)^{-1} } \norm{\Pi_{F_\omega || E_{\omega}^S}} \norm{\restr{S_{\sigma^{-n_0 k}(\omega)}^{n_0 k}}{F_{\sigma^{-n_0 k}(\omega)}}} \norm{V_1 - V_2}.
  \end{split}\end{equation}
  Let $\beta \in (0, (\lambda - \mu)/2)$. By Proposition \ref{prop:fast_space_props} there exists $\epsilon_\beta \in (0, \epsilon_0)$ and $C_\beta > 0$ so that for every $S \in \mathcal{LY}(C_1, C_2, r, R) \cap \mathcal{O}_{\epsilon_{\beta}}(L)$ and $k \in \Z^+$ we have
  \begin{equation}\label{eq:slow_contraction_mapping_2}
    \norm{ \left(\restr{S_{\sigma^{-n_0 k}(\omega)}^{n_0 k}}{E_{\sigma^{-n_0 k}(\omega)}^S}\right)^{-1} } \le C_\beta (\lambda - \beta)^{-n_0 k}.
  \end{equation}
  Fix $k$ large enough so that $c := C_\beta (M+1) (\mu + \beta)^{n_0 k}/(\lambda - \beta)^{n_0 k} < 1$.
  By Proposition \ref{prop:perturbed_spaces_bound} there exists $\epsilon_{\beta, k} \in (0, \epsilon_\beta)$ so that for $S \in \mathcal{LY}(C_1, C_2, r, R) \cap \mathcal{O}_{\epsilon_{\beta,k}}(L)$ we have
  \begin{equation}\label{eq:slow_contraction_mapping_3}
    \norm{\restr{ S_{\sigma^{-n_0 k}(\omega)}^{n_0 k}}{F_{\sigma^{-n_0 k}(\omega)}}} \le (\mu + \beta)^{n_0 k}.
  \end{equation}
  Set $\epsilon_1 := \epsilon_{\beta, k}$. We obtain the required statement by applying \eqref{eq:norm_bound}, \eqref{eq:slow_contraction_mapping_2}, and \eqref{eq:slow_contraction_mapping_3} to \eqref{eq:slow_contraction_mapping_1}, and then recalling that $c \in [0,1)$.
\end{proof}

The proof of Proposition \ref{prop:slow_space_props} is broken into a number of lemmas.

\begin{lemma}\label{lemma:slow_equivariance}
  For every $\omega \in \Omega$ and $S \in \mathcal{LY}(C_1, C_2, r, R) \cap \mathcal{O}_{\epsilon_1}(L)$ we have $S_{\omega}^{n_0} F_{\omega}^S \subseteq F_{\sigma^{n_0}(\omega)}^S$.
  \begin{proof}
    Since $(S_{\omega}^{n_0})_* V_{\sigma^{n_0}(\omega)}^S = V_{\omega}^S$ we have
    \begin{equation*}\begin{split}
      S_{\omega}^{n_0} F_{\omega}^S &= S_{\omega}^{n_0 } (\Id + V_{\omega}^S)(F_\omega)\\
      &= S_{\omega}^{n_0 } (\Id + (S_{\omega}^{n_0})_* V_{\sigma^{n_0}(\omega)}^S)(F_{\omega})\\
      &= S_{\omega}^{n_0 } \left(\Id + \left(\restr{S_{\omega}^{n_0 }}{E_{\omega}^S}\right)^{-1} \left(V_{\sigma^{n_0}(\omega)}^S \Pi_{F_{\sigma^{n_0 }(\omega)} || E_{\sigma^{n_0}(\omega)}^S} - \Pi_{E_{\sigma^{n_0 }(\omega)}^S || F_{\sigma^{n_0 }(\omega)}} \right)S_{\omega}^{n_0 }\right)(F_{\omega})\\
      &= \left(\left(\Id - \Pi_{E_{\sigma^{n_0 }(\omega)}^S || F_{\sigma^{n_0 }(\omega)}}\right) + V_{\sigma^{n_0 }(\omega)}^S \Pi_{F_{\sigma^{n_0 }(\omega)} || E_{\sigma^{n_0 }(\omega)}^S} \right)S_{\omega}^{n_0}(F_{\omega})\\
      &= (\Id  + V_{\sigma^{n_0 }(\omega)}^S) \Pi_{F_{\sigma^{n_0 }(\omega)} || E_{\sigma^{n_0 }(\omega)}^S} S_{\omega}^{n_0}(F_{\omega})\\
      &\subseteq (\Id  + V_{\sigma^{n_0 }(\omega)}^S)(F_{\sigma^{n_0}(\omega)})= F_{\sigma^{n_0 }(\omega)}^S.
    \end{split}\end{equation*}
  \end{proof}
\end{lemma}

\begin{lemma}\label{lemma:slow_space_angle_bound}
  We have
  \begin{equation}
    \sup \left\{ \norm{\Pi_{E^S_\omega || F_\omega^S}}  : \omega \in \Omega, S \in \mathcal{LY}(C_1, C_2, r, R) \cap \mathcal{O}_{\epsilon_1}(L) \right\} < \infty.
  \end{equation}
  \begin{proof}
    By Proposition \ref{prop:slow_contraction_mapping} we have for every $\omega \in \Omega$ that
    \begin{equation*}
      \norm{V_{\omega}^S} \le \norm{(S_{\omega}^{n_0 k})_* V_{\sigma^{n_0 k}(\omega)}^S - (S_{\omega}^{n_0 k})_*(0)} + \norm{(S_{\omega}^{n_0 k})_*(0)} \le c\norm{V_{\sigma^{n_0 k}(\omega)}^S} + \norm{(S_{\omega}^{n_0 k})_*(0)},
    \end{equation*}
    from which it follows that
    \begin{equation*}
      \sup_{\omega \in \Omega} \norm{V_{\omega}^S} \le (1-c)^{-1} \sup_{\omega \in \Omega} \norm{(S_{\omega}^{n_0 k})_*(0)}.
    \end{equation*}
    Since
    \begin{equation*}
      (S_{\omega}^{n_0 k})_*(0) = -\left(\restr{S_{\omega}^{n_0 k}}{E_{\omega}^S}\right)^{-1} \Pi_{E_{\sigma^{n_0 k}(\omega)}^S || F_{\sigma^{n_0 k}(\omega)}} S_{\omega}^{n_0 k},
    \end{equation*}
    the bounds used in the proof of Proposition \ref{prop:slow_contraction_mapping} imply that
    \begin{equation*}
      \norm{(S_{\omega}^{n_0 k})_*(0)} \le \norm{\left(\restr{S_{\omega}^{n_0 k}}{E_{\omega}^S}\right)^{-1}} \norm{\Pi_{E_{\sigma^{n_0 k}(\omega)}^S || F_{\sigma^{n_0 k}(\omega)}}} \norm{\restr{S_{\omega}^{n_0 k}}{F_\omega}} < c.
    \end{equation*}
    Hence for every $S \in \mathcal{LY}(C_1, C_2, r, R) \cap \mathcal{O}_{\epsilon_1}(L)$ we have
    \begin{equation}\label{eq:slow_space_angle_bound_1}
      \sup_{\omega \in \Omega} \norm{V_{\omega}^S} \le \frac{c}{1-c}.
    \end{equation}
    By Proposition \ref{prop:graph_chart} we have $\Pi_{F_{\omega}^S || E_{\omega}^S } = (\Id + V_{\omega}^S) \Pi_{F_\omega || E_{\omega}^S}$ and so $\norm{\Pi_{F_{\omega}^S || E_{\omega}^S }} \le (1  + \norm{V_{\omega}^S}) \norm{ \Pi_{F_\omega || E_{\omega}^S}}$.
    Recalling the bound \eqref{eq:norm_bound}, we get
    \begin{equation*}
      \sup \left\{ \norm{\Pi_{E^S_\omega || F_\omega^S}}  : \omega \in \Omega, S \in \mathcal{LY}(C_1, C_2, r, R) \cap \mathcal{O}_{\epsilon_1}(L) \right\} \le \frac{M+1}{1-c}.
    \end{equation*}
  \end{proof}
\end{lemma}

\begin{lemma}\label{lemma:slow_space_chart_stability}
  For every $\delta > 0$ there exists $\epsilon_\delta \in (0, \epsilon_1]$ such that if $S \in \mathcal{LY}(C_1, C_2, r, R) \cap \mathcal{O}_{\epsilon_\delta}(L)$ and $\omega \in \Omega$ then $\norm{V_{\omega}^S} \le \delta$.
  \begin{proof}
    For every $m \in \Z^+$ we have
    \begin{equation}\label{eq:slow_space_chart_stability_1}
      \begin{split}
      \norm{V_{\omega}^S} &=  \norm{(S_{\omega}^{n_0km})_* V_{\sigma^{n_0 km}(\omega)}^S}\\
      &= \norm{\left(\restr{S_{\omega}^{n_0 k m}}{E_{\omega}^S}\right)^{-1} \left(V_{\sigma^{n_0 km}(\omega)}^S \Pi_{F_{\sigma^{n_0 km}(\omega)} || E_{\sigma^{n_0 km}(\omega)}^S} - \Pi_{E_{\sigma^{n_0 km}(\omega)}^S || F_{\sigma^{n_0 km}(\omega)}} \right)\restr{S_{\omega}^{n_0 km}}{F_{\omega}} }.
    \end{split}\end{equation}
    Fix $\beta \in (0 ,(\lambda - \mu)/2)$. By Proposition \ref{prop:fast_space_props} there exists $\epsilon_\beta \in (0, \epsilon_1]$ and $C_\beta$ so that if $S \in \mathcal{LY}(C_1, C_2, r, R) \cap \mathcal{O}_{\epsilon_\beta}(L)$ then for every $m \in \Z^+$ we have
    \begin{equation}\label{eq:slow_space_chart_stability_2}
      \norm{\left(\restr{S_{\omega}^{n_0 k m}}{E_{\omega}^S}\right)^{-1}} \le C_\beta(\lambda - \beta)^{-n_0 k m}.
    \end{equation}
    Let $N_\beta$ be the constant produced by Proposition \ref{prop:perturbed_spaces_bound} and fix $m > N_\beta/(n_0 k)$ large enough so that
    \begin{equation}\label{eq:slow_space_chart_stability_3}
      C_\beta \left( \frac{c(1+M)}{1-c} + M \right) \left( \frac{\mu + \beta}{\lambda - \beta}\right)^{n_0 k m} \le \delta.
    \end{equation}
    By Proposition \ref{prop:perturbed_spaces_bound} there is $\epsilon_{\delta} \in (0, \epsilon_\beta]$ such that if $S \in \mathcal{LY}(C_1, C_2, r, R) \cap \mathcal{O}_{\epsilon_\delta}(L)$ then
    \begin{equation}\label{eq:slow_space_chart_stability_4}
      \norm{\restr{S_{\omega}^{n_0 km}}{F_{\omega}}} \le (\mu + \beta)^{m n_0 k}.
    \end{equation}
    Recalling \eqref{eq:slow_space_angle_bound_1} from the proof of Lemma \ref{lemma:slow_space_angle_bound}, and then applying \eqref{eq:slow_space_chart_stability_2}, \eqref{eq:slow_space_chart_stability_3} and \eqref{eq:slow_space_chart_stability_3} to \eqref{eq:slow_space_chart_stability_1} yields the required inequality.
  \end{proof}
\end{lemma}

\begin{lemma}\label{lemma:slow_space_stability}
  For every $\delta > 0$ there exists $\epsilon_\delta \in (0, \epsilon_1]$ such that if $S \in \mathcal{LY}(C_1, C_2, r, R) \cap \mathcal{O}_{\epsilon_\delta}(L)$ and $\omega \in \Omega$ then $d_H( F_{\omega}^S, F_\omega) \le \delta$.
  \begin{proof}
    Since $\Phi_{F_\omega \oplus E_{\omega}^S}(F_{\omega}) = 0$, by Lemma \ref{lemma:graph_rep_continuity} and \eqref{eq:norm_bound} we have
    \begin{equation*}
      d_H( F_{\omega}^S, F_{\omega}) \le 2 \norm{\Pi_{F_\omega || E_\omega^S}}\norm{V_\omega^S - \Phi_{F_\omega \oplus E_{\omega}^S}(F_{\omega})} \le 2 (M + 1) \norm{V_{\omega}^S},
    \end{equation*}
    and so the required inequality follows immediately from Lemma \ref{lemma:slow_space_chart_stability}.
  \end{proof}
\end{lemma}

\begin{lemma}\label{lemma:tnorm_conv_perturbed_proj}
  For every $\delta > 0$ there exists $\epsilon_{\delta} \in (0, \epsilon_1]$ such that for all $S \in \mathcal{LY}(C_1, C_2, r, R) \cap \mathcal{O}_{\epsilon_\delta}(L)$ one has
  \begin{equation*}
    \sup_{\omega \in \Omega} \tnorm{\Pi_{F_{\omega}^S || E_{\omega}^S } - \Pi_{ F_\omega || E_{\omega}}} \le \delta.
  \end{equation*}
  \begin{proof}
    By the triangle inequality we get
    \begin{equation}\label{eq:tnorm_conv_perturbed_proj_1}
      \tnorm{\Pi_{F_{\omega}^S || E_{\omega}^S } - \Pi_{ F_{\omega} || E_{\omega}}}
      \le \norm{\Pi_{F_{\omega}^S || E_{\omega}^S } - \Pi_{ F_{\omega} || E_{\omega}^S}} + \tnorm{\Pi_{E_{\omega}^S || F_{\omega} } - \Pi_{E_{\omega} || F_{\omega}}}.
    \end{equation}
    By Proposition \ref{prop:graph_chart} and \eqref{eq:norm_bound} we have
    \begin{equation*}
      \norm{\Pi_{F_{\omega}^S || E_{\omega}^S } - \Pi_{ F_{\omega} || E_{\omega}^S}} \le \norm{V_\omega^S} \norm{\Pi_{ F_{\omega} || E_{\omega}^S}} \le (M + 1) \norm{V_\omega^S}.
    \end{equation*}
    Hence by Lemma \ref{lemma:slow_space_chart_stability} there exists $\epsilon_{\delta, 1} \in (0,\epsilon_1]$ such that if $S \in \mathcal{LY}(C_1, C_2, r, R) \cap \mathcal{O}_{\epsilon_{\delta,1}}(L)$ then
    \begin{equation}\label{eq:tnorm_conv_perturbed_proj_2}
      \sup_{\omega \in \Omega} \norm{\Pi_{F_{\omega}^S || E_{\omega}^S } - \Pi_{ F_{\omega} || E_{\omega}^S}} \le \delta / 2.
    \end{equation}
    On the other hand, by Proposition \ref{prop:fast_space_props} there exists $\epsilon_{\delta, 2} \in (0, \epsilon_0]$ such that if $S \in \mathcal{LY}(C_1, C_2, r, R) \cap \mathcal{O}_{\epsilon_{\delta,2}}(L)$ then
    \begin{equation}\label{eq:tnorm_conv_perturbed_proj_3}
      \sup_{\omega \in \Omega} \tnorm{\Pi_{E_{\omega}^S || F_{\omega} } - \Pi_{E_{\omega} || F_{\omega}}} \le \delta / 2.
    \end{equation}
    Upon setting $\epsilon_{\delta} = \min \{\epsilon_{\delta, 1},  \epsilon_{\delta, 2}\}$ we may conclude by applying \eqref{eq:tnorm_conv_perturbed_proj_2} and \eqref{eq:tnorm_conv_perturbed_proj_3} to \eqref{eq:tnorm_conv_perturbed_proj_1}.
  \end{proof}
\end{lemma}

\begin{lemma}\label{lemma:slow_decay}
  For $\beta \in (0, (\lambda - \mu)/2)$ there exists $\epsilon_\beta \in (0, \epsilon_1]$ and $m \in \Z^+$ such that if $S \in \mathcal{LY}(C_1, C_2, r, R) \cap \mathcal{O}_{\epsilon_\delta}(L)$ and $\omega \in \Omega$ then
  \begin{equation*}
    \norm{\restr{S_{\omega}^{n_0 k m}}{F_{\omega}^S}} \le (\mu + \beta)^{n_0 k m}.
  \end{equation*}
  \begin{proof}
    We have
    \begin{equation}\label{eq:slow_decay_1}
      \norm{\restr{S_{\omega}^{n_0 k m}}{F_{\omega}^S} } \le \norm{ \restr{S_{\omega}^{n_0 k m}}{F_{\omega}} } + \norm{S_{\omega}^{n_0 k m}} d_H(F_{\omega}, F_{\omega}^S) \le \norm{\restr{S_{\omega}^{n_0 k m}}{F_{\omega}}} + C_3 R^{n_0 k m} d_H(F_{\omega}, F_{\omega}^S).
    \end{equation}
    By Proposition \ref{prop:perturbed_spaces_bound} there exists $\epsilon_{\beta,1} \in (0, \epsilon_1)$ and $m \in \Z^+$ such that if $S \in \mathcal{LY}(C_1, C_2, r, R) \cap \mathcal{O}_{\epsilon_{\beta,1}}(L)$ and $\omega \in \Omega$ then
    \begin{equation}\label{eq:slow_decay_2}
      \norm{\restr{S_{\omega}^{n_0 k m}}{F_{\omega}}} \le \frac{(\mu + \beta)^{n_0 k m}}{2}.
    \end{equation}
    By Lemma \ref{lemma:slow_space_stability} there exists $\epsilon_{\beta,2} \in (0, \epsilon_\beta)$ such that if $S \in \mathcal{LY}(C_1, C_2, r, R) \cap \mathcal{O}_{\epsilon_{\beta,2}}(L)$ then
    \begin{equation}\label{eq:slow_decay_3}
      \sup_{\omega \in \Omega} d_H(F_{\omega}, F_{\omega}^S) \le (2C_3)^{-1} \left(\frac{\mu + \beta}{R}\right)^{n_0 k m}.
    \end{equation}
    We obtain the required inequality by setting $\epsilon_\beta = \min\{ \epsilon_{\beta,1},\epsilon_{\beta,2}\}$ and then applying \eqref{eq:slow_decay_2} and \eqref{eq:slow_decay_3} to \eqref{eq:slow_decay_1}.
  \end{proof}
\end{lemma}

\begin{proof}[{The proof of Proposition \ref{prop:slow_space_props}}]
  Lemma \ref{lemma:slow_space_angle_bound} proves \eqref{eq:slow_space_props_1}, while
  Lemma \ref{lemma:slow_equivariance} proves that $S_{\omega}^{n_0} F_{\omega}^S \subseteq F_{\sigma^{n_0}(\omega)}^S$ for every $\omega \in \Omega$.
  We get \eqref{eq:slow_space_props_2} and \eqref{eq:slow_space_props_3} from Lemmas \ref{lemma:slow_space_stability} and \ref{lemma:tnorm_conv_perturbed_proj}, respectively.

  Thus it remains to prove \eqref{eq:slow_space_props_4}, which we will do using Lemma \ref{lemma:slow_decay}. With the notation of Lemma \ref{lemma:slow_decay} set $n_1 = n_0 k m$. For $n \in \Z^+$ write $n = \ell n_1 + j$ where $\ell \in \Z^+$ and $j \in \{0, \dots, n_1 - 1\}$.
  By Lemma \ref{lemma:slow_decay} and the equivariance of $(F_\omega^S)_{\omega \in \Omega}$ we have for $S \in \mathcal{LY}(C_1, C_2, r, R) \cap \mathcal{O}_{\epsilon_\beta}(L)$ that
  \begin{equation*}
    \norm{\restr{S_{\omega}^{\ell n_1}}{F_{\omega}^S}} \le \prod_{i=0}^{\ell - 1} \norm{
    \restr{S_{\sigma^{(in_1)}(\omega)}^{n_1}}{F_{\sigma^{(in_1)}(\omega)}^S } } \le (\mu + \beta)^{\ell n_1},
  \end{equation*}
  and so
  \begin{equation*}
    \norm{\restr{S_{\omega}^{n}}{F_{\omega}^S}} \le \norm{S_{\sigma^{\ell n_1}(\omega)}^j} \norm{\restr{S_{\omega}^{\ell n_1}}{F_{\omega}^S}} \le C_3 \left(\frac{R}{\mu + \beta}\right)^{j} (\mu + \beta)^{n}.
  \end{equation*}
  Since $\mu + \beta \le R$ we obtain \eqref{eq:slow_space_props_4} upon setting $C_\beta = C_3 \left(\frac{R}{\mu + \beta}\right)^{n_1 -1}$.
\end{proof}

\subsection{Completing the proof of Theorem \ref{thm:stability_cocycle}}\label{sec:end_of_proof}

We have assembled most of the ingredients that are required to complete the proof of Theorem \ref{thm:stability_cocycle}.
Indeed, all of the conclusions of Theorem \ref{thm:stability_cocycle} are verified by Propositions \ref{prop:fast_space_props} and \ref{prop:slow_space_props}, except for the following result.

\begin{proposition}
  There exists $\epsilon' \in (0, \epsilon_1)$ such that if $S \in \mathcal{LY}(C_1, C_2, r, R) \cap \mathcal{O}_{\epsilon'}(L)$ then the fast spaces $(E_\omega^S)_{\omega \in \Omega} \in \prod_{\omega \in \Omega} \mathcal{G}_d(X_\omega)$ and slow spaces $(F_\omega^S)_{\omega \in \Omega} \in \prod_{\omega \in \Omega} \mathcal{G}^d(X_\omega)$ produced by Propositions \ref{prop:fast_contraction_mapping} and \ref{prop:slow_contraction_mapping}, respectively, form a hyperbolic splitting of index $d$ for $S$.
  \begin{proof}
    Fix $\beta \in (0, (\lambda - \mu)/2)$.
    By Propositions \ref{prop:fast_space_props} and \ref{prop:slow_space_props} there exists $\epsilon' >0$ and $C_\beta$ such that if $S \in \mathcal{LY}(C_1, C_2, r, R) \cap \mathcal{O}_{\epsilon'}(L)$, $n \in \Z^+$ and $\omega \in \Omega$ then
    \begin{equation*}
      \norm{\restr{S_\omega^n}{F_\omega^S}} \le C_\beta (\mu + \beta)^n,
    \end{equation*}
    and if, in addition, $v \in E_\omega^S$ then
    \begin{equation*}
      \norm{S_\omega^n v} \ge C_\beta^{-1} (\lambda - \beta)^n.
    \end{equation*}
    Hence, it suffices to prove that for every $\omega \in \Omega$ we have $S_\omega E_\omega^S = E_{\sigma(\omega)}^S$ and $S_\omega F_\omega^S \subseteq F_{\sigma(\omega)}^S$.
    We will prove these separately.

    \paragraph{The equivariance of $(F_\omega^S)_{\omega \in \Omega}$.}
    If $S_{\omega}F_{\omega}^S \not\subseteq F_{\sigma(\omega)}^S$ then there exists $f \in F_{\omega}^S$ such that $\norm{f} = 1$ and $S_{\omega} f \notin F_{\sigma(\omega)}^S$.
    Thus $\codim(F_{\sigma(\omega)}^S \oplus \vspan\{ S_{\omega} f\}) = d - 1$, and so there exists $e \in E_{\sigma(\omega)}^S \cap (F_{\sigma(\omega)}^S \oplus \vspan\{ S_{\omega} f\})$ with $\norm{e} = 1$.
    Write $e = aS_{\omega} f + f'$ where $a$ is a scalar and $f' \in F_{\sigma(\omega)}^S$. For every $n \in \Z^+$ we have
    \begin{equation*}
      C_\beta^{-1} (\lambda - \beta)^n \le \norm{S^m_{\sigma(\omega)}e} \le \abs{a} \norm{S^{n+1}_{\omega}f} + \norm{S^n_{\sigma(\omega)}f'} \le C_\beta (\mu + \beta)^n\left(\abs{a}(\mu + \beta) \norm{f} + \norm{f'}\right).
    \end{equation*}
    Since $\lambda - \beta > \mu + \beta$ we obtain a contradiction by taking $n \to \infty$.

    \paragraph{The equivariance of $(E_\omega^S)_{\omega \in \Omega}$.}
    If $S_{\omega} E_{\omega}^S \ne E_{\sigma(\omega)}^S$ then there exists $e \in E_{\omega}^S$ such that $\norm{e} = 1$ and $S_{\omega} e \notin E_{\sigma(\omega)}^S$.
    Recall that for the constant $n_0$ produced by Propositions \ref{prop:fast_contraction_mapping} we have $S^{m n_0}_{\sigma^{-m n_0}(\omega)} E_{\sigma^{-m n_0}(\omega)}^S = E_{\omega}^S$ for every $m \in \Z^+$.
    Hence, for each $m \in \Z^+$ there is a unique vector $e_m \in E_{\sigma^{-m n_0}(\omega)}^S$  satisfying $S_{\sigma^{-m n_0}(\omega)}^{mn_0} e_m = e$.
    Since $S_{\sigma^{-mn_0+1}(\omega)}^{mn_0} E_{\sigma^{-mn_0+1}(\omega)}^S = E_{\sigma(\omega)}^S$ we must have $S_{\sigma^{-m n_0}(\omega)} e_m \notin E_{\sigma^{-m n_0+1}(\omega)}^S$.
    Thus $\dim(E_{\sigma^{-mn_0+1}(\omega)}^S \oplus \vspan\{ S_{\sigma^{-mn_0}(\omega)} e_m \}) = d+1$, and so there exists $f_m \in (E_{\sigma^{-mn_0+1}(\omega)}^S \oplus \vspan\{ S_{\sigma^{-mn_0}(\omega)} e_m\}) \cap F_{\sigma^{-mn_0 +1}(\omega)}^S$ with $\norm{f_m} = 1$.
    Writing $f_m = a_m S_{\sigma^{-mn_0}(\omega)}e_m + g_m$ for some scalar $a_m$ and $g_m \in E_{\sigma^{-mn_0 +1}(\omega)}^S$,
    we have
    \begin{equation*}\begin{split}
      C_\beta (\mu + \beta)^{m n_0 -1} &\ge \norm{S^{mn_0-1}_{\sigma^{-mn_0+1}(\omega)} f_m} \\
      &= \norm{a_m S^{m n_0}_{\sigma^{-mn_0}(\omega)} e_m + S^{mn_0-1}_{\sigma^{-mn_0+1}(\omega)} g_m} \\
      &\ge \max\left\{ \abs{a_m} \norm{S^{m n_0}_{\sigma^{-mn_0}(\omega)} e_m} \norm{\Pi_{ \vspan\{S_{\omega} e\} || E_{\sigma(\omega)}^S}}^{-1}, \norm{S^{mn_0-1}_{\sigma^{-mn_0+1}(\omega)} g_m} \norm{\Pi_{ E_{\sigma(\omega)}^S || \vspan\{S_{\omega} e\} }}^{-1} \right\}\\
      &\ge \frac{(\lambda - \beta)^{m n_0 -1}}{2C_\beta} \max\left\{ \abs{a_m} (\lambda - \beta) \norm{e_m} , \norm{g_m} \right\} \norm{\Pi_{ \vspan\{S_{\omega} e\} || E_{\sigma(\omega)}^S}}^{-1}.
    \end{split}\end{equation*}
    Since $f_m = a_m S_{\sigma^{-mn_0}(\omega)}e_m + g_m$ and $\norm{f_m} = 1$ we have $1 \le \abs{a_m} C_3 R\norm{e_m} + \norm{g_m}$, and so
    \begin{equation*}
      C_\beta (\mu + \beta)^{m n_0 -1} \ge \frac{(\lambda - \beta)^{m n_0 -1}}{2C_\beta} \max\left\{ (\lambda - \beta) \frac{1 - \norm{g_m}}{C_3R}, \norm{g_m} \right\} \norm{\Pi_{ \vspan\{S_{\omega} e\} || E_{\sigma(\omega)}^S}}^{-1}.
    \end{equation*}
    For any value of $\norm{g_m}$ we have
    \begin{equation*}
      \max\left\{ (\lambda - \beta) \frac{1 - \norm{g_m}}{C_3R}, \norm{g_m} \right\} \ge \frac{\lambda - \beta}{C_3R + \lambda - \beta}.
    \end{equation*}
    Thus
    \begin{equation*}
      C_\beta (\mu + \beta)^{m n_0 - 1} \ge \frac{(\lambda - \beta)^{mn_0}}{2C_\beta(C_3R + \lambda - \beta)} \norm{\Pi_{ \vspan\{S_{\omega} e\} || E_{\sigma(\omega)}^S}}^{-1},
    \end{equation*}
    and so we obtain a contradiction by sending $m \to \infty$.
  \end{proof}
\end{proposition}

\section{Application to random linear systems}\label{sec:stability_lyapunov}

In this section we will use Theorem \ref{thm:stability_cocycle} to prove the stability of the Oseledets splitting and Lyapunov exponents of certain random linear systems (see Theorem \ref{thm:stability_lyapunov}).
In order to properly formulate our results we need some language from \cite{GTQuas1} (although we note the existence of alternatives, such as \cite{blumenthal2016volume,froyland2013semi}).
\begin{definition}\label{def:random_linear_system}
  A separable strongly measurable random linear system is a tuple $\mathcal{Q} = (\Omega, \mathcal{F}, \mathbb{P}, \sigma, X, Q)$ such that $(\Omega, \mathcal{F}, \mathbb{P})$ is a Lebesgue space, $\sigma : \Omega \to \Omega$ is a $\mathbb{P}$-preserving transformation of $(\Omega, \mathcal{F}, \mathbb{P})$, $X$ is a separable Banach space, and the generator $Q : \Omega \to \LL(X)$ is strongly measurable i.e. for every $x \in X$ the map $\omega \mapsto Q_\omega(x)$ is $(\mathcal{F}, \mathcal{B}_X)$-measurable where $\mathcal{B}_X$ is the Borel $\sigma$-algebra on $X$.
  We say that $\mathcal{Q}$ has an ergodic invertible base if $\sigma$ is invertible and $\mathbb{P}$-ergodic.
\end{definition}

\begin{remark}
  We will frequently use an alternative characterisation of strong measurability from \cite[Appendix A]{GTQuas1}: in the context of Definition \ref{def:random_linear_system} this condition is equivalent to $Q$ being $(\mathcal{F},\mathcal{S})$-measurable, where $\mathcal{S}$ is the Borel $\sigma$-algebra of the strong operator topology on $\LL(X)$.
\end{remark}

\begin{definition}\label{def:oseledets_splitting}
  Let $\mathcal{Q} = (\Omega, \mathcal{F}, \mathbb{P}, \sigma, X, Q)$
  be a separable strongly measurable random linear system.
  Suppose that there exists $k_Q \in \Z^+$, constants $\lambda_{1,Q} > \lambda_{2,Q} > \dots > \lambda_{k_Q,Q} > \mu_Q$, a map $F_Q : \Omega \mapsto \mathcal{G}(X)$, and for each $i \in \{1, \dots, k_Q\}$ a positive integer $d_{i,Q}$ and a map $E_{i,Q} : \Omega \to \mathcal{G}_{d_{i,Q}}(X)$,
  such that
  \begin{enumerate}
    \item For a.e. $\omega$ we have
    \begin{equation}\label{eq:oseledets_splitting}
      X = \left(\bigoplus_{i=1}^{k_Q} E_{i, Q}(\omega) \right) \oplus F_Q(\omega),
    \end{equation}
    and each of the projections associated to the decomposition \eqref{eq:oseledets_splitting} is strongly measurable.
    \item For every $i \in \{1, \dots, k_Q\}$ and a.e. $\omega \in \Omega$ we have $Q_\omega E_{i,Q}(\omega) = E_{i,Q}(\sigma(\omega))$, and for each non-zero $v \in E_{i,Q}(\omega)$ one has
    \begin{equation}\label{eq:fast_growth}
      \lim_{n \to \infty} \frac{1}{n} \log \norm{ Q_\omega^n v} = \lambda_{i,Q}.
    \end{equation}
    \item For a.e. $\omega \in \Omega$ one has $Q_\omega F_Q(\omega) \subseteq F_Q(\sigma(\omega))$ and
    \begin{equation}\label{eq:slow_growth}
      \lim_{n \to \infty } \frac{1}{n} \log \norm{ \restr{Q_\omega^n}{F_Q(\omega)}} \le \mu_Q.
    \end{equation}
  \end{enumerate}
  Then we call \eqref{eq:oseledets_splitting} an Oseledets splitting for $\mathcal{Q}$ of dimension $d = \sum_{i=1}^{k_Q} d_{i,Q}$.
  The numbers $\{\lambda_{i,Q} \}_{i=1}^{k_Q}$ are called the exceptional Lyapunov exponents of $\mathcal{Q}$, and we say that $d_{i,Q}$ is the multiplicity of $\lambda_{i,Q}$.
  The spaces $E_{i,Q}(\omega)$ and $F_Q(\omega)$ are called Oseledets subspaces of $\mathcal{Q}$.
  For convenience we set $\lambda_{k_Q+1,Q} = \mu_Q$.
  Finally, the Lyapunov exponents of $\mathcal{Q}$ counted with multiplicities is the sequence
  \begin{equation}\label{eq:lyapunov_w_multiplicity}
    \lambda_{1,Q}, \dots, \lambda_{1,Q}, \lambda_{2, Q} \dots, \lambda_{2, Q}, \lambda_{3,Q}, \dots, \lambda_{k_Q,Q},
  \end{equation}
  where each $\lambda_{i,Q}$ occurs $d_{i,Q}$ times. For $\ell \in \{1, \dots, d\}$ we set $\gamma_{\ell,Q}$ to be the $\ell$th element of \eqref{eq:lyapunov_w_multiplicity} (from left to right).
\end{definition}

\begin{remark}\label{remark:measureable_fast_spaces}
  It follows from Lemma \ref{lemma:strong_cont_proj_grassmannian} that $\omega \mapsto E_{i, Q}(\omega)$ is $(\mathcal{F}, \mathcal{B}_{\mathcal{G}(X)})$-measurable for every $i \in \{1, \dots, k_Q\}$.
\end{remark}

\begin{remark}
    The existence of an Oseledets splittings may be guaranteed by a multiplicative ergodic theorem. There are now a plethora of such theorems, starting with \cite{oseledets1968multiplicative} and being generalised in a number of directions, but for our desired application we are only concerned with semi-invertible multiplicative ergodic theorems on Banach spaces.
    The semi-invertibility of such a result refers to the requirement that $\sigma$ is invertible, but that no invertibility assumption is placed on the generator $Q$.
    In an infinite-dimensional Banach space there is also a requirement that the random linear system being considered is quasi-compact, which, roughly speaking, implies that the iterates of the cocycle become increasingly close to a compact cocycle. We refer the reader to \cite{GTQuas1,gonzalez2015concise,blumenthal2016volume,froyland2013semi} for precise statements of various semi-invertible multiplicative ergodic theorems.
    Finally, we note that a semi-invertible multiplicative ergodic theorem for compact cocycles on a continuous field of Banach spaces was recently developed \cite{varzaneh2019dynamical}, in which case the Banach space is allowed to vary fiber-wise.
    This setting is quite similar that of Section \ref{sec:saks_space_stability}, and suggests the possibility of generalising the results of this section to cocycles on Banach fields.
\end{remark}

To a separable strongly measurable random linear system $(\Omega, \mathcal{F}, \mathbb{P}, \sigma, X, Q)$ we may associate a canonical bounded linear endomorphism of $\mathbb{X} = \bigsqcup_{\omega \in \Omega} \{\omega\} \times X$, which we also denote by $Q$, that is defined by
\begin{equation*}
  Q(\omega,f) = (\sigma(\omega), Q_\omega f).
\end{equation*}
To apply Theorem \ref{thm:stability_cocycle} we require a hyperbolic splitting for $Q$ when considered as an element of $\End(\mathbb{X}, \sigma)$.
The following definition makes precise this requirement in the context of Oseledets splittings.

\begin{definition}\label{def:hyperbolic_oseledets}
  Suppose that $\mathcal{Q} = (\Omega, \mathcal{F}, \mathbb{P}, \sigma, X, Q)$ is a separable strongly measurable random linear system with an Oseledets splitting of dimension $d$ as in Definition \ref{def:oseledets_splitting}.
  For each $i \in \{1, \dots, k_Q\}$ let $U_{i,Q}(\omega) = \bigoplus_{j \le i } E_{i, Q}(\omega)$ and $V_{i,Q}(\omega) = \left(\bigoplus_{j > i} E_{j, Q}(\omega)\right) \oplus F_Q(\omega)$.
  We say that $\mathcal{Q}$ has a hyperbolic Oseledets splitting up to the dimension $d$ if there exists a $\sigma$-invariant set $\Omega' \subseteq \Omega$ of full $\mathbb{P}$-measure such that for each $i \in \{1, \dots, k_Q\}$ the families of subspaces $\{U_{i,Q}(\omega)\}_{\omega \in \Omega'}$ and $\{V_{i,Q}(\omega)\}_{\omega \in \Omega'}$ form the equivariant fast and slow spaces, respectively, for a hyperbolic splitting of the restriction of $Q$ to $\mathbb{X}' = \bigsqcup_{\omega \in \Omega'} \{\omega\} \times X$ when $Q$ is considered as an element of $\End(\mathbb{X}, \sigma)$.
\end{definition}

\begin{remark}\label{remark:alt_hyperbolic_oseledets}
  Unpacking the various requirements in Definition \ref{def:hyperbolic_oseledets}, we observe that the Oseledets splitting of $\mathcal{Q}$ being hyperbolic is equivalent to the existence of a $\sigma$-invariant set $\Omega' \subseteq \Omega$ of full $\mathbb{P}$-measure, constants $\Theta, C > 0$ and $\eta < 2^{-1}\min_{1 \le i \le k_Q} \{\lambda_{i,Q} - \lambda_{i+1,Q}\}$ such that for every $i \in \{1, \dots, k_Q\}$, $\omega \in \Omega'$ and $n \in \Z^+$ we have
  \begin{equation}\label{eq:alt_hyperbolic_oseledets_1}
    \max\left\{\norm{\Pi_{U_{i,Q}(\omega) || V_{i,Q}(\omega)}}, \norm{\Pi_{V_{i,Q}(\omega) || U_{i,Q}(\omega)}}\right\} < \Theta,
  \end{equation}
  \begin{equation}\label{eq:alt_hyperbolic_oseledets_2}
    \norm{\restr{Q_\omega^n}{V_{i,Q}(\omega)}} \le C e^{n(\lambda_{i+1,Q} + \eta)},
  \end{equation}
  and
  \begin{equation}\label{eq:alt_hyperbolic_oseledets_3}
    \norm{\left(\restr{Q_\omega^n}{U_{i,Q}(\omega)}\right)^{-1}} \le C^{-1} e^{-n(\lambda_{i,Q} - \eta)}.
  \end{equation}
\end{remark}

Before stating our main result for this section we require some notation.
Suppose that  $\mathcal{Q} = (\Omega, \mathcal{F}, \mathbb{P}, \sigma, X, Q)$
is a separable strongly measurable random linear system with Oseledets splitting of dimension $d$.
Rather than indexing the projections onto Oseledets spaces with the index of their Lyapunov exponents, it will be more convenient to state our perturbation results by indexing projections by collections of Lyapunov exponents.
If $I \subseteq \R$ is a open interval such that $I \subseteq (\mu_Q, \infty)$ and $\partial I \cap \{\lambda_{i,Q} : 1 \le i \le k_Q \} = \emptyset$ then we say that $I$ separates the Lyapunov spectrum of $\mathcal{Q}$.
When $I$ separates the Lyapunov spectrum of $\mathcal{Q}$ we may define $\Pi_{I,Q}(\omega) \in \LL(X)$ to be the projection onto
\begin{equation*}
  \bigoplus_{i : \lambda_{i,Q} \in I} E_{i,Q}
\end{equation*}
according to the decomposition \eqref{eq:oseledets_splitting}. Finally, if $(X, \norm{\cdot}, \wnorm{\cdot})$ is a Saks space and $\epsilon > 0$ then, as in Section \ref{sec:saks_space_stability}, we set
\begin{equation*}
  \mathcal{O}_{\epsilon}(Q) = \left\{ P: \Omega \mapsto \LL(X) \, \bigg\vert \, P \text{ is strongly measurable with } \esssup_{\omega \in \Omega} \tnorm{Q_\omega - P_\omega} < \epsilon \right\}.
\end{equation*}
Our main result for this section is the following.

\begin{theorem}\label{thm:stability_lyapunov}
  Suppose that $(X, \norm{\cdot}, \wnorm{\cdot})$ is a Saks space, with $(X, \norm{\cdot})$ being a Banach space, that $\mathcal{Q} = (\Omega, \mathcal{F}, \mathbb{P}, \sigma, X, Q)$ is a separable strongly measurable random linear system with ergodic invertible base and a hyperbolic Oseledets splitting of dimension $d \in \Z^+$, and that $Q \in \mathcal{LY}(C_1, C_2, r, R) \cap \End_S(\mathbb{X},\sigma)$ for some $C_1,C_2,R > 0$ and $r \in [0, e^{\mu_Q})$.
  There exists $\epsilon_0 >0$ such that if $\mathcal{P} = (\Omega, \mathcal{F}, \mathbb{P}, \sigma, X, P)$ is a separable strongly measurable random linear system with $P \in \mathcal{LY}(C_1, C_2, r, R) \cap \mathcal{O}_{\epsilon_0}(Q)$ then $\mathcal{P}$ also has an Oseledets splitting of dimension $d$.
  In addition, there exists $c_0 < 2^{-1}\min_{1 \le i \le k_Q} \{\lambda_{i,Q} - \lambda_{i+1,Q}\}$ such that each $I_{i} = (\lambda_{i,Q} - c_0, \max\{\lambda_{i,Q}, \log(\delta_{1i} R)\} + c_0)$, $i \in \{1, \dots, k_Q\}$, separates the Lyapunov spectrum of $\mathcal{P}$, and the corresponding projections satisfy
  \begin{equation}\label{eq:stability_lyapunov_0000}
    \forall i \in \{1, \dots, k_Q\}, \text{ a.e. } \omega \in \Omega \quad \rank(\Pi_{I_i, P}(\omega)) = d_{i,Q},
  \end{equation}
  and
  \begin{equation}\label{eq:stability_lyapunov_00000}
    \sup \left\{ \esssup_{\omega \in \Omega} \norm{\Pi_{I_i, P}(\omega) } : P \in \mathcal{LY}(C_1, C_2, r, R) \cap \mathcal{O}_{\epsilon_0}(Q), 1\le i \le k_Q \right\} < \infty.
  \end{equation}
  Moreover, for every $\nu > 0$ there exists $\epsilon_\nu \in (0, \epsilon_0)$ so that if $P \in \mathcal{LY}(C_1, C_2, r, R) \cap \mathcal{O}_{\epsilon_\nu}(Q)$ then
  \begin{equation}\label{eq:stability_lyapunov_0}
      \sup_{1 \le i \le d} \abs{\gamma_{i, Q} - \gamma_{i, P}} \le \nu,
  \end{equation}
  \begin{equation}\label{eq:stability_lyapunov_00}
      \sup_{1 \le i \le k_Q} \esssup_{\omega \in \Omega} \tnorm{\Pi_{I_{i},Q}(\omega) - \Pi_{I_{i},P}(\omega)} \le \nu,
  \end{equation}
  and
  \begin{equation}\label{eq:stability_lyapunov_000}
      \esssup_{\omega \in \Omega} d_H(F_{k_Q,Q}(\omega), F_{k_P,P}(\omega) ) \le \nu.
  \end{equation}
\end{theorem}

\begin{remark}
  Note that $\Pi_{I_{i},Q}(\omega)$ is simply the projection onto $E_{i,Q}(\omega)$ according to the Oseledets splitting of $\mathcal{Q}$.
\end{remark}

\begin{remark}
  By possibly rescaling $\abs{\cdot}$, without loss of generality we may assume that the Saks space $(X, \norm{\cdot}, \wnorm{\cdot})$ in Theorem \ref{thm:stability_lyapunov} is normal.
\end{remark}

\begin{remark}\label{remark:comparison_to_kl}
  Theorem \ref{thm:stability_cocycle} may be considered a generalisation of the results of \cite{keller1999stability}.
  Indeed, in the case where $\Omega$ is a singleton we obtain a version of the results of \cite{keller1999stability}.
  We note that one condition from \cite{keller1999stability} has been substantially weakened, namely condition (2) from \cite{keller1999stability} is generalised to the requirement that $Q$ is a Saks space equicontinuous endomorphism (see Proposition \ref{prop:saks_equicont}, \eqref{eq:equicont_seq_1} and Remark \ref{remark:ss_equi}), which we only require for the unperturbed endomorphism $Q$, and not for any perturbation.
  In addition, the convergence of the slow spaces in the Grassmannian as in \eqref{eq:stability_lyapunov_000} is new.
  We did not pursue H{\"o}lder bounds on the $\tnorm{\cdot}$-error between the perturbed and unperturbed projections as in \cite{keller1999stability}.
  It is natural to conjecture that the conclusion of Theorem
  \ref{thm:stability_lyapunov} (and Theorem \ref{thm:stability_cocycle}) could be strengthened to obtain H{\"o}lder error bounds in \eqref{eq:stability_lyapunov_0}, \eqref{eq:stability_lyapunov_00} and \eqref{eq:stability_lyapunov_000} under the additional assumption that $\esssup_{\omega \in \Omega} \wnorm{Q_\omega} < \infty$.
\end{remark}

\begin{remark}\label{remark:not_hyperbolic}
  Contrary to what one might expect given Theorem \ref{thm:stability_cocycle}, in Theorem \ref{thm:stability_lyapunov} we cannot conclude that $\mathcal{P}$ possesses a hyperbolic Oseledets splitting.
  The obstruction for this is the following: if $\mathcal{Q}$ has a Lyapunov exponent $\lambda_{j,Q}$ with $d_{j,Q} > 1$ then after perturbing the cocycle one expects the exponent to immediately split into $d_{j,Q}$ distinct exponents.
  None of the hypotheses of Theorem \ref{thm:stability_lyapunov} may be used to control the angle between the Oseledets spaces for these new Lyapunov exponents, which prevents us from concluding these splittings are hyperbolic.
  However, it follows from Theorem \ref{thm:stability_cocycle}, that if every Lyapunov exponent of $\mathcal{Q}$ has multiplicity 1 then the Oseledets splitting for $\mathcal{P}$ is hyperbolic.
\end{remark}

The proof of Theorem \ref{thm:stability_lyapunov} is broken into a number of steps. In Section \ref{sec:characterising_oseledets} we produce an Oseledets splitting of dimension $d$ for $\mathcal{P}$, and then we relate this Oseledets splitting to various hyperbolic splittings produced by Theorem \ref{thm:stability_cocycle}.
Once this is done, in Section \ref{sec:Lyapunov_converge} we characterise and then prove the stability of the Lyapunov exponents.

However, before embarking on the proof of Theorem \ref{thm:stability_lyapunov}, we will discuss its relation to the \cite[Theorem 1.10]{bogenschutz2000stochastic}, to which our result bares a strong resemblance. The primary differences are the following:
\begin{enumerate}
  \item In \cite{bogenschutz2000stochastic} it is required that convergence in \eqref{eq:fast_growth} and \eqref{eq:slow_growth} is uniform in $\omega$, while we only require the weaker bounds \eqref{eq:alt_hyperbolic_oseledets_2} and \eqref{eq:alt_hyperbolic_oseledets_3}.
  \item The perturbations in \cite{bogenschutz2000stochastic} are required to be asymptotically small: (i) each iterate of the perturbed cocycle must converge uniformly in the strong operator topology to the corresponding iterate of the unperturbed cocycle, and (ii) there exists $s \in (\lambda_{k_Q +1,Q}, \lambda_{k_Q,Q})$ and $N \in \Z^+$ such that for every $n > N$ there is $\epsilon(n)$ so that for all $\epsilon \in (0, \epsilon(n))$ and a.e. $\omega \in \Omega$ one has
  \begin{equation*}
    \norm{Q_\omega^n - P_\omega^n}_{\LL(X)} \le e^{ns}.
  \end{equation*}
  We compare (i) to closeness in the Saks space sense in Proposition \ref{prop:sot_ss} and show that our hypotheses are weaker for pre-compact Saks spaces, which is a common setting for Perron-Frobenius operator cocycles.
  On the other hand, the condition (ii) is not directly comparable to any of our hypotheses, although it is comparable `in spirit' to our requirement that the perturbed cocycle lies in a Lasota-Yorke class: the exponent $s$ plays a similar role to the $r$ term in our Lasota-Yorke inequalities, in that one cannot conclude anything about the stability of any Lyapunov exponents of modulus smaller than $s$ in \cite{bogenschutz2000stochastic}, or $\log r$ in Theorem \ref{thm:stability_lyapunov}.
  \item Due to the weaker requirements of our result, our conclusions on the stability of the Oseledets spaces are weaker than that of \cite{bogenschutz2000stochastic}.
  \item We require the additional hypotheses that the unperturbed cocycle is a Saks space equicontinuous endomorphism, which presupposes that $X$ admits a Saks space structure. However, (pre-)compact Saks spaces are commonly used to study the statistical properties of dynamical systems via Perron-Frobenius operators, and so these hypotheses are natural for our primary application.
\end{enumerate}

\begin{proposition}\label{prop:sot_ss}
  Suppose that $(X, \norm{\cdot}, \wnorm{\cdot})$ is a (pre-)compact Saks space, and that $\{Q_n\}_{n \in \Z^+} \subseteq \LL_S(X)$ is an equicontinuous subset of $\LL_S(X)$ which converges in the strong operator topology to $Q \in \LL_S(X)$. Then $Q_n \to Q$ in $(\LL_S(X), \norm{\cdot}_{\LL(X)}, \tnorm{\cdot})$.
  \begin{proof}
    That $\{Q_n\}_{n \in \Z^+}$ is bounded in $\LL(X)$ follows from Proposition \ref{prop:saks_equicont}.
    Let $G_\epsilon \subseteq B_{\norm{\cdot}}$ be a finite set such that
    \begin{equation*}
      \inf_{\norm{f} = 1}\inf_{g \in G} \wnorm{f -g} \le \epsilon.
    \end{equation*}
    Then
    \begin{equation*}
      \tnorm{Q_n - Q} \le \sup_{g \in G_\epsilon} \norm{(Q_n - Q)g} + \sup_{\norm{f} = 1} \inf_{g \in G_\epsilon} \wnorm{(Q_n - Q)(f-g)}
    \end{equation*}
    Since $\{Q\} \cup \{Q_n\}_{n \in \Z^+}$ is equicontinuous in $\LL_S(X)$, by Proposition \ref{prop:saks_equicont} we have for every $\kappa > 0$ a $C_\kappa$ such that for every $n \in \Z^+$
    \begin{equation*}
        \tnorm{Q_n - Q} \le \sup_{g \in G_\epsilon} \norm{(Q_n - Q)g} + 2\kappa + C_\kappa \epsilon.
    \end{equation*}
    Sending $n \to \infty$ yields
    \begin{equation}\label{eq:sot_ss_1}
      \limsup_{n \to \infty} \tnorm{Q_n - Q} \le 2\kappa + C_\kappa \epsilon.
    \end{equation}
    By first choosing $\kappa$ to be very small, and then shrinking $\epsilon$ appropriately, we may make the right hand side of \eqref{eq:sot_ss_1} as small as we like, which implies that $\lim_{n \to \infty} \tnorm{Q_n - Q} = 0$.
  \end{proof}
\end{proposition}

\subsection{Characterising the perturbed Oseledets splitting}\label{sec:characterising_oseledets}

Recall $\eta$ and $\Omega'$ from Remark \ref{remark:alt_hyperbolic_oseledets}, and let $\beta_0 > 0 $ satisfy $\eta + \beta_0 < 2^{-1}\min_{1 \le i \le k_Q} \{\lambda_{i,Q} - \lambda_{i+1,Q}\}$.
For each $i \in \{1, \dots, k_Q\}$ we may apply Theorem \ref{thm:stability_cocycle} to $Q$ with respect to the hyperbolic splitting composed of fast spaces $\{U_{i,Q}(\omega)\}_{\omega \in \Omega'}$ and slow spaces $\{V_{i,Q}(\omega)\}_{\omega \in \Omega'}$ to produce $\epsilon_{0}, C_0, \Theta_0 > 0$ so that if $P \in \mathcal{LY}(C_1, C_2, r, R)$ is strongly measurable and satisfies
\begin{equation}\label{eq:tnorm_closeness}
    \sup_{\omega \in \Omega'} \tnorm{Q_\omega - P_\omega} < \epsilon_0
\end{equation}
then $P$ has a hyperbolic splitting of index $\sum_{i \le j} d_{j,Q}$ (in the sense of Definition \ref{def:hyperbolic_splitting}). Moreover, if we denote the fast and slow spaces of these splittings by $\{U_{i,P}(\omega)\}_{\omega \in \Omega'}$ and $\{V_{i,P}(\omega)\}_{\omega \in \Omega'}$, respectively, then for every $n \in \Z^+$, $i \in \{1, \dots, k_Q\}$ and $\omega \in \Omega'$ we have
\begin{equation}\label{eq:oseledets_angle}
  \max \left\{ \norm{\Pi_{U_{i,P}(\omega) || V_{i,P}(\omega)} }, \norm{\Pi_{V_{i,P}(\omega) || U_{i,P}(\omega)} }\right\} < \Theta_0,
\end{equation}
\begin{equation}\label{eq:slow_oseledets}
  \norm{\restr{P_\omega^n}{V_{i,P}(\omega)}} \le C_0 e^{n(\lambda_{i+1,Q} + \eta +\beta_0)},
\end{equation}
and, for every $v \in U_{i,P}(\omega)$,
\begin{equation}\label{eq:fast_oseledets}
  \norm{P_\omega^n v} \ge C_0^{-1} e^{n(\lambda_{i,Q} -\eta - \beta_0)} \norm{v}.
\end{equation}

\begin{remark}
    If, rather than \eqref{eq:tnorm_closeness}, we just have that $P \in \mathcal{O}_{\epsilon_0}(Q)$, then we may instead consider the following construction. Let $\Omega_P \in \mathcal{F}$ have full $\mathbb{P}$-measure and satisfy
    \begin{equation*}
        \sup_{\omega \in \Omega_P} \tnorm{Q_\omega - P_\omega} < \epsilon_0.
    \end{equation*}
    By perhaps replacing $\Omega_P$ with $\bigcap_{n \in \Z} \sigma^{n}(\Omega_P)$ we may assume that $\Omega_P$ is $\sigma$-invariant.
    Let $\tilde{P} : \Omega \mapsto \LL(X)$ be defined by
    \begin{equation*}
        \tilde{P}_\omega =
        \begin{cases}
            P_\omega & \text{if } \omega \in \Omega' \cap \Omega_P, \\
            Q_\omega & \text{otherwise.}\\
        \end{cases}
    \end{equation*}
    Since $\tilde{P}_\omega = P_\omega$ a.e. and $(\Omega, \mathcal{F}, \mathbb{P})$ is a complete measure space it follows that $\tilde{P}$ is strongly measurable.
    By construction \eqref{eq:tnorm_closeness} holds with $\tilde{P}$ in place of $P$, and $\tilde{P} \in \mathcal{LY}(C_1, C_2, r, R)$ since $\Omega' \cap \Omega_P$ is $\sigma$-invariant.
    Thus Theorem \ref{thm:stability_cocycle} may be applied with $\tilde{P}$, which produces fast spaces $\{U_{i,\tilde{P}(\omega)}\}_{\omega \in \Omega'}$ and slow spaces $\{V_{i,\tilde{P}(\omega)}\}_{\omega \in \Omega'}$ for $\tilde{P}$, which restrict to fast and slow spaces for $P$ when considered on $\Omega_P$.
    Moreover, we obtain \eqref{eq:oseledets_angle}, \eqref{eq:slow_oseledets} and \eqref{eq:fast_oseledets} for $P$ and $\omega \in \Omega_P$ (i.e. for a.e. $\omega \in \Omega$).
    We will not discuss this technical point any further, and simply carry out of constructions a.e. for $P$.
\end{remark}

For each $i \in \{1, \dots, k_Q\}$ set
\begin{equation*}
  G_{i,P}(\omega) =
  \begin{cases}
    U_{1,P}(\omega) & i = 1,\\
    U_{i,P}(\omega) \cap V_{i-1,P}(\omega) & 1 < i \le k_Q.\\
  \end{cases}
\end{equation*}
and
\begin{equation*}
  H_{i,P}(\omega) =
  \begin{cases}
    V_{1,P}(\omega) & i = 1,\\
    V_{i,P}(\omega) \oplus U_{i-1,P}(\omega) & 1 < i \le k_Q.\\
  \end{cases}
\end{equation*}
Note that $\dim(G_{i,P}(\omega)) = \codim(H_{i,P}(\omega)) = d_{i,Q}$ and $X = G_{i,P}(\omega) \oplus H_{i,P}(\omega)$ for a.e. $\omega$ and each $i \in \{1, \dots, k_Q \}$. Moreover, for a.e. $\omega$ we have
\begin{equation}\label{eq:splitting_perturbed_blocks}
  X = \left(\bigoplus_{1 \le i \le k_Q} G_{i,P}(\omega) \right) \oplus V_{k_Q,P}(\omega).
\end{equation}
It is clear that $G_{i,Q}(\omega) = E_{i,Q}(\omega)$, and so we will consider $G_{i,P}(\omega)$ to be perturbation of $E_{i,Q}(\omega)$.
Our first main result for this section makes this idea rigorous, and is a straightforward application of Theorem \ref{thm:stability_cocycle}.
Later we will see that, in general, $G_{i,P}(\omega)$ is not an Oseledets space for $\mathcal{P}$, but rather a direct sum of finitely many Oseledets spaces of $\mathcal{P}$.

\begin{proposition}\label{prop:perturbed_oseledets_blocks_bound}
    With $\epsilon_0$ as at the beginning of this section, we have
    \begin{equation}\label{eq:perturbed_oseledets_blocks_bound_0}
        \sup \left\{ \esssup_{\omega \in \Omega} \norm{\Pi_{G_{i,P}(\omega) || H_{i, P}(\omega)}} : P \in \mathcal{LY}(C_1, C_2, r, R) \cap \mathcal{O}_{\epsilon_0}(Q), 1\le i \le k_Q \right\} \le \Theta_0^2 < \infty.
    \end{equation}
    Moreover, for every $\nu > 0$ there exists $\epsilon_\nu \in (0, \epsilon_0)$ so that if $P \in \mathcal{LY}(C_1, C_2, r, R) \cap \mathcal{O}_{\epsilon_\nu}(Q)$ and $i \in \{1, \dots, k_Q\}$ then
    \begin{equation}\label{eq:perturbed_oseledets_blocks_bound_00}
        \esssup_{\omega \in \Omega} \tnorm{\Pi_{E_{i,Q}(\omega) || H_{i, Q}(\omega)} - \Pi_{G_{i,P}(\omega) || H_{i, P}(\omega)}} \le \nu,
    \end{equation}
    and
    \begin{equation}\label{eq:perturbed_oseledets_blocks_bound_000}
      \esssup_{\omega \in \Omega} d_H(V_{k_Q,Q}(\omega), V_{k_Q,P}(\omega) ) \le \nu.
    \end{equation}
    \begin{proof}
    By \eqref{eq:oseledets_angle} we have
    \begin{equation*}
        \sup \left\{ \esssup_{\omega \in \Omega} \norm{\Pi_{U_{i,P}(\omega) || V_{i, P}(\omega)}} : P \in \mathcal{LY}(C_1, C_2, r, R) \cap \mathcal{O}_{\epsilon_0}(Q), 1\le i \le k_Q \right\} \le \Theta_0.
    \end{equation*}
    Since for $1< i \le k_Q$ we have
    \begin{equation}\label{eq:perturbed_oseledets_blocks_bound_1}
      \Pi_{G_{i,P}(\omega) || H_{i,P}(\omega)} = \Pi_{U_{i,P}(\omega) || V_{i, P}(\omega) } \Pi_{V_{i-1,P}(\omega) || U_{i-1, P}(\omega) },
    \end{equation}
    we may therefore bound the left hand side of \eqref{eq:perturbed_oseledets_blocks_bound_0} by $\max\{\Theta_0, \Theta_0^2\} = \Theta_0^2$, since $\Theta_0 \ge 1$ necessarily.

    We will now prove \eqref{eq:perturbed_oseledets_blocks_bound_00}, for which we note that it suffices consider each $i \in \{1, \dots, k_Q\}$ separately.
    By Theorem \ref{thm:stability_cocycle} there exists $\epsilon_{\nu} > 0$ so that if $P \in \mathcal{LY}(C_1, C_2, r, R) \cap \mathcal{O}_{\epsilon_\nu}(Q)$ then
    \begin{equation*}
        \esssup_{\omega \in \Omega} \tnorm{\Pi_{U_{1,Q}(\omega) || V_{1, Q}(\omega)} - \Pi_{U_{1,P}(\omega) || V_{1, P}(\omega)}} \le \nu,
    \end{equation*}
    which yields \eqref{eq:perturbed_oseledets_blocks_bound_00} for $i = 1$.
    Now assume that $1 < i \le k_Q$.
    If $P \in \mathcal{LY}(C_1, C_2, r, R) \cap \mathcal{O}_{\epsilon_0}(Q)$ and $1< i \le k_Q$ then by \eqref{eq:perturbed_oseledets_blocks_bound_1} we have for a.e. $\omega$ that
    \begin{equation}\begin{split}\label{eq:perturbed_oseledets_blocks_bound_2}
        \big\vert\kern-0.25ex\big\vert\kern-0.25ex\big\vert
        \Pi_{E_{i,Q}(\omega) || H_{i, Q}(\omega)}& - \Pi_{G_{i,P}(\omega) || H_{i, P}(\omega)}
        \big\vert\kern-0.25ex\big\vert\kern-0.25ex\big\vert \\
        \le &\tnorm{\Pi_{U_{i,Q}(\omega) || V_{i, Q}(\omega) }\left( \Pi_{V_{i-1,Q}(\omega) || U_{i-1, Q}(\omega) } - \Pi_{V_{i-1,P}(\omega) || U_{i-1, P}(\omega)}\right)} \\
        &+ \tnorm{\left(\Pi_{U_{i,Q}(\omega) || V_{i, Q}(\omega) } - \Pi_{U_{i,P}(\omega) || V_{i, P}(\omega) }\right) \Pi_{V_{i-1,P}(\omega) || U_{i-1, P}(\omega)}}  .
    \end{split}\end{equation}
    Lemma \ref{lemma:fast_space_ss} implies that for every $\kappa > 0$ there exists $C_\kappa$ such that
    \begin{equation*}\begin{split}
        \big\vert\kern-0.25ex\big\vert\kern-0.25ex\big\vert
        \Pi_{U_{i,Q}(\omega) || V_{i, Q}(\omega) }&\left( \Pi_{V_{i-1,Q}(\omega) || U_{i-1, Q}(\omega) } - \Pi_{V_{i-1,P}(\omega) || U_{i-1, P}(\omega)}\right)
        \big\vert\kern-0.25ex\big\vert\kern-0.25ex\big\vert\\
        \le &\kappa \norm{\Pi_{V_{i-1,Q}(\omega) || U_{i-1, Q}(\omega) } - \Pi_{V_{i-1,P}(\omega) || U_{i-1, P}(\omega)}} \\
        &+ C_\kappa \tnorm{\Pi_{V_{i-1,Q}(\omega) || U_{i-1, Q}(\omega) } - \Pi_{V_{i-1,P}(\omega) || U_{i-1, P}(\omega)}}\\
        \le &2\kappa \Theta_0 + C_\kappa \tnorm{\Pi_{V_{i-1,Q}(\omega) || U_{i-1, Q}(\omega) } - \Pi_{V_{i-1,P}(\omega) || U_{i-1, P}(\omega)}}.
    \end{split}\end{equation*}
    Thus from \eqref{eq:perturbed_oseledets_blocks_bound_2} we obtain
    \begin{equation}\begin{split}\label{eq:perturbed_oseledets_blocks_bound_3}
        \tnorm{\Pi_{E_{i,Q}(\omega) || H_{i, Q}(\omega)} - \Pi_{G_{i,P}(\omega) || H_{i, P}(\omega)}} \le &2\kappa \Theta_0 + C_\kappa\tnorm{\Pi_{V_{i-1,Q}(\omega) || U_{i-1, Q}(\omega) } - \Pi_{V_{i-1,P}(\omega) || U_{i-1, P}(\omega)}} \\
        &+ \Theta_0\tnorm{\Pi_{U_{i,Q}(\omega) || V_{i, Q}(\omega) } - \Pi_{U_{i,P}(\omega) || V_{i, P}(\omega) }}.
    \end{split}\end{equation}
    Fix $\kappa = \frac{\nu}{4\Theta_0}$.
    By Theorem \ref{thm:stability_cocycle}, there exists $\epsilon_\nu \in (0, \epsilon_0)$ so that if $P \in \mathcal{LY}(C_1, C_2, r, R) \cap \mathcal{O}_{\epsilon_\nu}(Q)$ and $i \in \{1, \dots, k_Q\}$ then
    \begin{equation}\label{eq:perturbed_oseledets_blocks_bound_4}
        \esssup_{\omega \in \Omega} \tnorm{\Pi_{U_{i,Q}(\omega) || V_{i, Q}(\omega) } - \Pi_{U_{i,P}(\omega) || V_{i, P}(\omega) }} \le \frac{\nu}{2(C_\kappa + \Theta_0)}.
    \end{equation}
    Thus by applying \eqref{eq:perturbed_oseledets_blocks_bound_4} to \eqref{eq:perturbed_oseledets_blocks_bound_3} we obtain \eqref{eq:perturbed_oseledets_blocks_bound_00}.
    Finally, we obtain \eqref{eq:perturbed_oseledets_blocks_bound_000} due to our application of Theorem \ref{thm:stability_cocycle} with respect to the hyperbolic splitting of $X$ into equivariant fast spaces $\{U_{k_Q,Q}(\omega) \}_{\omega \in \Omega'}$ and slow spaces $\{V_{k_Q,Q}(\omega) \}_{\omega \in \Omega'}$.
    \end{proof}
\end{proposition}

The second main result of this section confirms that the perturbed cocycle $\mathcal{P}$ has an Oseledets splitting, and that this Oseledets splitting refines the splitting in  \eqref{eq:splitting_perturbed_blocks}.

\begin{proposition}\label{prop:perturbed_oseledets_existence}
  With $\epsilon_0$ as in Proposition \ref{prop:perturbed_oseledets_blocks_bound}, if $P \in \mathcal{LY}(C_1, C_2, r, R) \cap \mathcal{O}_{\epsilon_0}(Q)$ then $\mathcal{P}$ has an Oseledets splitting of dimension $d$ and if for each $i \in \{1, \dots, k_Q\}$ we set
  \begin{equation*}
    S(i) = \left\{ j : \sum_{1 \le \ell < i} d_{\ell,Q}< \sum_{1 \le \ell \le j} d_{\ell,P} \le \sum_{1 \le \ell \le i} d_{\ell,Q}\right\}
  \end{equation*}
  then for a.e. $\omega$ we have
  \begin{equation}\label{eq:perturbed_oseledets_existence_0}
    G_{i,P}(\omega) = \bigoplus_{j \in S(i)} E_{j,P}(\omega),
  \end{equation}
  and $F_P(\omega) = V_{k_Q,P}(\omega)$.
\end{proposition}

The idea behind the proof of Proposition \ref{prop:perturbed_oseledets_existence} is rather simple: since each family $\{G_{i,P}(\omega)\}_{\omega \in \Omega}$ consists of $d_{i,Q}$-dimensional subspaces and is invariant under the action of $P$ we are essentially in the setting of the classical multiplicative ergodic theorem of Oseledets \cite{oseledets1968multiplicative}.
Unfortunately, actualising this idea requires the strong measurability of several constructions, the proofs of which are rather tedious. As such, many of the purely technical proofs have been deferred to Appendix \ref{app:stability_lyapunov_proofs}.

\begin{lemma}\label{lemma:measurable_hyperbolic_perturbed_spaces}
  For every $i \in \{1, \dots, k_Q\}$ the map $\omega \mapsto \Pi_{U_{i,P}(\omega) || V_{i,P}(\omega)}$ is strongly measurable.
  \begin{proof}
    From the construction of $\{U_{i,P}(\omega)\}_{\omega \in \Omega}$ in Proposition \ref{prop:fast_contraction_mapping} and by Proposition \ref{prop:graph_chart} there is $n_0 \in \Z^+$ such that almost uniformly we have
    \begin{equation*}
      \Pi_{U_{i,P}(\omega) || V_{i,Q}(\omega)} = \lim_{m\to \infty} \left(\Id +  (P_{\sigma^{-mn_0}(\omega)}^{mn_0})^*(0)\right) \Pi_{U_{i,Q}(\omega) || V_{i,Q}(\omega)},
    \end{equation*}
    where the graph transform $(P_{\sigma^{-mn_0}(\omega)}^{mn_0})^*$ maps $\LL(U_{i,Q}(\sigma^{-mn_0}(\omega)), V_{i,Q}(\sigma^{-mn_0}(\omega)))$ to $\LL(U_{i,Q}(\omega), V_{i,Q}(\omega))$.
    By \cite[Lemma A.5]{GTQuas1} the map $\omega \mapsto P_{\sigma^{-mn_0}(\omega)}^{mn_0}$ is strongly measurable for each $m$.
    Hence, as both $\omega \mapsto \Pi_{U_{i,Q}(\omega) || V_{i,Q}(\omega)}$ and $\omega \mapsto \Pi_{U_{i,Q}(\sigma^{-mn_0}(\omega)) || V_{i,Q}(\sigma^{-mn_0}(\omega))}$ are strongly measurable, by Proposition \ref{prop:measurable_graph_tranform} the map $\omega \mapsto \left(\Id +  (P_{\sigma^{-mn_0}(\omega)}^{mn_0})^*(0)\right) \Pi_{U_{i,Q}(\omega) || V_{i,Q}(\omega)}$ is strongly measurable for every $m \in \Z^+$.
    By Proposition \ref{prop:fast_space_props} we have $\esssup_{\omega \in \Omega} \norm{\Pi_{U_{i,P}(\omega) || V_{i,Q}(\omega)}} < \infty$, and so Lemma \ref{lemma:limit_of_strong_measurable} implies that $\omega \mapsto \Pi_{U_{i,P}(\omega) || V_{i,Q}(\omega)}$ is strongly measurable.

    From the construction of $\{V_{i,P}(\omega)\}_{\omega \in \Omega}$ in Proposition \ref{prop:slow_contraction_mapping} and by Proposition \ref{prop:graph_chart} there exists $n_1 \in \Z^+$ such that almost uniformly we have
    \begin{equation*}
      \Pi_{U_{i,P}(\omega) || V_{i,P}(\omega)} = \lim_{m \to \infty}\left( \Pi_{U_{i,P}(\omega) || V_{i,Q}(\omega)} - (P_{\omega}^{m n_1})_*(0)\Pi_{V_{i,Q}(\omega) || U_{i,P}(\omega)} \right),
    \end{equation*}
    where the graph transform $(P_{\omega}^{mn_1})_*$ maps $\LL(V_{i,Q}(\sigma^{mn_1}(\omega)), U_{i,P}(\sigma^{mn_1}(\omega)))$ to $\LL(V_{i,Q}(\omega), U_{i,P}(\omega))$.
    As both $\omega \mapsto \Pi_{U_{i,P}(\omega) || V_{i,Q}(\omega)}$ and $\omega \mapsto \Pi_{U_{i,P}(\sigma^{mn_1}(\omega)) || V_{i,Q}(\sigma^{mn_1}(\omega))}$ are strongly measurable, by Proposition \ref{prop:measurable_graph_tranform} the map $\omega \mapsto \Pi_{U_{i,P}(\omega) || V_{i,Q}(\omega)} - (P_{\omega}^{mn_1})_*(0)\Pi_{V_{i,Q}(\omega) || U_{i,P}(\omega)}$ is strongly measurable for every $m \in \Z^+$.
    By \eqref{eq:oseledets_angle} we have $\esssup_{\omega \in \Omega} \norm{\Pi_{U_{i,P}(\omega) || V_{i,P}(\omega)}} < \infty$, and so $\omega \mapsto \Pi_{U_{i,P}(\omega) || V_{i,P}(\omega)}$ is strongly measurable by Lemma \ref{lemma:limit_of_strong_measurable}.
  \end{proof}
\end{lemma}

\begin{lemma}\label{lemma:perturbed_oseledets_blocks}
  For each $i \in \{1, \dots, k_Q\}$ the map $\omega \mapsto \Pi_{G_{i,P}(\omega) || H_{i,P}(\omega)}$ is strongly measurable.
  \begin{proof}
    The cases where $i =1$ is covered by Lemma \ref{lemma:measurable_hyperbolic_perturbed_spaces}.
    For $1 < i \le k_Q$ we have
    \begin{equation}\label{eq:perturbed_oseledets_blocks_1}
      \Pi_{G_{i,P}(\omega) || H_{i,P}(\omega)} = \Pi_{U_{i,P}(\omega) || V_{i, P}(\omega) } \Pi_{V_{i-1,P}(\omega) || U_{i-1, P}(\omega) },
    \end{equation}
    and so $\omega \mapsto \Pi_{G_{i,P}(\omega) || H_{i,P}(\omega)}$ is strongly measurable by Lemma \ref{lemma:measurable_hyperbolic_perturbed_spaces} and \cite[Lemma A.5]{GTQuas1}.
  \end{proof}
 \end{lemma}

A key tool in the proof of Proposition \ref{prop:perturbed_oseledets_existence} is the following result on the existence of measurable change of basis maps, which allows to reduce our setting to the classical one of an invertible cocycle on a finite dimensional vector space.
We note that a similar construction is carried out in \cite[Chapter 7]{lian2010lyapunov}.
We defer the proof to Appendix \ref{app:stability_lyapunov_proofs}.

\begin{lemma}\label{lemma:measurable_change_of_basis}
  If $(\Omega, \mathcal{F}, \mathbb{P})$ is a Lebesgue space, $X$ is a separable Banach space, $d \in \Z^+$ and $\omega \mapsto \Pi_\omega$ is a strongly measurable map such that each $\Pi_\omega$ is rank-$d$ projection and $\esssup_{\omega \in \Omega} \norm{\Pi_{\omega}} < \infty$ then for every $\epsilon > 0$ there exists a strongly measurable map $A : \Omega \to \LL(X, \C^d)$ such that $\restr{A_\omega}{\Pi_{\omega}(X)} : \Pi_\omega(X) \to \C^d$ is a bijection, $\ker(A_\omega) = \ker(\Pi_\omega)$ for a.e. $\omega$, and the map $\omega \mapsto \left(\restr{A_\omega}{\Pi_{\omega}(X)}\right)^{-1}$ is strongly measurable.
  Moreover,
  \begin{equation}\label{eq:measurable_change_of_basis_0}
    \esssup_{\omega \in \Omega} \norm{\restr{A_\omega}{\Pi_{\omega}(X)}} \le \left(\frac{2}{1-\epsilon}\right)^{d-1},
  \end{equation}
  and
  \begin{equation}\label{eq:measurable_change_of_basis_1}
    \esssup_{\omega \in \Omega} \norm{\left(\restr{A_\omega}{\Pi_{\omega}(X)}\right)^{-1}} \le \sqrt{d}.
  \end{equation}
\end{lemma}

\begin{remark}
  The bound \eqref{eq:measurable_change_of_basis_0}, while sufficient for our purposes, is likely an artefact of the proof - one typically expects a bound comparable to $\sqrt{d}$.
\end{remark}

\begin{proof}[{The proof of Proposition \ref{prop:perturbed_oseledets_existence}}]
  For each $i \in \{1, \dots, k_Q\}$ let $A_{i,\omega}$ denote the map produced by applying Lemma \ref{lemma:measurable_change_of_basis} to $\omega \mapsto \Pi_{G_{i,P}(\omega) || H_{i,P}(\omega)}$ with $\epsilon$ very small, and set
  \begin{equation*}
    P_{i,\omega} = A_{i, \sigma(\omega)} P_\omega \left(\restr{A_{i,\omega}}{G_{i,P}(\omega)} \right)^{-1}.
  \end{equation*}
  Then $\mathcal{P}_i = (\Omega, \mathcal{F}, \mathbb{P}, \sigma, \C^{d_{i,Q}},\omega \mapsto P_{i,\omega})$ is a separable strongly measurable random linear system with an ergodic invertible base.
  Moreover, by \eqref{eq:fast_oseledets} and the estimates in Lemma \ref{lemma:measurable_change_of_basis} for every $n \in \Z^+$ and $i \in \{1, \dots, k_Q\}$ we have
  \begin{equation}\begin{split}\label{eq:perturbed_oseledets_existence_1}
    \norm{\left(P_{i,\omega}^{n}\right)^{-1}} &\le \norm{A_{i, \omega}} \norm{\left(\restr{P_\omega^n}{G_{i,P}(\omega)}\right)^{-1}} \norm{\left(\restr{A_{i,\sigma^{n}(\omega)}}{G_{i,P}(\sigma^{n}(\omega))}\right)^{-1}} \\
    &\le C_0 \sqrt{d} \left(\frac{2}{1-\epsilon}\right)^{d-1} e^{-n(\lambda_{i,Q} - \eta -\beta_0)}.
  \end{split}\end{equation}
  On the other hand, by \eqref{eq:slow_oseledets} and the estimates in Lemma \ref{lemma:measurable_change_of_basis} for every $n \in \Z^+$ and $1 < i \le k_Q$ we have
  \begin{equation}\label{eq:perturbed_oseledets_existence_2}
    \norm{P_{i, \omega}^n} \le \norm{A_{i, \sigma^n(\omega)}} \norm{\restr{P_\omega^n}{G_{i,P}(\omega)}} \norm{\left(\restr{A_{i,\omega}}{G_{i,P}(\omega)}\right)^{-1}} \le C_0 \sqrt{d} \left(\frac{2}{1-\epsilon}\right)^{d-1} e^{n(\lambda_{i,Q} + \eta + \beta_0)},
  \end{equation}
  while for $i = 1$ we have
  \begin{equation}\label{eq:perturbed_oseledets_existence_3}
    \norm{P_{i, \omega}^n} \le C_0 C_3 \sqrt{d} \left(\frac{2}{1-\epsilon}\right)^{d-1} R^{n}.
  \end{equation}
  Thus $\log^+ \norm{P_{i, \omega}^{\pm 1}} \in L^1(\Omega, \mathcal{F}, \mathbb{P})$, and so by Oseledets' Multiplicative Ergodic Theorem \cite{oseledets1968multiplicative}, each $\mathcal{P}_i$ has an Oseledets splitting of dimension $d_{i,Q}$ given by
  \begin{equation}\label{eq:perturbed_oseledets_existence_4}
    \C^d = \bigoplus_{j=1}^{k_{P_i} } E_{j, P_i}(\omega).
  \end{equation}
  By pulling back these Oseledets spaces to $X$ we obtain for each $i \in \{1, \dots, k_Q\}$ and a.e. $\omega$ the splitting
  \begin{equation*}
     G_{i,P}(\omega) = \bigoplus_{j=1}^{k_{P_i} } \left(\restr{A_{i,\omega}}{G_{i,P}(\omega)} \right)^{-1} E_{j, P_i}(\omega),
  \end{equation*}
  and so in view of \eqref{eq:splitting_perturbed_blocks} we have
  \begin{equation}\label{eq:perturbed_oseledets_existence_5}
    X = \left(\bigoplus_{i=1}^{k_Q} \left(\bigoplus_{j=1}^{k_{P_i} } \left(\restr{A_{i,\omega}}{G_{i,P}(\omega)} \right)^{-1} E_{j, P_i}(\omega) \right) \right) \oplus V_{k_Q, P}(\omega).
  \end{equation}
  Let $k_{P} = \sum_{i=1}^{k_Q} k_{P_i}$. For $1 \le \ell \le k_P$ set $h(\ell) = \max \{ \sum_{i=1}^t k_{P_{i}} : \sum_{i=1}^t k_{P_{i}} \le \ell \}$, $g(\ell) = \ell - h(\ell)$ and
  \begin{equation*}
     E_{\ell, P}(\omega) = \left(\restr{A_{h(\ell),\omega}}{G_{h(\ell),P}(\omega)} \right)^{-1} E_{g(\ell), P_{h(\ell)}}(\omega).
  \end{equation*}
  If we set $F_{P}(\omega) = V_{k_Q, P}(\omega)$ then we may rewrite \eqref{eq:perturbed_oseledets_existence_5} as
  \begin{equation}\label{eq:perturbed_oseledets_existence_6}
      X = \left(\bigoplus_{\ell = 1}^{k_P} E_{\ell, P}(\omega) \right) \oplus F_P(\omega).
  \end{equation}
  We claim that \eqref{eq:perturbed_oseledets_existence_6} is an Oseledets splitting for $\mathcal{P}$ of dimension $d$.

  \paragraph{The strong measurability of the Oseledets projections.}
  The projection onto $F_P(\omega)$ according to \eqref{eq:perturbed_oseledets_existence_6} is strongly measurable by Lemma \ref{lemma:measurable_hyperbolic_perturbed_spaces}.
  The projection onto each $E_{\ell, P}(\omega)$ according to the decomposition \eqref{eq:perturbed_oseledets_existence_5} is given by
  \begin{equation*}
    \left(\restr{A_{h(\ell),\omega}}{G_{h(\ell),P}(\omega)}\right)^{-1} \Pi_{g(\ell),h(\ell),\omega} A_{h(\ell),\omega}\Pi_{G_{h(\ell),P}(\omega) || H_{h(\ell),P}(\omega)},
  \end{equation*}
  where $\Pi_{g(\ell),h(\ell),\omega}$ denotes the projection onto $E_{g(\ell), P_{h(\ell)}}(\omega)$ according to the splitting in \eqref{eq:perturbed_oseledets_existence_4}.
  Thus the projection onto $E_{\ell, P}(\omega)$ according to the decomposition \eqref{eq:perturbed_oseledets_existence_6}, being the composition of strongly measurable maps, is strongly measurable by \cite[Lemma A.5]{GTQuas1}.

  \paragraph{The properties of the fast Oseledets spaces.}
  It is easily checked that for any $\ell \in \{1, \dots, k_P\}$ and a.e. $\omega$ we have
  \begin{equation*}\begin{split}
    P_\omega^n (E_{\ell, P}(\omega)) &= P_\omega^n \left(\left(\restr{A_{h(\ell),\omega}}{G_{h(\ell),P}(\omega)} \right)^{-1}E_{g(\ell),P_{h(\ell)}}(\omega)\right) \\
    &= \left(\restr{A_{h(\ell),\sigma^{n}(\omega)}}{G_{h(\ell),P}(\sigma^{n}(\omega))} \right)^{-1} E_{g(\ell),P_{h(\ell)}}(\sigma^{n}(\omega)) = E_{\ell, P}(\sigma^{n}(\omega)).
  \end{split}\end{equation*}
  In addition, due to the bounds \eqref{eq:measurable_change_of_basis_0} and \eqref{eq:measurable_change_of_basis_1} we have for a.e. $\omega $ and every non-zero $v \in E_{\ell, P}(\omega)$ that $A_{h(\ell),\omega}v \in E_{g(\ell), P_{h(\ell)}}(\omega)$ and so
  \begin{equation*}
     \lim_{n \to \infty} \frac{1}{n} \log \norm{P_{\omega}^n v} = \lim_{n \to \infty} \frac{1}{n} \log \norm{\left(\restr{A_{h(\ell), \sigma^{n}(\omega)}}{G_{h(\ell), P}(\sigma^{n}(\omega)} \right)^{-1}\left(\restr{P_{h(\ell),\omega}^n}{E_{g(\ell), P_{h(\ell)}}(\omega)} \right)A_{h(\ell),\omega}v} = \lambda_{g(\ell), P_{h(\ell)}}.
  \end{equation*}

  \paragraph{The ordering of the Lyapunov exponents.}
  For each $\ell \in \{1, \dots, k_P\}$ we set $\lambda_{\ell, P} = \lambda_{g(\ell), P_{h(\ell)}}$ so that $\lambda_{\ell, P}$ is the Lyapunov exponent associated to the spaces $\{E_{\ell, P}(\omega)\}_{\omega \in \Omega}$.
  Clearly $\lambda_{\ell_1, P} < \lambda_{\ell_2,P}$ whenever $\ell_1 > \ell_2$ and $h(\ell_1) = h(\ell_2)$, since then $g(\ell_1) > g(\ell_2)$ and so  $\lambda_{\ell_1, P} = \lambda_{g(\ell_1), P_{h(\ell_1)}} < \lambda_{g(\ell_2), P_{h(\ell_2)}} = \lambda_{\ell_2, P}$.
  On the other hand, if $\ell_1 > \ell_2$ and $h(\ell_1) \ne h(\ell_2)$ then since $\eta + \beta_0 < 2^{-1}\min_{1 \le i \le k_Q} \{\lambda_{i,Q} - \lambda_{i+1,Q}\}$ we may use \eqref{eq:perturbed_oseledets_existence_1} and \eqref{eq:perturbed_oseledets_existence_2} to conclude that
  \begin{equation*}
    \lambda_{\ell_1, P} \le \lambda_{h(\ell_1),Q} + \eta + \beta_0 < \lambda_{h(\ell_2),Q} - \eta - \beta_0 \le  \lambda_{\ell_2 , P}.
  \end{equation*}
  Thus $\lambda_{1, P} > \lambda_{2, P} > \dots > \lambda_{k,P}$.

  \paragraph{The properties of the slow Oseledets spaces.}
  That $P_\omega F_P(\omega) \subseteq F(\sigma(\omega))$ a.e. follows from our application of Theorem \ref{thm:stability_cocycle} in the construction of $V_{k_Q, P}(\omega)$.
  By \eqref{eq:perturbed_oseledets_existence_2} we have a.e. that
  \begin{equation}\label{eq:perturbed_oseledets_existence_7}
    \lim_{n \to \infty} \frac{1}{n} \log \norm{\restr{P_{\omega}^n}{V_{k_Q, P}(\omega)}} := \mu_{P} \le \lambda_{k_Q+1, Q} + \eta + \beta_0.
  \end{equation}
  By \eqref{eq:perturbed_oseledets_existence_1} we get $\lambda_{k_P,P} > \lambda_{k_Q,Q} - \eta - \beta_0$.
  Since $\eta + \beta_0 < 2^{-1}\min_{1 \le i \le k_Q} \{\lambda_{i,Q} - \lambda_{i+1,Q}\}$ it follows that $\mu_{P} < \lambda_{k_P,P}$.

  \paragraph{The identity \eqref{eq:perturbed_oseledets_existence_0}.}
  If we set $s(i) = \sum_{t=1}^{i-1} k_{P_t}$ then
  \begin{equation}\label{eq:perturbed_oseledets_existence_8}
      G_{i,P}(\omega) = \bigoplus_{j=1}^{k_{P_i} } E_{j + s(i), P}(\omega).
  \end{equation}
  Then for $j \in \{1, \dots, k_{P_i}\}$ we have
  \begin{equation}\begin{split}\label{eq:perturbed_oseledets_existence_9}
    \sum_{1 \le \ell \le j + s(i)} d_{\ell,P} &= \left(\sum_{1 \le \ell \le s(i)} d_{\ell,P}\right) + \left(\sum_{1 \le \ell \le j} d_{s(i) + \ell,P} \right)\\
    &= \left(\sum_{1 \le t < i} \sum_{1 \le m \le k_{P_t}} d_{s(t) + m,P}\right) + \left(\sum_{1 \le \ell \le j} d_{s(i) + \ell,P} \right).
  \end{split}\end{equation}
  Since $d_{s(t) + m,P} = d_{m, P_{t}}$ we get
  \begin{equation}\label{eq:perturbed_oseledets_existence_10}
      \sum_{1 \le t < i} \sum_{1 \le m \le k_{P_t}} d_{s(t) + m,P} = \sum_{1 \le t < i} d_{t,Q},\quad \text{ and } \quad  0 < \sum_{1 \le \ell \le j} d_{s(i) + \ell,P} < d_{i,Q}.
  \end{equation}
  Thus by combining \eqref{eq:perturbed_oseledets_existence_9} and \eqref{eq:perturbed_oseledets_existence_10} we see that $j + s(i) \in S(i)$.
  Running our argument in reverse, we observe that if $\ell \in S(i)$ then $h(\ell) = s(i)$ and so $\ell = g(\ell) + s(i)$ with $g(\ell) \in \{1, \dots, k_{P_i}\}$.
  Thus we obtain \eqref{eq:perturbed_oseledets_existence_0} by re-indexing the direct sum \eqref{eq:perturbed_oseledets_existence_8}.
\end{proof}

\begin{proof}[{The first part of the proof of Theorem \ref{thm:stability_lyapunov}}]
  As per Proposition \ref{prop:perturbed_oseledets_existence}, if $P \in \mathcal{LY}(C_1, C_2, r, R) \cap \mathcal{O}_{\epsilon_0}(Q)$ then $\mathcal{P}$ has an Oseledets splitting of dimension $d$.
  Set $c_0 = \beta_0 + \eta$.
  From the proof of Proposition \ref{prop:perturbed_oseledets_existence}, and in particular the estimates \eqref{eq:perturbed_oseledets_existence_1}, \eqref{eq:perturbed_oseledets_existence_2} and \eqref{eq:perturbed_oseledets_existence_3}, we have for every $i \in \{1, \dots, k_Q\}$ that
  \begin{equation}\label{eq:stability_lyapunov}
    \{\lambda_{j, P} : j \in S(i) \} \subseteq I_i = (\lambda_{i,Q} - c_0 , \max\{ \lambda_{i,Q}, \log(\delta_{i1} R)\} + c_0).
  \end{equation}
  Moreover, by \eqref{eq:perturbed_oseledets_existence_7} and the ensuing discussion, we have $I_i \subseteq (\mu_P, \infty)$ for each $i$.
  Thus, $\partial I_{i_1} \cap \partial I_{i_2} = \emptyset$ whenever $i_1 \ne i_2$. As for every $j \in \{1, \dots, k_P\}$ we have $j \in S(i)$ for some $i \in \{1, \dots, k_Q\}$, it follows that $\partial I_i \cap \{\lambda_{j, P} : 1\le j \le k_P\} = \emptyset$ for every $i$.
  Hence each $I_i$ separates the Lyapunov spectrum of $\mathcal{P}$.
  In view of Proposition \ref{prop:perturbed_oseledets_existence} we therefore have $\Pi_{I_i,P}(\omega) = \Pi_{G_{i,P}(\omega) || H_{i,P}(\omega)}$, and so we obtain \eqref{eq:stability_lyapunov_0000} upon recalling that $\dim(G_{i,P}(\omega)) = d_{i,Q}$ for a.e. $\omega \in \Omega$.
  We get \eqref{eq:stability_lyapunov_00000} and \eqref{eq:stability_lyapunov_00} from \eqref{eq:perturbed_oseledets_blocks_bound_0} and \eqref{eq:perturbed_oseledets_blocks_bound_00}, respectively, in Proposition \ref{prop:perturbed_oseledets_blocks_bound}.
  Finally, as $V_{k_Q,Q}(\omega) = F_{Q}(\omega)$ and $V_{k_P,P} = F_P(\omega)$ we obtain \eqref{eq:stability_lyapunov_000} from \eqref{eq:perturbed_oseledets_blocks_bound_000} in Proposition \ref{prop:perturbed_oseledets_blocks_bound}.
\end{proof}

\subsection{Convergence of the Lyapunov exponents}\label{sec:Lyapunov_converge}

In this section we focus on the proving the estimate \eqref{eq:stability_lyapunov_0}. A key tool in our proof will be a generalisation of the determinant to operators on Banach spaces, which we will use to access the Lyapunov exponents of $\mathcal{P}$. When $E \in \mathcal{G}(X)$ is finite dimensional we denote by $m_E$ the Haar measure on $E$, normalised so that $m_E(B_E)$ has the measure of the $\dim(E)$-dimensional Euclidean unit ball.
For each $d \in \Z^+$ we define a map $\det : \LL(X) \times \mathcal{G}_d(X) \to \R$ by
\begin{equation}\label{eq:det}
  \det(A | E) = \frac{m_{AE}(A(B_E))}{m_E(B_E)}.
\end{equation}
We refer the reader to \cite[Section 2.2]{blumenthal2016volume} for an overview of the basic properties of the determinant.

\begin{lemma}\label{lemma:measurability_determinant}
  Recall $\epsilon_0$ from Proposition \ref{prop:perturbed_oseledets_blocks_bound}.
  For every $P \in \mathcal{LY}(C_1, C_2, r, R) \cap \mathcal{O}_{\epsilon_0}(Q)$, $n \in \Z^+$ and $\ell \in \{1, \dots, k_P\}$ the maps $\omega \mapsto \log\det(P^n_\omega | E_{\ell,P}(\omega))$, $\omega \mapsto \log \norm{\restr{P^n_\omega}{E_{\ell,P}(\omega)}}$ and $\omega \mapsto \log \norm{\left(\restr{P^n_\omega}{E_{\ell,P}(\omega)}\right)^{-1}}^{-1}$ are $(\mathcal{F}, \mathcal{B}_{\R})$-measurable and in $L^1(\Omega, \mathcal{F}, \mathbb{P})$.
  \begin{proof}
    Fix $n$ and $\ell$. Let $\psi : \Omega \to \LL(X) \times \mathcal{G}_{d_{\ell,P}}(X)$ be defined by $\psi(\omega) = (P^n_\omega, E_{\ell,P}(\omega))$. The map $\omega \mapsto P^n_\omega$ is strongly measurable by \cite[Lemma A.5]{GTQuas1}.
    On the other hand, the projection onto $E_{\ell,P}(\omega)$ is strongly measurable since it is an Oseledets space, and so $\omega \mapsto E_{\ell,P}(\omega)$ is $(\mathcal{F}, \mathcal{B}_{\mathcal{G}(X)})$-measurable by Lemma \ref{lemma:strong_cont_proj_grassmannian}.
    Thus $\psi$ is $(\mathcal{F}, \mathcal{S} \times \mathcal{B}_{\mathcal{G}(X)})$-measurable.
    That $\omega \mapsto \log \det(\psi(\omega))$ is $(\mathcal{F}, \mathcal{B}_{\R})$-measurable follows from
    Proposition \ref{prop:measurable_determinant}, while the $(\mathcal{F}, \mathcal{B}_{\R})$-measurability of $\omega \mapsto \log \norm{\restr{P^n_\omega}{E_{\ell,P}(\omega)}}$ is a consequence of \cite[Lemma B.16]{GTQuas1}.
    To see that $\omega \mapsto \log \norm{\left(\restr{P^n_\omega}{E_{\ell,P}(\omega)}\right)^{-1}}$ is measurable we note that
    \begin{equation}\label{eq:measurability_determinant_1}
        \norm{\left(\restr{P^n_\omega}{E_{\ell,P}(\omega)}\right)^{-1}} = \norm{\restr{\left(\restr{\Pi_{E_{\ell,P}(\sigma^{n}(\omega))} P^n_\omega}{E_{\ell,P}(\omega)}\right)^{-1} \Pi_{E_{\ell,P}(\sigma^{n}(\omega))}} {E_{\ell,P}(\sigma^{n}(\omega))}},
    \end{equation}
    where $\Pi_{E_{\ell,P}(\sigma^{n}(\omega))}$ denotes the projection onto $E_{\ell,P}(\sigma^{n}(\omega))$ according to the Oseledets decomposition for $\mathcal{P}$.
    The map $\omega \mapsto \left(\restr{\Pi_{E_{\ell,P}(\sigma^{n}(\omega))} P^n_\omega}{E_{\ell,P}(\omega)}\right)^{-1} \Pi_{E_{\ell,P}(\sigma^{n}(\omega))}$ is $(\mathcal{F}, \mathcal{B}_{\R})$-measurable by Proposition \ref{prop:measurable_inverse}. Hence the right hand side of \eqref{eq:measurability_determinant_1} is $(\mathcal{F}, \mathcal{B}_{\R})$-measurable by \cite[Lemma B.16]{GTQuas1}, which of course implies that the left hand side of \eqref{eq:measurability_determinant_1} is $(\mathcal{F}, \mathcal{B}_{\R})$-measurable.

    Since
    \begin{equation*}
        \norm{\left(\restr{P^n_\omega}{E_{\ell,P}(\omega)}\right)^{-1}}^{-1} B_{E_{\ell,P}(\sigma^{n}(\omega))} \subseteq P_\omega^n B_{E_{\ell,P}(\omega)} \subseteq \norm{\restr{P^n_\omega}{E_{\ell,P}(\omega)}}B_{E_{\ell,P}(\sigma^{n}(\omega))},
    \end{equation*}
    we have
    \begin{equation}\label{eq:measurability_determinant_2}
        \log\norm{\left(\restr{P^n_\omega}{E_{\ell,P}(\omega)}\right)^{-1}}^{-1} \le \frac{1}{d_{\ell,P}} \log \det(P^n_\omega | E_{\ell,P}(\omega)) \le \log \norm{\restr{P^n_\omega}{E_{\ell,P}(\omega)}}.
    \end{equation}
    By \eqref{eq:fast_oseledets} and Proposition \ref{prop:perturbed_oseledets_existence} we have
    \begin{equation}\label{eq:measurability_determinant_3}
        C_0^{-1} e^{n(\lambda_{k_Q,Q} -\eta - \beta_0)} \le \norm{\left(\restr{P^n_\omega}{E_{\ell,P}(\omega)}\right)^{-1}}^{-1}.
    \end{equation}
    On the other hand, since $P \in \mathcal{LY}(C_1, C_2, r, R)$ we have
    \begin{equation}\label{eq:measurability_determinant_4}
        \norm{\restr{P^n_\omega}{E_{\ell,P}(\omega)}} \le C_3 R^{n}.
    \end{equation}
    Since $(\Omega, \mathcal{F}, \mathbb{P})$ is a probability space, upon combining \eqref{eq:measurability_determinant_2}, \eqref{eq:measurability_determinant_3} and \eqref{eq:measurability_determinant_4} we get that $\omega \mapsto \log\det(P^n_\omega | E_{\ell,P}(\omega))$, $\omega \mapsto \log \norm{\restr{P^n_\omega}{E_{\ell,P}(\omega)}}$ and $\omega \mapsto \log \norm{\left(\restr{P^n_\omega}{E_{\ell,P}(\omega)}\right)^{-1}}^{-1}$ are all contained in $L^1(\Omega, \mathcal{F}, \mathbb{P})$.
  \end{proof}
\end{lemma}

\begin{proposition}\label{prop:exponent_formula}
  Recall $\epsilon_0$ from Proposition \ref{prop:perturbed_oseledets_blocks_bound}.
  For every $P \in \mathcal{LY}(C_1, C_2, r, R) \cap \mathcal{O}_{\epsilon_0}(Q)$, $n \in \Z^+$ and $\ell \in \{1, \dots, k_{P}\}$ we have
  \begin{align}
    \lambda_{\ell,P} &= \frac{1}{n d_{\ell,P} } \intf_{\Omega} \log \det( P^n_\omega | E_{\ell,P}(\omega) ) d\mathbb{P}, \label{eq:exponent_formula_det}\\
    &= \lim_{m \to \infty} \frac{1}{m } \intf_{\Omega} \log \norm{\restr{P^{m}_\omega}{E_{\ell,P}(\omega)}} d\mathbb{P}, \label{eq:exponent_formula_norm}\\
    &= \lim_{m \to \infty} \frac{1}{m} \intf_{\Omega} \log \norm{\left(\restr{P^{m}_\omega}{E_{\ell,P}(\omega)}\right)^{-1}}^{-1} d\mathbb{P}. \label{eq:exponent_formula_conorm}
  \end{align}
  \begin{proof}
    Recalling \eqref{eq:measurability_determinant_2} from the proof of Lemma \ref{lemma:measurability_determinant}, we have for every $j \in \Z^+$ that
    \begin{equation*}
        \frac{1}{nj}\intf \log\norm{\left(\restr{P^{nj}_\omega}{E_{\ell,P}(\omega)}\right)^{-1}}^{-1} d\mathbb{P} \le \frac{1}{nj d_{\ell,P}} \intf \log \det(P^{nj}_\omega | E_{\ell,P}(\omega)) d\mathbb{P} \le \frac{1}{nj} \intf \log \norm{\restr{P^{nj}_\omega}{E_{\ell,P}(\omega)}} d\mathbb{P}.
    \end{equation*}
    Hence it suffices to prove that
    \begin{equation}\label{eq:exponent_norm_inequalities}
        \limsup_{m \to \infty} \frac{1}{m } \intf_{\Omega} \log \norm{\restr{P^{m}_\omega}{E_{\ell,P}(\omega)}} d\mathbb{P} \le \lambda_{\ell,P} \le \liminf_{m \to \infty} \frac{1}{m} \intf_{\Omega} \log \norm{\left(\restr{P^{m}_\omega}{E_{\ell,P}(\omega)}\right)^{-1}}^{-1} d\mathbb{P},
    \end{equation}
    and that
    \begin{equation}\label{eq:exponent_formula_det_1}
        \frac{1}{nj d_{\ell,P}} \intf \log \det(P^{nj}_\omega | E_{\ell,P}(\omega)) d\mathbb{P} = \frac{1}{n d_{\ell,P} } \intf_{\Omega} \log \det( P^n_\omega | E_{\ell,P}(\omega) ) d\mathbb{P}.
    \end{equation}

    \paragraph{The identity \eqref{eq:exponent_formula_det_1}.}
    Since the determinant is multiplicative \cite[Proposition 2.13]{blumenthal2016volume} we have
    \begin{equation}\label{eq:exponent_formula_1}
        \frac{1}{nj d_{\ell,P}} \intf \log \det(P^{nj}_\omega | E_{\ell,P}(\omega)) d\mathbb{P}  = \frac{1}{nj d_{\ell,P}} \sum_{i=0}^{j-1} \intf \log \det(P^{n}_{\sigma^{ni}(\omega)} | E_{\ell,P}(\sigma^{ni}(\omega))) d\mathbb{P}.
    \end{equation}
    Since $\mathbb{P}$ is $\sigma$-invariant, for $i \in \{0, \dots, j-1\}$ we have
    \begin{equation}\label{eq:exponent_formula_2}
        \intf \log \det(P^{n}_{\sigma^{ni}(\omega)} | E_{\ell,P}(\sigma^{ni}(\omega))) d\mathbb{P} = \intf \log \det(P^{n}_\omega | E_{\ell,P}(\omega)) d\mathbb{P}.
    \end{equation}
    Combining \eqref{eq:exponent_formula_1} and \eqref{eq:exponent_formula_2} yields \eqref{eq:exponent_formula_det_1}.

    \paragraph{The first inequality in \eqref{eq:exponent_norm_inequalities}.}
    By Lemma \ref{lemma:measurability_determinant} we have $\left\{\omega \mapsto \log\norm{\restr{P^{m}_\omega}{E_{\ell,P}(\omega)}} \right\}_{m \in \Z^+} \subseteq L^1(\Omega, \mathcal{F}, \mathbb{P})$. Since $\left\{\omega \mapsto \log \norm{\restr{P^{m}_\omega}{E_{\ell,P}(\omega)}} \right\}_{m \in \Z^+}$ is subadditive with respect to $\sigma$ and as $\sigma$ is $\mathbb{P}$-ergodic, by Kingman's subadditive ergodic theorem we have for a.e. $\omega$ that
    \begin{equation}\label{eq:exponent_formula_3}
        \lim_{m \to \infty} \frac{1}{m } \log \norm{\restr{P^{m}_\omega}{E_{\ell,P}(\omega)}} = \lim_{m \to \infty} \frac{1}{m } \intf_{\Omega} \log \norm{\restr{P^{m}_\omega}{E_{\ell,P}(\omega)}} d\mathbb{P}.
    \end{equation}
    Fix a normalised Auerbach basis $\{v_i\}_{i=1}^{d_{\ell,P}}$ for $E_{\ell,P}(\omega)$, and for each $m$ let $v_{m} \in E_{\ell,P}(\omega)$ satisfy $\norm{w_m} = 1$ and $\norm{\restr{P^{m}_\omega}{E_{\ell,P}(\omega)}} = \norm{P^{m}_\omega w_m}$.
    If we write $w_m = \sum_{i=1}^{d_{\ell,P}} a_{i,m} v_{i}$ then
    \begin{equation*}
        \norm{\restr{P^{m}_\omega}{E_{\ell,P}(\omega)}} \le \left(\max_{i \in \{1, \dots, d_{\ell,P}\}} \norm{P^{m}_\omega v_i}\right) \sum_{i=1}^{d_{\ell,P}} \abs{a_{i,m}}.
    \end{equation*}
    Since $\{v_i\}_{i=1}^{d_{\ell,P}}$ is Auerbach, by \cite[Corollary A.7]{quas2019explicit} we have $\sum_{i=1}^{d_{\ell,P}} \abs{a_{i,m}} \le d_{\ell,P}$, and so
    \begin{equation*}
       \lim_{m \to \infty} \frac{1}{m } \log \norm{\restr{P^{m}_\omega}{E_{\ell,P}(\omega)}} \le \max_{i \in \{1, \dots, d_{\ell,P}\}}\left( \limsup_{m \to \infty} \frac{1}{m} \log \left( \norm{P^{m}_\omega v_i}\right)\right) = \lambda_{\ell,P},
    \end{equation*}
    which, in view of \eqref{eq:exponent_formula_3}, yields the first inequality in \eqref{eq:exponent_norm_inequalities}.

    \paragraph{The second inequality in \eqref{eq:exponent_norm_inequalities}.}
    The proof is very similar to the one in the previous paragraph.
    By Lemma \ref{lemma:measurability_determinant} we have $\left\{\omega \mapsto \log \norm{\left(\restr{P^{m}_{\omega}}{E_{\ell,P}(\omega)}\right)^{-1}} \right\}_{m \in \Z^+} \subseteq L^1(\Omega, \mathcal{F}, \mathbb{P})$. Since $\left\{\omega \mapsto \log \norm{\left(\restr{P^{m}_{\omega}}{E_{\ell,P}(\omega)}\right)^{-1}} \right\}_{m \in \Z^+}$ is subadditive with respect to $\sigma$, and as $\sigma$ is invertible and $\mathbb{P}$-ergodic, by Kingman's subadditive ergodic theorem we have for a.e. $\omega$ that
    \begin{equation}\label{eq:exponent_formula_4}
        \lim_{m \to \infty} \frac{1}{m } \log \norm{\left(\restr{P^{m}_{\sigma^{-m}(\omega)}}{E_{\ell,P}(\sigma^{-m}(\omega))}\right)^{-1}} = \lim_{m \to \infty} \frac{1}{m } \intf \log \norm{\left(\restr{P^{m}_{\omega}}{E_{\ell,P}(\omega)}\right)^{-1}} d\mathbb{P}.
    \end{equation}
    Fix a normalised Auerbach basis $\{v_i\}_{i=1}^{d_{\ell,P}}$ for $E_{\ell,P}(\omega)$, and for each $m$ let $w_{m} \in E_{\ell,P}(\omega)$ satisfy $\norm{w_m} = 1$ and
    \begin{equation*}
        \norm{\left(\restr{P^{m}_{\sigma^{-m}(\omega)}}{E_{\ell,P}(\sigma^{-m}(\omega))}\right)^{-1}} = \norm{\left(\restr{P^{m}_{\sigma^{-m}(\omega)}}{E_{\ell,P}(\sigma^{-m}(\omega))}\right)^{-1} w_{m}}.
    \end{equation*}
    If we write $w_m = \sum_{i=1}^{d_{\ell,P}} a_{i,m} v_{i}$ then
    \begin{equation*}
        \norm{\left(\restr{P^{m}_{\sigma^{-m}(\omega)}}{E_{\ell,P}(\sigma^{-m}(\omega))}\right)^{-1}} \le \left(\max_{i \in \{1, \dots, d_{\ell,P}\}} \norm{\left(\restr{P^{m}_{\sigma^{-m}(\omega)}}{E_{\ell,P}(\sigma^{-m}(\omega))}\right)^{-1} v_{i}}\right) \sum_{i=1}^{d_{\ell,P}} \abs{a_{i,m}}.
    \end{equation*}
    Since $\{v_i\}_{i=1}^{d_{\ell,P}}$ is Auerbach, by \cite[Corollary A.7]{quas2019explicit} we have $\sum_{i=1}^{d_{\ell,P}} \abs{a_{i,m}} \le d_{\ell,P}$, and so
    \begin{equation*}\begin{split}
       \lim_{m \to \infty} \frac{1}{m} \log \norm{\left(\restr{P^{m}_{\sigma^{-m}(\omega)}}{E_{\ell,P}(\sigma^{-m}(\omega))}\right)^{-1}} &\le \max_{i \in \{1, \dots, d_{\ell,P}\}}\left(\limsup_{m \to \infty} \frac{1}{m } \log \left( \norm{\left(\restr{P^{m}_{\sigma^{-m}(\omega)}}{E_{\ell,P}(\sigma^{-m}(\omega))}\right)^{-1} v_{i}}\right)\right) \\
       &= -\lambda_{\ell,P},
    \end{split}\end{equation*}
    which, in view of \eqref{eq:exponent_formula_4}, yields the second inequality in \eqref{eq:exponent_norm_inequalities}.
  \end{proof}
\end{proposition}

Throughout the proof of Theorem \ref{thm:stability_lyapunov} we will use the following corollary of Lemma \ref{lemma:fast_space_eccentricity}, which is obtained by applying Lemma \ref{lemma:fast_space_eccentricity} to $P \in \mathcal{LY}(C_1, C_2, r, R) \cap \mathcal{O}_{\epsilon_0}(Q)$ with the hyperbolic splitting consisting of fast spaces $\{U_{k_Q, P}(\omega)\}_{\omega \in \Omega}$ and slow spaces $\{V_{k_Q, P}(\omega)\}_{\omega \in \Omega}$.

\begin{lemma}\label{lemma:oseledets_eccentricity}
  Recall $\epsilon_0$ from Proposition \ref{prop:perturbed_oseledets_blocks_bound}.
  There exists $K > 0$ so that for every $P \in \mathcal{LY}(C_1, C_2, r, R) \cap \mathcal{O}_{\epsilon_0}(Q)$, a.e. $\omega$ and every $v \in \bigoplus_{i=1}^{k_P} E_{i, P}(\omega)$ we have $\norm{v} \le K \wnorm{v}$.
\end{lemma}

We may now finish the proof of Theorem \ref{thm:stability_lyapunov}.
For the sake of brevity, throughout the proof we use $\Pi_{E_{i,Q}(\omega)}$ to denote the projection onto $E_{i,Q}(\omega)$ according to the Oseledets splitting of $Q$, and $\Pi_{G_{i,P}(\omega)}$ to denote the projection onto $G_{i,P}(\omega)$ according to the splitting in \eqref{eq:splitting_perturbed_blocks}.

\begin{proof}[{The second part of the proof of Theorem \ref{thm:stability_lyapunov}}]
    It remains to prove the bound \eqref{eq:stability_lyapunov_0}.
    Let $\ell \in \{1, \dots, d\}$ and note that it suffices to produce for each $\nu > 0$ a $\epsilon_{\nu,\ell}$ such that if $P \in \mathcal{LY}(C_1, C_2, r, R) \cap \mathcal{O}_{\epsilon_{\nu,\ell}}(Q)$ then $\abs{\gamma_{\ell, P} - \gamma_{\ell, Q}} \le \nu$.
    By Proposition \ref{prop:perturbed_oseledets_existence} we have $\gamma_{\ell, Q} = \lambda_{i, Q}$ for some $i \in \{1, \dots, k_Q\}$ and $\gamma_{\ell, P} = \lambda_{j, P}$ for some $j \in S(i)$.
    Recalling \eqref{eq:measurability_determinant_2} from the proof of Lemma \ref{lemma:measurability_determinant} we have for each $n \in \Z^+$ and a.e. $\omega$ that
    \begin{equation*}
        \frac{1}{n}\log \norm{\left(\restr{P^n_\omega}{E_{j,P}(\omega)}\right)^{-1}}^{-1} \le \frac{1}{n d_{j, P}} \log\det( P^n_\omega | E_{j,P}(\omega) ) \le \frac{1}{n} \log \norm{\restr{P^n_\omega}{E_{j,P}(\omega)}}.
    \end{equation*}
    Since $\wnorm{\cdot} \le \norm{\cdot}$ and by Lemma \ref{lemma:oseledets_eccentricity}, we get
    \begin{equation}\label{eq:stability_lyapunov_2}
        \frac{1}{n} \log\left(\inf_{v \in E_{j,P}(\omega) \cap B_{\norm{\cdot}}} \wnorm{P^n_\omega v}\right) \le \frac{1}{n d_{j, P}} \det( P^n_\omega | E_{j,P}(\omega) ) \le \frac{\log K}{n} + \frac{1}{n} \log\left(\sup_{v \in E_{j,P} \cap B_{\norm{\cdot}}} \wnorm{P^n_\omega v} \right).
    \end{equation}
    The rest of the proof will run as follows: we will first pursue some technical bounds, which we will then use to obtain lower and upper bounds in terms of $\lambda_{i,Q}$ for the left-most and right-most terms, respectively, in \eqref{eq:stability_lyapunov_2}.

    \paragraph{Some technical bounds.}
    For every $v \in E_{j,P}(\omega) \cap B_{\norm{\cdot}}$ we have
    \begin{equation}\begin{split}\label{eq:stability_lyapunov_3}
        \abs{\frac{\wnorm{Q^n_\omega \Pi_{E_{i,Q}(\omega)}v}}{\wnorm{P^n_\omega v}} - 1} &= \wnorm{P^n_\omega v}^{-1} \abs{\wnorm{Q^n_\omega \Pi_{E_{i,Q}(\omega)}v} - \wnorm{P^n_\omega v}}
        \\
        &\le \wnorm{P^n_\omega v}^{-1} \left( \tnorm{\restr{Q^n_\omega( \Id - \Pi_{E_{i,Q}(\omega)})}{E_{j,P}(\omega)}} + \tnorm{Q^n_\omega - P_\omega^n}\right).
    \end{split}\end{equation}
    By \eqref{eq:fast_oseledets}, Lemma \ref{lemma:oseledets_eccentricity}, and as $\eta + \beta_0 < 2^{-1}\min_{1 \le i \le k_Q} \{\lambda_{i,Q} - \lambda_{i+1,Q}\}$, we have
    \begin{equation}\label{eq:stability_lyapunov_4}
        \wnorm{P^n_\omega v}^{-1} \le K \norm{P^n_\omega v}^{-1} \le K C_0 e^{-n\lambda_{k_Q+1,Q}}.
    \end{equation}
    By Lemma \ref{lemma:equicont_power} we have $Q^n \in \End_S(\mathbb{X}, \sigma)$, and so for every $n \in \Z^+$ and $\kappa > 0$ there exists $C_{\kappa,n}$ such that
    \begin{equation}\begin{split}\label{eq:stability_lyapunov_5}
        \tnorm{\restr{Q^n_\omega( \Id - \Pi_{E_{i,Q}(\omega)})}{E_{j,P}(\omega)}} &= \tnorm{Q^n_\omega( \Pi_{G_{i,P}(\omega)} - \Pi_{E_{i,Q}(\omega)})} \\
        &\le \kappa \norm{\Pi_{G_{i,P}(\omega)} - \Pi_{E_{i,Q}(\omega)}} + C_{\kappa, n} \tnorm{\Pi_{G_{i,P}(\omega)} - \Pi_{E_{i,Q}(\omega)}},
    \end{split}\end{equation}
    where we also used the fact that $\restr{\Pi_{G_{i,P}(\omega)}}{E_{j,P}(\omega)} = \restr{\Id}{E_{j,P}(\omega)}$.
    From the proof of Proposition  \ref{prop:perturbed_oseledets_blocks_bound} we have $\norm{\Pi_{G_{i,P}(\omega)} - \Pi_{E_{i,Q}(\omega)}} \le 2\Theta_0^2$.
    Thus, by applying \eqref{eq:stability_lyapunov_4}, \eqref{eq:stability_lyapunov_5} to \eqref{eq:stability_lyapunov_3} we obtain
    \begin{equation*}
        \abs{\frac{\wnorm{Q^n_\omega \Pi_{E_{i,Q}(\omega)}v}}{\wnorm{P^n_\omega v}} - 1} \le K C_0 e^{-n\lambda_{k_Q+1,Q}}\left(2\kappa \Theta_0^2 + C_{\kappa, n} \tnorm{\Pi_{G_{i,P}(\omega)} - \Pi_{E_{i,Q}(\omega)}} + \tnorm{Q^n_\omega - P_\omega^n}\right).
    \end{equation*}
    Fix $\gamma > 0$ and take $\kappa = \gamma K^{-1}e^{n\lambda_{k_Q+1,Q}}/(4\Theta_0^2 C_0)$.
    By Propositions \ref{lemma:tnorm_vanish} and \ref{prop:perturbed_oseledets_blocks_bound}, for every $n \in \Z^+$ there exists $\epsilon_{\gamma, n} > 0$ so that if $P \in \mathcal{LY}(C_1, C_2, r, R) \cap \mathcal{O}_{\epsilon_{\gamma, n}}(Q)$ then
    \begin{equation*}
        K C_0 e^{-n\lambda_{k_Q+1,Q}}\left(C_{\kappa, n} \tnorm{\Pi_{G_{i,P}(\omega)} - \Pi_{E_{i,Q}(\omega)}} + \tnorm{Q^n_\omega - P_\omega^n}\right) \le \frac{\gamma}{2},
    \end{equation*}
    and so
     \begin{equation*}
        \abs{\frac{\wnorm{Q^n_\omega \Pi_{E_{i,Q}(\omega)}v}}{\wnorm{P^n_\omega v}} - 1} \le \gamma.
    \end{equation*}
    Hence
    \begin{equation}\label{eq:stability_lyapunov_6}
        \abs{\log\left( \wnorm{P^n_\omega v} \right) - \log\left( \wnorm{Q^n_\omega \Pi_{E_{i,Q}(\omega)}v} \right)} \le \max\left\{ \log(1 + \gamma), -\log(1 - \gamma) \right\} := e(\gamma).
    \end{equation}
    We finish this part of the proof by deriving a lower bound for $\norm{\Pi_{E_{i,Q}(\omega)} v}$ when $v \in E_{j,P}(\omega) \cap B_{\norm{\cdot}}$. By Lemma \ref{lemma:oseledets_eccentricity} and as $E_{j,P}(\omega) \subseteq G_{i,P}(\omega)$ we have
    \begin{equation*}\begin{split}
        \norm{\Pi_{E_{i,Q}(\omega)} v} \ge \wnorm{\Pi_{E_{i,Q}(\omega)} v}
        &\ge \wnorm{\Pi_{G_{i,P}(\omega)}v} - \tnorm{\Pi_{G_{i,P}(\omega)} - \Pi_{E_{i,Q}(\omega)}} \\
        &\ge K^{-1} - \tnorm{\Pi_{G_{i,P}(\omega)} - \Pi_{E_{i,Q}(\omega)}}.
    \end{split}\end{equation*}
    Thus by Proposition \ref{prop:perturbed_oseledets_blocks_bound} there exists some $\epsilon'$ such that if $P \in \mathcal{LY}(C_1, C_2, r, R) \cap \mathcal{O}_{\epsilon'}(Q)$ then $\norm{\Pi_{E_{i,Q}(\omega)} v} \ge 1/(2K)$. We assume that $\epsilon_{\gamma,n} \le \epsilon'$ without loss of generality.
    \paragraph{An upper bound for the right hand side of \eqref{eq:stability_lyapunov_2}.}
    From \eqref{eq:stability_lyapunov_6}, Proposition \ref{prop:perturbed_oseledets_blocks_bound}, and as $ \norm{\Pi_{E_{i,Q}(\omega)} v} \ne 0$ for $v \in E_{j,P} \cap B_{\norm{\cdot}}$, we get
    \begin{equation*}\begin{split}
        \log\left(\sup_{v \in E_{j,P} \cap B_{\norm{\cdot}}} \wnorm{P^n_\omega v} \right) &\le \sup_{v \in E_{j,P} \cap B_{\norm{\cdot}}}\left(\log\left( \frac{\norm{Q^n_\omega \Pi_{E_{i,Q}(\omega)} v}}{\norm{\Pi_{E_{i,Q}(\omega)} v}} \right) + \log\left(\norm{\Pi_{E_{i,Q}(\omega)} v}\right)\right)+ e(\gamma) \\
        &\le \log\left( \norm{\restr{Q^n_\omega}{E_{i,Q}}} \right) + \log(\Theta_0^2)+ e(\gamma).
    \end{split}\end{equation*}
    From \eqref{eq:stability_lyapunov_2} we deduce that if $P \in \mathcal{LY}(C_1, C_2, r, R) \cap \mathcal{O}_{\epsilon_{\gamma, n}}(Q)$ then
    \begin{equation*}
        \frac{1}{n d_{j, P}} \intf \det( P^n_\omega | E_{j,P}(\omega) ) d\mathbb{P} \le \frac{1}{n} \intf \log\left( \norm{\restr{Q^n_\omega}{E_{i,Q}}} \right) d\mathbb{P} + \frac{1}{n}\left(\log(K\Theta_0^2) + e(\gamma)\right).
    \end{equation*}
    Applying Proposition \ref{prop:exponent_formula} we see that for every $\nu$ we may take $n$ to be very large and $\gamma$ sufficiently small to produce $\epsilon_{\nu_1}$ (depending on $\gamma$ and $n$) so that if $P \in \mathcal{LY}(C_1, C_2, r, R) \cap \mathcal{O}_{\nu_1}(Q)$ then
     \begin{equation*}
        \lambda_{j,P} = \frac{1}{n d_{j, P}} \intf \det( P^n_\omega | E_{j,P}(\omega) ) d\mathbb{P} \le  \lambda_{i,Q} + \nu.
    \end{equation*}
    Hence
    \begin{equation}\label{eq:stability_lyapunov_7}
        \gamma_{\ell,P } - \gamma_{\ell, Q} = \lambda_{j,P} - \lambda_{i,Q} \le \nu.
    \end{equation}

    \paragraph{A lower bound for the left hand side of \eqref{eq:stability_lyapunov_2}.}
    From \eqref{eq:stability_lyapunov_6}, Lemma \ref{lemma:oseledets_eccentricity}, and as $\norm{\Pi_{E_{i,Q}(\omega)} v} \ge 1/(2K)$ we get
    \begin{equation*}\begin{split}
        \log\left(\inf_{v \in E_{j,P}(\omega) \cap B_{\norm{\cdot}}} \wnorm{P^n_\omega v}\right) &\ge \inf_{v \in E_{j,P}(\omega) \cap B_{\norm{\cdot}}} \left( \log\left(\frac{\wnorm{Q^n_\omega \Pi_{E_{i,Q}(\omega)} v}}{\norm{\Pi_{E_{i,Q}(\omega)} v}}\right) + \log\left(\norm{\Pi_{E_{i,Q}(\omega)} v}\right)\right) - e(\gamma)\\
        &\ge \log\left(\norm{\left(\restr{Q^n_\omega}{E_{i,Q}(\omega)}\right)^{-1}}^{-1}\right) - \log(K) -\log(2K) - e(\gamma).
    \end{split}\end{equation*}
    Thus by \eqref{eq:stability_lyapunov_2} for $P \in \mathcal{LY}(C_1, C_2, r, R) \cap \mathcal{O}_{\epsilon_{\gamma, n}}(Q)$ we have
    \begin{equation*}
        \frac{1}{n d_{j, P}} \intf \det( P^n_\omega | E_{j,P}(\omega) ) d\mathbb{P} \ge \frac{1}{n} \intf \log\left(\norm{\left(\restr{Q^n_\omega}{E_{i,Q}(\omega)}\right)^{-1}}^{-1}\right) d\mathbb{P} - \frac{\log(2K^2) + e(\gamma)}{n}.
    \end{equation*}
     Applying Proposition \ref{prop:exponent_formula} as in the previous step, we see that for every $\nu > 0$ we may take $n$ to be very large and $\gamma$ sufficiently small to produce $\epsilon_{\nu_2}$ (depending on $\gamma$ and $n$) so that if $P \in \mathcal{LY}(C_1, C_2, r, R) \cap \mathcal{O}_{\epsilon_{\nu_2}}(Q)$ then $\lambda_{j,P} \ge \lambda_{i,P} - \nu$.
     Hence
     \begin{equation}\label{eq:stability_lyapunov_8}
        \gamma_{\ell,Q } - \gamma_{\ell, P} = \lambda_{i,Q} - \lambda_{\ell,P}  \le \nu.
    \end{equation}
    Setting $\nu = \min\{\nu_1, \nu_2\}$, and then combining \eqref{eq:stability_lyapunov_7} and \eqref{eq:stability_lyapunov_8}  yields $\abs{\gamma_{\ell,P } - \gamma_{\ell, Q}} \le \nu $ for $P \in \mathcal{LY}(C_1, C_2, r, R) \cap \mathcal{O}_{\epsilon_{\nu}}(Q)$.
    As discussed at the beginning of the proof, this suffices to prove \eqref{eq:stability_lyapunov_0}, which completes the proof of Theorem \ref{thm:stability_lyapunov}.
\end{proof}

\section{Application to random dynamical systems}
\label{sec:app_random_dynamics}

In this section we demonstrate the application of Theorems \ref{thm:stability_cocycle} and \ref{thm:stability_lyapunov} to cocycles of Perron-Frobenius operators associated to random dynamical systems consisting of $\mathcal{C}^k$ expanding maps on $S^1$, with $k \ge 2$.
We will consider two types of perturbations to such maps: fiber-wise `deterministic' perturbations to the random dynamics\footnote{In this rather unfortunate oxymoron, a fiber-wise `deterministic' perturbation simply means that the random maps are fiber-wise perturbed to nearby maps.}, and perturbations that arise via numerical approximations of the Perron-Frobenius cocycle.

Fix a Lebesgue probability space $(\Omega, \mathcal{F}, \mathbb{P})$, and an invertible, $\mathbb{P}$-ergodic map $\sigma : \Omega \to \Omega$. We will considering random dynamical systems taking values in the following sets.

\begin{definition}
  For $k \ge 2$, $\alpha \in (0,1)$ and $K > 0$ we set
  \begin{equation*}
    \LY_k(\alpha, K) = \left\{ T \in \mathcal{C}^k(S^1, S^1) : \inf \abs{T'} \ge \alpha^{-1} \text{ and } d_{\mathcal{C}^k}(T,0) \le K \right\}.
  \end{equation*}
  We say that $\mathcal{T} : \Omega \to \LY_k(\alpha, K)$ is measurable if it is measurable with respect to $\mathcal{F}$ and the Borel $\sigma$-algebra on $\mathcal{C}^{k}(S^1, S^1)$.
\end{definition}

Suppose $k \ge 2$, $\alpha \in (0,1)$, and $K > 0$. If $T \in \LY_k(\alpha,K)$ then we denote by $L_T : L^1(S^1) \to L^1(S^1)$ the associated Perron-Frobenius operator, which is defined by duality via
\begin{equation}\label{eq:def_pf}
  \intf L_Tf \cdot g d\Leb = \intf f \cdot g \circ T d\Leb \quad \forall f \in L^1(S^1), g \in L^\infty(S^1).
\end{equation}
In a slight abuse of notation, whenever $\mathcal{T} : \Omega \to \LY_k(\alpha, K)$ is measurable, we denote by $L_{\mathcal{T}} : \Omega \to \LL(L^1)$ the map defined by $L_{\mathcal{T}}(\omega) = L_{\mathcal{T}(\omega)}$.
The regularity of maps in $\LY_k(\alpha, K)$ suggests that we should consider how their Perron-Frobenius operators act on objects with some smoothness, rather than on $L^1(S^1)$.
For $k \in \N$ the Sobolev space $W^{k,1}(S^1)$ is defined by
\begin{equation*}
 W^{k,1}(S^1) = \{ f \in L^p(S^1) : f^{(\ell)} \text{ exists in the weak sense and } \norm{f^{(\ell)}}_{L^1} < \infty \text{ for each } 0 \le \ell \le k \}.
\end{equation*}
Each $W^{k,1}(S^1)$ becomes a Banach space when equipped with the norm
\begin{equation*}
 \norm{f}_{W^{k,1}} = \norm{f}_{L^1} + \norm{f^{(k)}}_{L^1}.
\end{equation*}
For each $k \ge 1$ the embedding of $W^{k,1}(S^1)$ into $W^{k-1,1}(S^1)$ is compact by the Rellich--Kondrachov Theorem.
Moreover, $\norm{f^{(k)}}_{L^1} = \Var(f^{(k-1)})$ and so by following the arguments in Examples \ref{example:bv} and \ref{example:sobolev} we conclude that $\norm{\cdot}_{W^{k,1}}$ is upper semicontinuous with respect to $\norm{\cdot}_{W^{k-1,1}}$.
Thus $(W^{k,1}(S^1), \norm{\cdot}_{W^{k,1}}, \norm{\cdot}_{W^{k-1,1}})$ is a pre-compact Saks space.
We remind the reader that each $W^{k,1}(S^1)$ is separable as a Banach space.

\begin{proposition}\label{prop:c2_cocycles_measurable}
  If $\mathcal{T} : \Omega \to \LY_k(\alpha, K)$ is measurable for  $k \ge 2$, $\alpha \in (0,1)$ and $K > 0$, then $(\Omega, \mathcal{F}, \mathbb{P}, \sigma, W^{k-1,1}(S^1), L_{\mathcal{T}})$ is a separable strongly measurable random linear system with ergodic invertible base.
  Moreover $(\Omega, \mathcal{F}, \mathbb{P}, \sigma, W^{k-1,1}(S^1),L_{\mathcal{T}})$ has an Oseledets splitting of dimension $d \ge 1$ with $\lambda_{1,\mathcal{T}} = 0$.
\end{proposition}

We will now make precise our first type of perturbation. For measurable maps $\mathcal{S}, \mathcal{T} : \Omega \to \LY_k(\alpha, K)$ we set
\begin{equation*}
  d_{k-1}(\mathcal{S}, \mathcal{T}) = \esssup_{\omega \in \Omega} d_{\mathcal{C}^{k-1}}(\mathcal{S}(\omega), \mathcal{T}(\omega)).
\end{equation*}
For $\epsilon > 0$ and measurable $\mathcal{T} : \Omega \to \LY_k(\alpha, K)$ we set
\begin{equation*}
  \mathcal{O}_{\epsilon, k, \alpha, K}(\mathcal{T}) = \left\{ \mathcal{S}: \Omega \to \LY_k(\alpha, K) \, \bigg\vert \, \mathcal{S} \text{ is measurable with } d_{k-1}(\mathcal{T}, \mathcal{S}) \le \epsilon \right\}.
\end{equation*}
The next result concerns the stability of Oseledets splitting and Lyapunov exponents of cocycles of Perron-Frobenius operator associated to maps in $\LY_k(\alpha, K)$ under perturbations which are small in the $d_{k-1}$ metric.
We adopt the notation of Section \ref{sec:stability_lyapunov}, aside from frequently replacing $L_\mathcal{T}$ (resp. $L_\mathcal{S}$) by $\mathcal{T}$ (resp. $\mathcal{S}$) in various subscripts.

\begin{theorem}\label{thm:random_deterministic_perturbation}
  Fix $k \ge 2$, $\alpha \in (0,1)$ and $K > 0$, and suppose that $\mathcal{T} : \Omega \to \LY_k(\alpha, K)$ is measurable, and that $(\Omega, \mathcal{F}, \mathbb{P}, \sigma, W^{k-1,1}(S^1), L_{\mathcal{T}})$ admits a hyperbolic Oseledets splitting of dimension $d$ with $(k-1)\log \alpha < \mu_{L_{\mathcal{T}}}$.
  There exists $\epsilon > 0$ such that if $\mathcal{S} \in \mathcal{O}_{\epsilon, k, \alpha, K}(\mathcal{T})$ then $(\Omega, \mathcal{F}, \mathbb{P}, \sigma, W^{k-1,1}(S^1), L_{\mathcal{S}})$ has an Oseledets splitting of dimension $d$.
  In addition, there exists $c_0,R_0 > 0$ such that each $I_{i} = (\lambda_{i,\mathcal{T}} - c_0, \max\{\lambda_{i,\mathcal{T}}, \log(\delta_{1i} R_0)\} + c_0)$, $i \in \{1, \dots, k_{\mathcal{T}}\}$, separates the Lyapunov spectrum of $(\Omega, \mathcal{F}, \mathbb{P}, \sigma, W^{k-1,1}(S^1), L_{\mathcal{S}})$, and the corresponding projections satisfy
  \begin{equation*}
    \forall i \in \{1, \dots, k_{\mathcal{T}}\}, \text{ a.e. } \omega \in \Omega \quad \rank(\Pi_{I_i, \mathcal{S}}(\omega)) = d_{i,\mathcal{T}},
  \end{equation*}
  and
  \begin{equation*}
    \sup \left\{ \esssup_{\omega \in \Omega} \norm{\Pi_{I_i, \mathcal{S}}(\omega) }_{\LL(W^{k-1,1})}\, \bigg\vert \, \mathcal{S} \in \mathcal{O}_{\epsilon, k, \alpha, K}(\mathcal{T}), 1\le i \le k_{\mathcal{T}} \right\} < \infty.
  \end{equation*}
  Moreover, for every $\beta > 0$ there exists $\epsilon_\beta > 0$ so that if $\mathcal{S} \in \mathcal{O}_{\epsilon_{\beta}, k, \alpha, K}(\mathcal{T})$ then
  \begin{equation*}
      \sup_{1 \le i \le d} \abs{\gamma_{i, \mathcal{T}} - \gamma_{i, \mathcal{S}}} \le \beta,
  \end{equation*}
  \begin{equation*}
    \sup_{1 \le i \le k_\mathcal{T}} \esssup_{\omega \in \Omega} \norm{\Pi_{I_i, \mathcal{T}}(\omega) - \Pi_{I_i, \mathcal{S}}(\omega)}_{\LL(W^{k-1,1},W^{k-2,1})} \le \beta,
  \end{equation*}
  and
  \begin{equation*}
    \esssup_{\omega \in \Omega} d_H(F_{\mathcal{T}}(\omega), F_{\mathcal{S}}(\omega)) \le \beta.
  \end{equation*}
\end{theorem}

Our second application concerns the numerical approximation of the Oseledets splitting and Lyapunov exponents associated to a Perron-Frobenius operator cocycle.
For each $n \in \Z^+$ the $n$th Fej{\'e}r kernel $J_{n} : S^1 \to \C$ is defined by
\begin{equation*}
  J_{n}(t) = \sum_{k=-n}^n \left(1 - \frac{\abs{k}}{n+1} \right)e^{2 \pi i k t}.
\end{equation*}
Convolution with the $n$th Fej{\'e}r kernel corresponds to taking the Ces{\`a}ro average of the first $n+1$ partial Fourier series, so that for each $f \in L^1(S^1)$ one has
\begin{equation}\label{eq:fejer_fourier_series}
  (J_n * f)(x) = \frac{1}{n+1}\sum_{\ell=0}^n \left(\sum_{j=-\ell}^\ell \hat{f}(j) e^{2\pi i j x} \right) = \sum_{\ell=-n}^n \left(1 - \frac{\abs{k}}{n+1}\right) \hat{f}(\ell) e^{2\pi i \ell x},
\end{equation}
where $\hat{f}(\ell) = \intf f(x) e^{-2\pi i \ell x} d \Leb$. The following proposition, which is well-known, summarises the relevant properties of the Fej{\'e}r kernel in our setting.

\begin{proposition}\label{prop:fejer_props}
  For $n \in \Z^+$ let $\mathcal{J}_n : L^1(S^1) \to L^1(S^1)$ denote the operator defined by
  \begin{equation*}
    \mathcal{J}_n(f) = J_n \ast f.
  \end{equation*}
  For every $n,k \in \Z^+$ the operator $\mathcal{J}_n$ is Markov\footnote{That is, the positive cone in $W^{k,1}(S^1)$ is invariant under $\mathcal{J}_n$, and $\mathcal{J}_n$ preserves integrals.} and restricts to a contraction in $\LL(W^{k,1}(S^1))$.
  In addition, if $k \ge 1$ then
  \begin{equation}\label{eq:fejer_props_1}
    \lim_{n \to \infty} \norm{\mathcal{J}_n - \Id}_{\LL(W^{k,1}, W^{k-1,1})} = 0.
  \end{equation}
\end{proposition}

When $\mathcal{T} : \Omega \to \LY_k(\alpha, K)$ is measurable and  $n \in \Z^+$ we define $L_{\mathcal{T},n} : \Omega \to \LL(W^{k-1,1}(S^1))$ by $L_{\mathcal{T},n}(\omega) = \mathcal{J}_n  L_{\mathcal{T}}(\omega)$.
Note that each $L_{\mathcal{T},n}(\omega)$ has finite rank and preserves $\vspan\{ e^{2\pi i \ell x} : -n \le \ell \le n\}$. Hence, by a constant change of basis we may view $(\Omega, \mathcal{F}, \mathbb{P}, \sigma, W^{k-1,1}(S^1), L_{\mathcal{T},n})$ as a matrix cocycle on $\C^{2n+1}$.
One could then use this matrix representation to approximate the Oseledets splitting and Lyapunov exponents of the original cocycle by computing the singular value decomposition of very large iterates of the matrix cocycle, as in \cite{froyland2010coherent}.
While a completely rigorous proof of convergence for such an algorithm is outside of the scope of this paper, we believe that the following theorem is a substantial step in the direction of such a result.

\begin{theorem}\label{thm:random_fejer_approx}
  Fix $k \ge 2$, $\alpha \in (0,1)$ and $K > 0$.
  If $\mathcal{T} : \Omega \to \LY(\alpha, K)$ is measurable and $(\Omega, \mathcal{F}, \mathbb{P}, \sigma, W^{k-1,1}(S^1), L_{\mathcal{T}})$ admits a hyperbolic Oseledets splitting of dimension $d$ with $(k-1)\log \alpha < \mu_{L_{\mathcal{T}}}$, then there exists $N$ such that if $n > N$ then $(\Omega, \mathcal{F}, \mathbb{P}, \sigma, W^{k-1,1}(S^1), L_{\mathcal{T},n})$ admits an Oseledets splitting of dimension $d$.
  In addition, there exists $c_0, R_0 > 0$ such that each $I_{i} = (\lambda_{i,\mathcal{T}} - c_0, \max\{\lambda_{i,\mathcal{T}}, \log(\delta_{1i} R_0)\} + c_0)$, $i \in \{1, \dots, k_{\mathcal{T}}\}$, separates the Lyapunov spectrum of $(\Omega, \mathcal{F}, \mathbb{P}, \sigma, W^{1,1}(S^1), L_{\mathcal{T},n})$, and the corresponding projections satisfy
  \begin{equation*}
    \forall i \in \{1, \dots, k_{\mathcal{T}}\}, \text{ a.e. } \omega \in \Omega \quad \rank(\Pi_{I_i, L_{\mathcal{T},n}}(\omega)) = d_{i,\mathcal{T}},
  \end{equation*}
  and
  \begin{equation*}
    \sup \left\{ \esssup_{\omega \in \Omega} \norm{\Pi_{I_i, L_{\mathcal{T},n}}(\omega) }_{\LL(W^{k-1,1})}\, \bigg\vert \, n > N, 1\le i \le k_{\mathcal{T}} \right\} < \infty.
  \end{equation*}
  In addition, for every $\beta > 0$, there exists $N_\beta > N$ such that if $n > N_\beta$ then
  \begin{equation*}
      \sup_{1 \le i \le d} \abs{\gamma_{i, L_{\mathcal{T}}} - \gamma_{i, L_{\mathcal{T},n}}} \le \beta,
  \end{equation*}
  \begin{equation*}
    \sup_{1 \le i \le k_\mathcal{T}} \esssup_{\omega \in \Omega} \norm{\Pi_{I_i, L_{\mathcal{T}}}(\omega) - \Pi_{I_i, L_{\mathcal{T},n}}(\omega)}_{\LL(W^{k-1,1},W^{k-2,1})} \le \beta,
  \end{equation*}
  and
  \begin{equation*}
    \esssup_{\omega \in \Omega} d_H(F_{L_{\mathcal{T}}}(\omega), F_{L_{\mathcal{T},n}}(\omega)) \le \beta.
  \end{equation*}
\end{theorem}

Before proving the results described thus far we will describe some concrete setting in which they may be applied. We note that the chief difficulty in applying Theorems \ref{thm:random_deterministic_perturbation} \ref{thm:random_fejer_approx} is not proving the existence of an Oseledets splitting (recall Proposition \ref{prop:c2_cocycles_measurable}), but verifying that the splitting is hyperbolic.

\begin{example}\label{example:fixed_map}
  Fix $k \ge 2$, $\alpha \in (0,1) $ and $K > 0$.
  For each $T \in \LY_{k}(\alpha,K)$ we may consider the constant random dynamical system given by $\omega \mapsto T$. In this case the associated Perron-Frobenius operator $L_T$ is quasi-compact\footnote{We remind the reader of the proof. One bounds the essential spectral radius by using Theorem \ref{thm:ITMSS} and Proposition \ref{prop:ly_ss}.
  Since $L_T$ preserves integrals we have $\rho(L_T) \ge 1$.
  If $\rho(L_T) > 1$ then $L_T$ has an eigenvalue of modulus greater than 1 on $W^{k-1,1}(S^1)$, which must also be an eigenvalue for $L_T$ on $L^1(S^1)$; but $L_T$ is a contraction on $L^1(S^1)$ and so no such eigenvalue can exist.} on $W^{k-1,1}(S^1)$ with $\rho_{\ess}(L_T) \le \alpha^{k-1} < \rho(L_T) = 1$.
  It follows that $(\Omega, \mathcal{F}, \mathbb{P}, \sigma, W^{k-1,1}(S^1), \omega \mapsto L_{T})$ admits a hyperbolic Oseledets splitting with $\mu > \log \alpha^{k-1}$: for any such $\mu \in (\log \alpha^{k-1}, 0)$ the fast Oseledets spaces are just direct sums of the eigenspaces of $L_T$ associated to eigenvalues of modulus greater than $e^{\mu}$ (of which there are finitely many), and the Lyapunov exponents are $\{ \log \abs{\lambda} : \lambda \in \sigma(L_T), \abs{\lambda} > \log \mu \}$.
  We refer the reader to \cite{keller2004eigenfunctions} and \cite{slipantschuk2013analytic} for examples of expanding maps on $S^1$ with non-trivial eigenvalues with modulus in $(\alpha^{k-1},1)$, and note that the different choice of Banach space in either paper is inconsequential due to \cite[Section A.2]{baladi2018dynamical}.
  We may therefore apply Theorem \ref{thm:random_deterministic_perturbation} to $(\Omega, \mathcal{F}, \mathbb{P}, \sigma, W^{k-1,1}(S^1), \omega \mapsto L_{T})$ with perturbation $(\Omega, \mathcal{F}, \mathbb{P}, \sigma, W^{k-1,1}(S^1), L_{\mathcal{S}})$ whenever $\mathcal{S} : \Omega \mapsto \LY_k(\alpha,K)$ is measurable and such that $\esssup_{\omega \in \Omega} d_{\mathcal{C}^{k-1}}(\mathcal{S}(\omega), T)$ is sufficiently small.
\end{example}

\begin{example}
  If, in the setting of Example \ref{example:fixed_map}, there exists $\mu \in (\log \alpha^{k-1}, 0)$ such that for every $r > \log \mu$ the set $\{\lambda \in \sigma(L_T) : \abs{\lambda} = r\}$ contains at most a single element, then every Lyapunov exponent of $(\Omega, \mathcal{F}, \mathbb{P}, \sigma, W^{k-1,1}(S^1), \omega \mapsto L_{T})$ has multiplicity one.
  Thus, by Remark \ref{remark:not_hyperbolic}, it follows that if
  $\mathcal{S} : \Omega \mapsto \LY_k(\alpha,K)$ is measurable
  and $\esssup_{\omega \in \Omega} d_{\mathcal{C}^{k-1}}(\mathcal{S}(\omega), T)$ is sufficiently small then the Oseledets splitting for $(\Omega, \mathcal{F}, \mathbb{P}, \sigma, W^{k-1,1}(S^1), L_{\mathcal{S}})$ that is produced by Theorem \ref{thm:random_deterministic_perturbation} is hyperbolic.
  Thus both Theorem \ref{thm:random_deterministic_perturbation} and Theorem \ref{thm:random_fejer_approx} may be applied to $(\Omega, \mathcal{F}, \mathbb{P}, \sigma, W^{k-1,1}(S^1), L_{\mathcal{S}})$.
\end{example}

\subsection{Proofs for Section \ref{sec:app_random_dynamics}}

The next proposition, which is well-known, summarises the basic properties of Perron-Frobenius operators associated to maps in $\LY_k(\alpha, K)$.
\begin{proposition}\label{prop:pf_summary}
  There exists $C_{k-1,\alpha,K} > 0$ such that for every $T \in \LY_k(\alpha, K)$ and $f \in W^{k-1,1}(S^1)$ we have
  \begin{equation}\label{eq:pf_summary_0}
    \norm{L_T f }_{W^{k-1,1}} \le \alpha^{k-1} \norm{ f }_{W^{k-1,1}} + C_{k-1,\alpha,K} \norm{f}_{W^{k-2,1}}.
  \end{equation}
  Hence $\{ L_T : T \in \LY_{k}(\alpha, K) \}$ is an equicontinuous subset of $\LL_S(W^{k-1,1})$, where $W^{k-1,1}(S^1)$ has the Saks space structure $(W^{k-1,1}(S^1), \norm{\cdot}_{W^{k-1,1}}, \norm{\cdot}_{W^{k-2,1}})$.
\end{proposition}

Before proving Proposition \ref{prop:pf_summary} we need the following lemma.
\begin{lemma}\label{lemma:higher_deriv}
  For every $k \in \N$ there exists multinomials $G_{k,\ell} : \R^{k} \to \R$, $\ell \in \{0,\dots,k\}$, such that for every $T \in \LY_k(\alpha, K)$ and $f \in W^{k,1}$ we have
  \begin{equation}\label{eq:higher_deriv_0}
    (L_Tf)^{(k)} = L_T\left( (T')^{-2k}\sum_{\ell=0}^k G_{k,\ell}(T', \dots, T^{(k+1)}) \cdot f^{(\ell)} \right).
  \end{equation}
  Moreover $G_{k,k}(x_1, \dots, x_{k+1}) = x_1^{k}$.
  \begin{proof}
    By differentiating the identity
    \begin{equation*}
      (L_Tf)(x) = \sum_{T(y) = x} \frac{f(y)}{\abs{T'(y)}}
    \end{equation*}
    one finds that $(L_Tf)' = L_T( (1/T') f' + (1/T')' f)$ whenever $f \in W^{1,1}(S^1)$, from which \eqref{eq:higher_deriv_0} follows by a straightforward induction on $k$.
    The claim that $G_{k,k}(x_1, \dots, x_k) = x_1^{k}$ for all $k$ follows upon noting that the coefficient of $f^{(k)}$ in \begin{equation*}
       (T')^{-2k}\sum_{\ell=0}^k G_{k,\ell}(T', \dots, T^{(k+1)}) \cdot f^{(\ell)}
    \end{equation*}
    is always $(T')^{-k}$.
  \end{proof}
\end{lemma}

\begin{proof}[{The proof of Proposition \ref{prop:pf_summary}}]
  For brevity we write $G_{k-1,\ell,T}$ in place of $G_{k-1,\ell}(T', \dots, T^{(k)})$.
  By Lemma \ref{lemma:higher_deriv} and as $L_T$ is Markov we have
  \begin{equation*}\begin{split}
    \norm{(L_Tf)^{(k-1)}}_{L^1} &= \intf \abs{ L_T\left((T')^{-2(k-1)}\sum_{\ell=0}^{k-1} G_{k-1,\ell,T} \cdot f^{(\ell)} \right) } d\Leb \\
    &\le \intf \frac{\abs{f^{(k-1)}}}{\abs{T'}^{k-1}} d\Leb + \sum_{\ell=0}^{k-2} \intf \abs{T'}^{-2(k-1)} \abs{G_{k-1,\ell,T}} \abs{f^{(\ell)}} d\Leb \\
    &\le \alpha^{k-1} \norm{f}_{W^{k-1,1}} + \left(\sum_{0= \ell}^{k-1} \alpha^{2(k-1)} D_{k-2,\ell} \norm{G_{k-1,\ell,T}}_{L^\infty} \right) \norm{f}_{W^{k-2,1}},
  \end{split}\end{equation*}
  where $D_{k-2,\ell}$ denotes the norm of the embedding of $W^{k-2,1}(S^1)$ into $W^{\ell,1}(S^1)$.
  Let
  \begin{equation*}
    C_{k-1,\alpha,K} = 1 + \sup_{T \in \LY_k(\alpha, K)} \left(\sum_{0= \ell}^{k-1} \alpha^{2(k-1)} D_{k-2,\ell} \norm{G_{k-1,\ell,T}}_{L^\infty} \right),
  \end{equation*}
  and note that $C_{k-1,\alpha,K} < \infty$.
  Since $L_T$ is Markov on $L^1(S^1)$ we therefore have
  \begin{equation*}
    \norm{L_T f}_{W^{k-1,1}} = \norm{L_T f}_{L^1} + \norm{(L_Tf)^{(k-1)}}_{L^1} \le \alpha^{k-1} \norm{f}_{W^{k-1,1}} + C_{k-1,\alpha,K}\norm{f}_{W^{k-2,1}},
  \end{equation*}
  which yields \eqref{eq:pf_summary_0}.
  Since the embedding of $W^{k-1,1}(S^1)$ into $W^{k-2,1}(S^1)$ is bounded, \eqref{eq:pf_summary_0} also implies that $\{ L_T : T \in \LY_k(\alpha, K) \}$ is a bounded subset of $\LL(W^{k-1,1}(S^1))$.
  For $k > 2$ the same argument shows that $\{ L_T : T \in \LY_k(\alpha, K) \}$ is a bounded subset of $\LL(W^{k-2,1}(S^1))$, while if $k = 2$ then $\{ L_T : T \in \LY_k(\alpha, K) \}$ is bounded in $\LL(W^{k-2,1}(S^1))$ since each $L_T$ is Markov on $W^{k-2,1}(S^1) = L^1(S^1)$.
  That $\{ L_T : T \in \LY_k(\alpha, K) \}$ is an equicontinuous subset of $\LL_S(W^{k-1,1}(S^1))$ then follows from Proposition \ref{prop:saks_equicont}.
\end{proof}

\begin{proposition}\label{prop:ly_ss}
  Let $\mathbb{W}^{k-1} = \bigsqcup_{\omega \in \Omega} \{\omega\} \times W^{k-1,1}$.
  There exists $R_{k-1,\alpha, K} \ge 1$ and $A_{k-1,\alpha, K} >0$ such that if $\mathcal{T} : \Omega \to \LY_k(\alpha, K)$ is measurable then $L_{\mathcal{T}} \in \mathcal{LY}(1,A_{k-1,\alpha, K},\alpha^{k-1},R_{k-1,\alpha, K}) \cap \End_{S}(\mathbb{W}^{k-1}, \sigma)$.
  \begin{proof}
    That $L_{\mathcal{T}} \in \End_{S}(\mathbb{W}^{k-1}, \sigma)$ follows trivially from Proposition \ref{prop:pf_summary}.
    Let
    \begin{equation*}
      R_{k-1,\alpha,K} = \max\{1, \sup\{ \norm{L_T}_{\LL(W^{k-2,1})} : T \in \LY_k(\alpha, K) \}\},
    \end{equation*}
    and note that $R_{k-1,\alpha,K}$ is finite by Proposition \ref{prop:pf_summary}.
    By Proposition \ref{prop:pf_summary}, for every $f \in W^{k-1,1}(S^1)$, $\omega \in \Omega$ and $n \in \Z^+$ we have
    \begin{equation*}\begin{split}
      \norm{(L_{\mathcal{T}(\sigma^{n}(\omega))} \circ \dots \circ L_{\mathcal{T}(\omega)})f}_{W^{k-1,1}} \le &\alpha^{k-1} \norm{(L_{\mathcal{T}(\sigma^{n-1}(\omega))} \circ \dots \circ L_{\mathcal{T}(\omega)})f}_{W^{k-1,1}} \\
      &+ C_{k-1,\alpha,K} R_{k-1,\alpha,K}^{n-1} \norm{f}_{W^{k-2,1}}.
    \end{split}\end{equation*}
    Iterating the above inequality yields
    \begin{equation*}\begin{split}
      \norm{(L_{\mathcal{T}(\sigma^{n}(\omega))} \circ \dots \circ L_{\mathcal{T}(\omega)})f}_{W^{k-1,1}} &\le \alpha^{n(k-1)} \norm{f}_{W^{k-1,1}} + \frac{C_{k-1,\alpha,K}}{R_{k-1,\alpha, K}} R_{k-1,\alpha, K}^{n} \left( \sum_{j=0}^{n-1} \left( \frac{\alpha^{k-1}}{R_{k-1,\alpha, K}}\right)^{j} \right)\norm{f}_{W^{k-2,1}}\\
      &\le \alpha^{n(k-1)} \norm{f}_{W^{k-1,1}} + \frac{C_{k-1,\alpha,K}}{R_{k-1,\alpha, K} - \alpha^{k-1}} R_{k-1,\alpha, K}^{n}\norm{f}_{W^{k-2,1}}.
    \end{split}\end{equation*}
    We obtain the claim upon setting $A_{k-1,\alpha, K} =C_{k-1,\alpha,K}(R_{k-1,\alpha, K} - \alpha^{k-1})^{-1}$.
  \end{proof}
\end{proposition}

\begin{proof}[{The proof of Proposition \ref{prop:c2_cocycles_measurable}}]
    To show that $(\Omega, \mathcal{F}, \mathbb{P}, \sigma, W^{k-1,1}(S^1), L_{\mathcal{T}})$ is a separable strongly measurable random linear system with ergodic invertible base it suffices to the map $\Delta : \LY_k(\alpha, K) \to \LL(W^{k-1,1})$ defined by $\Delta(T) = L_T$ is measurable with respect to the Borel $\sigma$-algebras on $\mathcal{C}^2(S^1,S^1)$ and $\LL(W^{k-1,1})$, where the later space is equipped with the strong operator topology.
    We will do this by showing that $\Delta$ is continuous: for every $f \in W^{k-1,1}$ we will show that $\norm{(L_T -L_S)f}_{W^{k-1,1}} \to 0$ as $d_{\mathcal{C}^k}(T,S) \to 0$.
    By \cite[(C1) of Lemma 2.4]{viviane2000positive}, for every $g \in \mathcal{C}^{k-1}(S^1)$ we have $\norm{(L_T -L_S)f}_{\mathcal{C}^{k-1}} \to 0$ as $d_{\mathcal{C}^k}(T,S) \to 0$.
    Fix $f \in W^{k-1,1}$ and for each $\epsilon > 0$ let $f_\epsilon \in \mathcal{C}^{k-1}(S^1)$ satisfy $\norm{f - f_\epsilon}_{W^{k-1,1}} \le \epsilon$.
    Then
    \begin{equation*}\begin{split}
      \norm{(L_T -L_S)f}_{W^{k-1,1}} &\le \norm{(L_T -L_S)f_\epsilon}_{W^{k-1,1}} + \norm{(L_T -L_S)(f - f_\epsilon)}_{W^{k-1,1}} \\
      &\le \norm{(L_T -L_S)f_\epsilon}_{C^{k-1}} + 2 \epsilon\sup_{T \in \LY_k(\alpha,K)} \norm{L_T}_{\LL(W^{k-1,1})} \\
      &\to 2 \epsilon\sup_{T \in \LY_k(\alpha,K)} \norm{L_T}_{\LL(W^{k-1,1})},
    \end{split}\end{equation*}
    as $d_{\mathcal{C}^2}(T,S) \to 0$.
    The set $\{ L_T : T \in \LY_k(\alpha,K)\}$ is bounded in $\LL(W^{k-1,1})$ by Proposition \ref{prop:pf_summary}, and so we obtain the required claim by sending $\epsilon \to 0$.

    We will sketch the proof that $\mathcal{P} = (\Omega, \mathcal{F}, \mathbb{P}, \sigma, W^{k-1,1}(S^1),L_{\mathcal{T}})$ has an Oseledets splitting.
    We aim to verify the hypotheses of \cite[Theorem 2.10]{GTQuas1} i.e. that the index of compactness $\kappa_{\mathcal{P}}^*$ of $\mathcal{P}$ is less than the maximal Lyapunov exponent $\lambda_\mathcal{P}^*$ of $\mathcal{P}$ (see \cite[Definition 2.3]{GTQuas1}).
    Our proof of this roughly follows the argument laid out in \cite[Lemma 3.16]{GTQuas1}.
    Recall from Proposition \ref{prop:ly_ss} that $L_{\mathcal{T}} \in \mathcal{LY}(1,A_{k-1,\alpha, K},\alpha^{k-1},R_{k-1,\alpha, K})$.
    By \cite[Lemma C.5]{GTQuas1}, it follows that $\kappa_{\mathcal{P}}^* \le \log \alpha^{k-1}<0$.
    On the other hand, since each $L_\omega$ is Markov, we have
    \begin{equation*}
      \norm{(L_{\mathcal{T}(\sigma^{n}(\omega))} \circ \dots \circ L_{\mathcal{T}(\omega)})1}_{W^{k-1,1}} \ge \norm{(L_{\mathcal{T}(\sigma^{n}(\omega))} \circ \dots \circ L_{\mathcal{T}(\omega)})1}_{L^1} = 1,
    \end{equation*}
    and so $\lambda_\mathcal{P} \ge 0$.
    Thus $\mathcal{P}$ has an Oseledets splitting with $\lambda_{1,\mathcal{T}} = \lambda_{\mathcal{P}}^* \ge 0$ by \cite[Theorem 2.10]{GTQuas1}.
    We will now show that $\lambda_{1,\mathcal{T}} = 0$. Let $\mathcal{P}' = (\Omega, \mathcal{F}, \mathbb{P}, \sigma, W^{1,1}(S^1),L_{\mathcal{T}})$, and note that the arguments of the previous paragraph imply that $\mathcal{P}'$ has an Oseledets splitting with $\kappa_{\mathcal{P}'}^* \le \log \alpha < \lambda_{\mathcal{P}'}^*$ and that $\lambda_{\mathcal{P}'}^* \ge 0$.
    However, the Lasota-Yorke inequality obtained for $\mathcal{P}'$ from Proposition \ref{prop:ly_ss} has $R_{k-1,\alpha,K} = 1$, and so
    \begin{equation*}
      \sup_{n \in \Z^+} \sup_{\omega \in \Omega} \norm{(L_{\mathcal{T}(\sigma^{n}(\omega))} \circ \dots \circ L_{\mathcal{T}(\omega)})}_{\LL(W^{1,1})} < \infty,
    \end{equation*}
    which implies that $\lambda_{\mathcal{P}'}^* \le 0$. Thus $\lambda_{\mathcal{P}'}^* = 0$.
    In the language of \cite[Appendix A]{gonzalez2018stability}, $\mathcal{P}$ is a dense restriction of $\mathcal{P}'$. As $\lambda_{\mathcal{P}'}^* \ge \max\{\kappa_{\mathcal{P}'}^*, \kappa_{\mathcal{P}}^*\}$ by \cite[Theorem 37]{gonzalez2018stability} we have $\lambda_{\mathcal{P}'}^* = \lambda_{\mathcal{P}}^* = 0$.
\end{proof}

\begin{proposition}\label{prop:tnorm_close_deterministic}
  There exists $Q_{k,\alpha,K} > 0$ such that for every $S, T \in \LY_k(\alpha,K)$ we have
  \begin{equation*}
    \norm{L_T - L_S}_{\LL(W^{k-1,1},W^{k-2,1})} \le Q_{k,\alpha,K} d_{\mathcal{C}^{k-1}}(S,T).
  \end{equation*}
  \begin{proof}
    The case where $k = 2$ is known: upon recalling from Example \ref{example:sobolev} that $\norm{f}_{\BV} = \norm{f}_{W^{1,1}}$ for $f \in W^{1,1}$, the result is given by \cite[Example 3.1]{liverani2004invariant}.
    We therefore focus on the case where $k > 2$, although we use the $k = 2$ case during our argument.
    Let $f \in W^{k-2,1}(S^1)$ and $g \in L^\infty(S^1)$. For brevity, if $R \in \LY_k(\alpha,K)$ then we will write $G_{k-2,\ell,R}$ in place of $G_{k-2,\ell}(R', \dots, R^{(k-1)})$.
    By Lemma \ref{lemma:higher_deriv} we have
    \begin{equation}\begin{split}\label{eq:tnorm_close_deterministic_1}
       \intf (L_Tf - L_Sf)^{(k-2)} g d \Leb &= \intf \left(L_T\left( \sum_{\ell=0}^{k-2} \frac{G_{k-2,\ell,T} \cdot f^{(\ell)}}{(T')^{2(k-2)}} \right) - L_S\left(\sum_{\ell=0}^{k-2} \frac{G_{k-2,\ell,S} \cdot f^{(\ell)}}{(S')^{2(k-2)}} \right) \right) \cdot g d\Leb\\
       &= \intf (L_T - L_S)\left( \sum_{\ell=0}^{k-2} \frac{G_{k-2,\ell,T} \cdot f^{(\ell)}}{(T')^{2(k-2)}} \right) \cdot g d\Leb \\
       &+ \sum_{\ell=0}^{k-2}\intf \left(\frac{G_{k-2,\ell,T} }{(T')^{2(k-2)}} - \frac{G_{k-2,\ell,S}}{(S')^{2(k-2)}}\right) \cdot f^{(\ell)} \cdot g \circ S d\Leb.
    \end{split}\end{equation}
    To bound the first term we apply the inequality for $k = 2$, which is valid since $d_{\mathcal{C}^1}(S,T) \le d_{\mathcal{C}^{k-1}}(S,T)$, yielding
    \begin{equation*}
      \intf (L_T - L_S)\left( \sum_{\ell=0}^{k-2}
      \frac{G_{k-2,\ell,T} \cdot f^{(\ell)} }{(T')^{2(k-2)}}
      \right) g d\Leb \le Q_{2,\alpha,K} \norm{g}_{L^\infty} \left(\sum_{\ell=0}^{k-2} \norm{\frac{G_{k-2,\ell,T} \cdot f^{(\ell)}}{(T')^{2(k-2)}}}_{W^{1,1}}\right) d_{\mathcal{C}^{k-1}}(S,T).
    \end{equation*}
    If we let $D_{k-1, \ell}$ denote the norm of the embedding of $W^{k-1,1}(S^1)$ into $W^{\ell,1}(S^1)$ and set
    \begin{equation*}
       Z_{k, \ell, \alpha,K} = \sup_{T \in \LY_k(\alpha,K)} \left(\left(D_{k-1,\ell}  + \alpha 2(k-2) K + D_{k-1,\ell+1}\right) \norm{G_{k-2,\ell,T}}_{L^\infty} + D_{k-1,\ell}\norm{G_{k-2,\ell,T}'}_{L^\infty}\right),
    \end{equation*}
    then by the product rule, the definition of $\norm{\cdot}_{W^{1,1}}$ and as $\alpha < 1$ we have
    \begin{equation*}
      \sum_{\ell=0}^{k-2} \norm{\frac{G_{k-2,\ell,T} \cdot f^{(\ell)}}{(T')^{2(k-2)}}}_{W^{1,1}}\le \sum_{\ell=0}^{k-2} Z_{k, \ell, \alpha,K} \norm{f}_{W^{k-1,1}}.
    \end{equation*}
    Thus
    \begin{equation}\label{eq:tnorm_close_deterministic_2}
      \intf (L_T - L_S)\left( \sum_{\ell=0}^{k-2}
      \frac{G_{k-2,\ell,T} \cdot f^{(\ell)} }{(T')^{2(k-2)}}
      \right) g d\Leb \le Q_{2,\alpha,K} \left(\sum_{\ell=0}^{k-2} Z_{k, \ell, \alpha,K}\right)\norm{g}_{L^\infty}  \norm{f}_{W^{k-1,1}} d_{\mathcal{C}^{k-1}}(S,T).
    \end{equation}
    On the other hand,
    \begin{equation*}
      \intf \left(\frac{G_{k-2,\ell,T} }{(T')^{2(k-2)}} - \frac{G_{k-2,\ell,S}}{(S')^{2(k-2)}}\right) \cdot f^{(\ell)} \cdot g \circ S d\Leb \le D_{k-1,\ell}\norm{g}_{L^\infty} \norm{f}_{W^{k-1,1}}  \norm{\left(\frac{G_{k-2,\ell,T} }{(T')^{2(k-2)}} - \frac{G_{k-2,\ell,S}}{(S')^{2(k-2)}}\right)}_{L^\infty}.
    \end{equation*}
    Since each of the multinomials $G_{k-2,\ell}$ is Lipschitz on $[-K,K]^{k-1}$ and $G_{k-2,\ell,T}$ (resp. $G_{k-2,\ell,S}$) only contains derivatives of $T$ (resp. $S$) of order less than $k-1$, for each $\ell \in \{0, \dots, k-2\}$ there exists $V_{k,\ell}$ such that for every $S,T \in \LY_k(\alpha, K)$ we have
    \begin{equation*}
      \norm{\left(\frac{G_{k-2,\ell,T} }{(T')^{2(k-2)}} - \frac{G_{k-2,\ell,S}}{(S')^{2(k-2)}}\right)}_{L^\infty} \le V_{k,\ell}d_{\mathcal{C}^{k-1}}(S,T).
    \end{equation*}
    It follows that
    \begin{equation}\label{eq:tnorm_close_deterministic_3}
      \sum_{\ell=0}^{k-2}\intf \left(\frac{G_{k-2,\ell,T} }{(T')^{2(k-2)}} - \frac{G_{k-2,\ell,S}}{(S')^{2(k-2)}}\right) \cdot f^{(\ell)} \cdot g \circ S d\Leb \le D_{k-1,\ell}\left(\sum_{\ell=0}^{k-2} V_{k,\ell}\right) \norm{g}_{L^\infty} \norm{f}_{W^{k-1,1}} d_{\mathcal{C}^{k-1}}(S,T).
    \end{equation}
    Applying \eqref{eq:tnorm_close_deterministic_2} and \eqref{eq:tnorm_close_deterministic_3} to \eqref{eq:tnorm_close_deterministic_1}, and then taking the supremum over $g \in L^\infty(S^1)$ with $\norm{g}_{L^\infty} = 1$ yields
    \begin{equation*}
      \norm{(L_Tf - L_Sf)^{(k-2)}}_{L^1} \le \left(\sum_{\ell=0}^{k-2} D_{k-1,\ell} V_{k,\ell} + Q_{2,\alpha,K} Z_{k, \ell, \alpha,K}\right)\norm{f}_{W^{k-1,1}}d_{\mathcal{C}^{k-1}}(S,T).
    \end{equation*}
    Thus, by using the case where $k = 2$ again, we obtain
    \begin{equation*}
      \norm{L_T - L_S}_{\LL(W^{k-1,1},W^{k-2,1})} \le Q_{2,\alpha,K}D_{k-1,1}d_{\mathcal{C}^{k-1}}(S,T) + \left(\sum_{\ell=0}^{k-2} D_{k-1,\ell} V_{k,\ell} + Q_{2,\alpha,K} Z_{k, \ell, \alpha,K}\right)d_{\mathcal{C}^{k-1}}(S,T),
    \end{equation*}
    as required.
  \end{proof}
\end{proposition}

\begin{proof}[{The proof of Theorem \ref{thm:random_deterministic_perturbation}}]
  By assumption $(\Omega, \mathcal{F}, \mathbb{P}, \sigma, W^{k-1,1}(S^1), L_{\mathcal{T}})$ is a separable strongly measurable random linear system with ergodic invertible base and a hyperbolic Oseledets splitting of dimension $d$.
  We have $L_{\mathcal{T}} \in \End_{S}(\mathbb{W}^{k-1}, \sigma) \cap \mathcal{LY}(1,A_{k-1,\alpha, K},\alpha^{k-1},R_{k-1,\alpha, K})$ by Proposition \ref{prop:ly_ss}, and so all the requirements of Theorem \ref{thm:stability_lyapunov} are verified for $(\Omega, \mathcal{F}, \mathbb{P}, \sigma, W^{k-1,1}(S^1), L_{\mathcal{T}})$.
  For measurable $\mathcal{S} : \Omega \to \LY_k(\alpha, K)$ we get that $(\Omega, \mathcal{F}, \mathbb{P}, \sigma, W^{k-1,1}(S^1), L_{\mathcal{S}})$ is also a separable strongly measurable random linear system by Proposition \ref{prop:c2_cocycles_measurable} and by Proposition \ref{prop:ly_ss} we have $L_{\mathcal{S}} \in \mathcal{LY}(1,A_{k-1,\alpha, K},\alpha^{k-1},R_{k-1,\alpha, K})$.
  For any $\epsilon > 0$ we may ensure that
  \begin{equation*}
    \esssup_{\omega \in \Omega} \norm{L_{\mathcal{T}(\omega)} - L_{\mathcal{S}(\omega)}}_{\LL(W^{k-1,1},W^{k-2,1})} \le \epsilon,
  \end{equation*}
  by making $d_{k-1}(\mathcal{T}, \mathcal{S})$ small and then using Proposition \ref{prop:tnorm_close_deterministic}.
  Thus, we obtain the conclusion of Theorem \ref{thm:stability_lyapunov} for the perturbation $(\Omega, \mathcal{F}, \mathbb{P}, \sigma, W^{k-1,1}(S^1), L_{\mathcal{S}})$, as required.
\end{proof}

\begin{proof}[{Proof of Proposition \ref{prop:fejer_props}}]
  It is well known that the Fej{\'e}r kernels approximate the identity \cite[Section 2.2 and 2.5]{katznelson2002introduction}.
  Thus $\mathcal{J}_n$ is Markov, as claimed. Since convolution and differentiation commute, for each $f \in W^{k,1}(S^1)$ we have
  \begin{equation*}
    \norm{\mathcal{J}_n f}_{W^{k,1}} = \norm{ \mathcal{J}_n f}_{L^1} + \norm{ \mathcal{J}_n (f^{(k)})}_{L^1} \le \norm{f}_{L^1} + \norm{f^{(k)}}_{L^1} = \norm{f}_{W^{k,1}},
  \end{equation*}
  and so $\mathcal{J}_n$ restricts to a contraction in $\LL(W^{k,1}(S^1))$.
  For $x,y \in S^1$ define $g_{x,y} : S^1 \to \R$ by
  \begin{equation*}
    g_{x,y} =
    \begin{cases}
      \chi_{x-y, x} & y < x, \\
      \chi_{x, x-y} & x < y.
    \end{cases}
  \end{equation*}
  Using Fubini-Tonelli and the fact that every $f \in W^{1,1}$ is absolutely continuous, we have
  \begin{equation*}\begin{split}
    \norm{\mathcal{J}_n f - f}_{L^1} &= \intf \abs{(J_n \ast f)(x) - f(x)} d\Leb(x)\\
    &= \intf \intf J_n(y) \abs{f(x-y) - f(x)} d\Leb(y) d\Leb(x) \\
    &\le \intf \intf \intf J_n(y) g_{x,y}(z) \abs{f'(z)} d\Leb(z) d\Leb(y) d\Leb(x) \\
    &= \intf J_n(y)\left( \intf  \abs{f'(z)} \left(\intf g_{x,y}(z) d\Leb(x)\right) d\Leb(z)\right) d\Leb(y) \\
    &= \left(\intf J_n(y) \abs{y} d\Leb(y) \right)\norm{f}_{W^{1,1}}.
  \end{split}\end{equation*}
  Since $\{J_n\}_{n \in \Z^+}$ approximates the identity we have $\intf J_n(y) \abs{y} d\Leb(y) \to 0$ as $n \to \infty$, which yields \eqref{eq:fejer_props_1} for $k = 1$.
  The claim for $k > 1$ follows from the case where $k = 1$, the fact that differentiation commutes with $\mathcal{J}_n$, and the fact that $W^{k,1}(S^1)$ continuously embeds into $W^{k-1,1}(S^1)$ and $W^{1,1}(S^1)$.
\end{proof}

\begin{proof}[{The proof of Theorem \ref{thm:random_fejer_approx}}]
  The proof is very similar to that of Theorem \ref{thm:random_deterministic_perturbation}.
  By assumption $(\Omega, \mathcal{F}, \mathbb{P}, \sigma, W^{k-1,1}(S^1), L_{\mathcal{T}})$ is a separable strongly measurable random linear system with ergodic invertible base and a hyperbolic Oseledets splitting of dimension $d \in \Z^+$.
  We have $L_{\mathcal{T}}  \in \End_{S}(\mathbb{W}^{k-1}, \sigma) \cap \mathcal{LY}(1,A_{k-1,\alpha, K},\alpha^{k-1},R_{k-1,\alpha, K})$ by Proposition \ref{prop:ly_ss}, and so all the requirements of Theorem \ref{thm:stability_lyapunov} are verified for $(\Omega, \mathcal{F}, \mathbb{P}, \sigma, W^{k-1,1}(S^1), L_{\mathcal{T}})$.
  Since the composition of strongly measurable maps is strongly measurable (\cite[Lemma A.5]{GTQuas1}), and the constant map $\omega \mapsto \mathcal{J}_n$ is strongly measurable, for each $n \in \Z^+$ we have that $(\Omega, \mathcal{F}, \mathbb{P}, \sigma, W^{k-1,1}(S^1), L_{\mathcal{T},n})$ is a separable strongly measurable random linear system.
  Since $\mathcal{J}_n$ is a contraction on $W^{k-1,1}$, from the Lasota-Yorke inequality \eqref{eq:pf_summary_0} we have for every $f\in W^{k-1,1}(S^1)$, $n \in \Z^+$ and $\omega \in \Omega$ that
  \begin{equation*}
    \norm{\mathcal{J}_n L_{\mathcal{T}(\omega)} f }_{W^{k-1,1}} \le \alpha^{k-1} \norm{ f }_{W^{k-1,1}} + C_{k-1,\alpha,K} \norm{f}_{W^{k-2,1}}.
  \end{equation*}
  By using the fact that $\mathcal{J}_n$ is a contraction on $W^{k-2,1}(S^1)$ and repeating the argument made in Proposition \ref{prop:ly_ss}, we deduce that $L_{\mathcal{T},n} \in \mathcal{LY}(1,A_{k-1,\alpha, K},\alpha^{k-1},R_{k-1,\alpha, K})$ for every $n \in \Z^+$.
  Thus, after using \eqref{eq:fejer_props_1} from Proposition \ref{prop:fejer_props} we obtain the conclusion of Theorem \ref{thm:stability_lyapunov} for the perturbation $(\Omega, \mathcal{F}, \mathbb{P}, \sigma, W^{1,1}(S^1), L_{\mathcal{T},n})$, as required.
\end{proof}

\section*{Acknowledgements}

H.C. is supported by an Australian Government Research Training Program Scholarship and the School of Mathematics and Statistics, UNSW.
He would like to thank Gary Froyland and Jason Atnip for their encouragement, helpful feedback and many stimulating discussions while writing this paper, and to Cecilia Gonz{\'a}lez-Tokman for helpfully answering a number of technical questions.
In addition, he would like to thank Fadi Antown for the many discussions regarding this paper that were had over lunch.
He would like to thank Philippe Thieullen for introducing him to the graph-representation of the Grassmannian and the associated graph transform, which form the technical backbone of this paper, and to Campus France and the Institut de Math{\'e}matiques de Bordeaux for providing funding for and hosting him during, respectively, his visit to Philippe in 2018.

\appendix

\section{Proofs for Section \ref{sec:prelims}}\label{app:preliminaries_proofs}

\subsection{Proofs for Section \ref{sec:grassmannian}}

\begin{proof}[{The proof of Lemma \ref{lemma:inv_proj}}]
  The map $\restr{\Pi_{E_1 || F}}{E_2}$ must be injective as otherwise $F \cap E_2 = \ker \Pi_{E_1 || F} \cap E_2 \ne \{0\}$, which contradicts $E_2 \oplus F = X$. To see that the map is surjective, we note that for any $e \in E_1$ one has $e = \Pi_{E_2 || F}e + \Pi_{F || E_2}e$ and so $\Pi_{E_1 || F}(\Pi_{E_2 || F}e) = e$.
\end{proof}

\begin{lemma}\label{lemma:proj_graph}
  If $E,  E' \in \mathcal{N}(F)$ then
  \begin{equation*}
    \Pi_{E' || F} = (\Id + \Phi_{E \oplus F}(E'))\Pi_{E || F}.
  \end{equation*}
  \begin{proof}
    Since $\ker(\Pi_{E || F}) = F$ we have $\ker((\Id + \Phi_{E \oplus F}(E'))\Pi_{E || F})) \subseteq F$ and so it suffices to prove that the restriction of $(\Id + \Phi_{E \oplus F}(E'))\Pi_{E || F}$ to $E'$ is the identity, but
    by definition we have
    \begin{equation*}
      \restr{(\Id + \Phi_{E \oplus F}(E'))\Pi_{E || F}}{E'} = \left(\restr{\Pi_{E || F}}{E'}\right)^{-1}\left(\restr{\Pi_{E || F}}{E'}\right) = \restr{\Id}{E'}.
    \end{equation*}
  \end{proof}
\end{lemma}

\begin{lemma}\label{lemma:graph_rep_inv}
  If $E \oplus F = X$ then $\Phi_{E \oplus F}^{-1} :  \mathcal{L}(E,F) \to \mathcal{N}(F)$ exists and, for each $L \in \LL(E,F)$, satisfies
  \begin{equation*}
    \Phi_{E \oplus F}^{-1}(L) = (\Id + L)(E).
  \end{equation*}
  \begin{proof}
    If $\Phi_{E \oplus F}(E_1) = \Phi_{E \oplus F}(E_2)$ for some $E_1,E_2 \in \mathcal{N}(F)$, then, when restricted to $E$,
    \begin{equation*}
      \left(\restr{\Pi_{E || F}}{E_1}\right)^{-1} = \left(\restr{\Pi_{E || F}}{E_2}\right)^{-1},
    \end{equation*}
    which implies that $E_1 = E_2$. Thus $\Phi_{E \oplus F}$ is injective.
    We now show that $\Phi_{E \oplus F}$ is surjective. Suppose that $L \in \LL(E,F)$. Then, as $L$ is bounded, $E' = (\Id + L)(E) \in \mathcal{G}(X)$. If $E' \cap F \ne \{0\}$ then there exists $e \in E \setminus \{0\}$ such that $e + L(e) \in F$, but this implies implies that $e \in F$, which is impossible. Thus $E' \cap F = \{0\}$.
    On the other hand, if $v \in X$ then, by writing $v = e + f$ according to the splitting $E \oplus F$, we observe that $v = (\Id + L)(e) + f - L(e)$, and so $E' + F = X$. Thus $E'$ and $F$ are topologically complementary subspaces by the Closed Graph Theorem.
    For every $e \in E$ we have $e + L(e) \in E'$ and $\Pi_{E || F}(e + L(e)) = e$, which implies that
    \begin{equation*}
      \Phi_{E \oplus F}(E') = \left(\restr{\Pi_{E || F}}{(\Id + L)(E)}\right)^{-1} - \Id = \Id + L - \Id = L,
    \end{equation*}
    and so $\Phi_{E \oplus F}$ is surjective.
    Moreover, since $E' = (\Id + L)E$, we see that $\Phi_{E \oplus F}^{-1}(L) = (\Id + L)(E)$.
  \end{proof}
\end{lemma}

\begin{proof}[{The proof of Lemma \ref{lemma:graph_rep_continuity}}]
  For every $\epsilon > 0$ there exists $u_1 \in \Phi_{E \oplus F}^{-1}(L_1)$ with $\norm{u_1} = 1$ such that
  \begin{equation*}
    \gap(\Phi_{E \oplus F}^{-1}(L_1), \Phi_{E \oplus F}^{-1}(L_2)) \le \epsilon + \inf_{u_2 \in \Phi_{E \oplus F}^{-1}(L_2)} \norm{u_1 - u_2}.
  \end{equation*}
  Since $L_1 = \left(\restr{\Pi_{E || F}}{\Phi_{E \oplus F}^{-1}(L_1)}\right)^{-1} - \Id$, by taking $u = \Pi_{E || F} u_1$ and $u_2 = (\Id + L_2)u$ we get
  \begin{equation*}\begin{split}
    \gap(\Phi_{E \oplus F}^{-1}(L_1), \Phi_{E \oplus F}^{-1}(L_2)) \le \epsilon + \norm{(\Id+L_1)u - (\Id + L_2)u} \le  \epsilon + \norm{L_1 - L_2} \norm{\Pi_{E || F}}.
  \end{split}\end{equation*}
  Since $\epsilon$ was arbitrary the same inequality holds for $\epsilon = 0$. The same bound clearly holds $\gap(\Phi_{E \oplus F}^{-1}(L_2),\Phi_{E \oplus F}^{-1}(L_1))$, and so we obtain the required inequality from \eqref{eq:gap_equiv}.
\end{proof}

\begin{proof}[{The proof of Lemma \ref{lemma:graph_rep_inv_continuity}}]
  By Lemma \ref{lemma:proj_graph} we have $\Pi_{E_i || F} = (\Id + \Phi_{E \oplus F}(E_i))\Pi_{E || F}$ for $i = 1, 2$, and so $\Pi_{E_i || F} - \Pi_{E || F} = \Phi_{E \oplus F}(E_i)$ on $E$. Thus
  \begin{equation*}
    \norm{\Phi_{E \oplus F}(E_1) - \Phi_{E \oplus F}(E_2)} = \norm{\Pi_{E_1 || F} - \Pi_{E_2 || F}} = \norm{\Pi_{F || E_1}\Pi_{E_2 || F}} \le \norm{\restr{\Pi_{F || E_1}}{E_2}}\norm{\Pi_{E_2 || F}}.
  \end{equation*}
  For $u_i \in E_i$ with $\norm{u_i} = 1$, $i \in \{ 1,2\}$, we have
  \begin{equation*}
    \norm{\Pi_{F || E_1}}\norm{u_2 - u_1} \ge \norm{\Pi_{F || E_1}(u_2 - u_1)} = \norm{\Pi_{F || E_1}u_2},
  \end{equation*}
  Taking the infimum over $u_1$ and then the supremum over $u_2$ yields
  \begin{equation*}\begin{split}
    \sup_{\substack{u_2 \in E_2 \\ \norm{u_2} = 1}} \inf_{\substack{u_1 \in E_2 \\ \norm{u_1} = 1}} \norm{u_1 - u_2} &\ge \left(\norm{\Pi_{F || E_1}}\right)^{-1} \norm{\restr{\Pi_{F || E_1}}{E_2}} \\
    &\ge \left(\norm{\Pi_{F || E_1}}\norm{\Pi_{E_2 || F}}\right)^{-1}\norm{\Phi_{E \oplus F}(E_1) - \Phi_{E \oplus F}(E_2)}.
  \end{split}\end{equation*}
  We obtain the required inequality upon noting that the roles of $E_1$ and $E_2$ may be swapped in the above argument, and then recalling the definition of $d_H$.
\end{proof}

\subsection{Proofs for Section \ref{sec:ss_primer}}

\begin{proof}[{The proof of Lemma \ref{lemma:saks_boundedness}}]
  Since $A$ is continuous it maps $\gamma[\norm{\cdot}_1, \wnorm{\cdot}_1]$-bounded sets to $\gamma[\norm{\cdot}_2, \wnorm{\cdot}_2]$-bounded sets.
  As $B_{\norm{\cdot}_1}$ is $\gamma[\norm{\cdot}_1, \wnorm{\cdot}_1]$-bounded it follows that $A(B_{\norm{\cdot}_1})$ is $\gamma[\norm{\cdot}_2, \wnorm{\cdot}_2]$-bounded.
  Proposition \ref{prop:ss_bounded} says that the $\gamma[\norm{\cdot}_2, \wnorm{\cdot}_2]$-bounded sets are exactly the $\norm{\cdot}_2$-bounded sets, and so $A(B_{\norm{\cdot}_1})$ is $\norm{\cdot}_2$-bounded. Thus $A \in \LL(X_1, X_2)$.
\end{proof}

\begin{proof}[{The proof of Proposition \ref{prop:linear_ops_saks}}]
  Without loss of generality we may assume that both $X_1$ and $X_2$ are normal Saks spaces. By Lemma \ref{lemma:saks_boundedness} we have $\LL_{S}(X_1, X_2) \subseteq \LL(X_1,X_2)$.
  Since $\tnorm{\cdot} \le \norm{\cdot}$ it follows that $\tnorm{\cdot}$ is finite on $\LL_{S}(X_1, X_2)$, and that $B_{\norm{\cdot}}$ is $\tnorm{\cdot}$-bounded.
  It remains to verify one of the conditions from Lemma \ref{lemma:equiv}: we will show that $B_{\norm{\cdot}}$ is $\tnorm{\cdot}$-closed.
  Suppose that $\{A_n\}_{n \in \Z^+} \subseteq B_{\norm{\cdot}} \cap \LL_S(X_1,X_2)$ is a $\tnorm{\cdot}$-convergent sequence with limit $A \in \LL_S(X_1,X_2)$.
  For every $\epsilon > 0$ there exists $f_\epsilon \in X_1$ with $\norm{f_\epsilon}_1  =1 $ such that $\norm{A} \le \norm{A f_\epsilon }_2 +\epsilon$.
  Since $\lim_{n \to \infty} \tnorm{A_n - A} = 0$ we have $\lim_{n \to \infty} \wnorm{(A_n - A)f_\epsilon}_2 = 0$. Moreover, $\norm{A_n f_\epsilon}_2 \le 1$ for every $n \in \Z^+$.
  Since $(X_2, \norm{\cdot}_2, \wnorm{\cdot}_2)$ satisfies each of the conditions in Lemma \ref{lemma:equiv} we have $\norm{A f_\epsilon}_2 \le 1$, which implies that $\norm{A} \le 1 + \epsilon$ for every $\epsilon > 0$.
  Hence $\norm{A} \le 1$ and so $A \in
  B_{\norm{\cdot}}$ i.e. $B_{\norm{\cdot}}$ is $\tnorm{\cdot}$-closed in $\LL_{S}(X_1, X_2)$.
  It follows that $(\LL_S(X_1,X_2), \norm{\cdot}, \tnorm{\cdot})$ is a Saks space, as claimed.
\end{proof}

\begin{proof}[{The proof of Proposition \ref{prop:saks_equicont}}]
  We prove the reverse implication first i.e. that for every $\gamma[\norm{\cdot}_2, \wnorm{\cdot}_2]$-open neighbourhood $U$ of $0$ there exists a $\gamma[\norm{\cdot}_1, \wnorm{\cdot}_1]$-open neighbourhood $V$ of $0$ such that $A_\alpha(V) \subseteq U$ for every $\alpha \in \mathcal{A}$.
  If $U$ is such a neighbourhood then, after recalling the form of the neighbourhood basis for $\gamma[\norm{\cdot}_2, \wnorm{\cdot}_2]$ from \eqref{eq:ss_neighbourhood_basis}, we observe that $U$ contains a set of the form
  \begin{equation*}
    \bigcup_{n =1}^\infty \left(\sum_{k=1}^n 2^{\ell_k} B_{\wnorm{\cdot}_2}^\mathrm{o} \cap 2^{k-1} B_{\norm{\cdot}_2}\right)
  \end{equation*}
  for some sequence $\{\ell_i\}_{i=1}^\infty \subseteq \R$, where $B_{\wnorm{\cdot}_1}^\mathrm{o}$ denotes the open unit $\wnorm{\cdot}_1$-ball.
  By assumption we have $\sup_{\alpha} \norm{A_\alpha} < \infty$ and so there exists $N \ge 0$ such that $A_\alpha(B_{\norm{\cdot}_1}) \subseteq 2^{N} B_{\norm{\cdot}_2}$ for every $\alpha \in \mathcal{A}$.
  For the moment fix $k > N$, and let us suppose that $f \in 2^{\nu_k}B_{\wnorm{\cdot}_1}^\mathrm{o} \cap 2^{k - N - 1}B_{\norm{\cdot}_1}$ for some $\nu_k \in \R$.
  Then $A_\alpha f \in 2^{k-N + N -1}B_{\norm{\cdot}_2} = 2^{k - 1} B_{\norm{\cdot}_2}$ for every $\alpha \in \mathcal{A}$.
  In addition, by \eqref{eq:saks_equicont_0} we have for every $\eta > 0$ and $\alpha \in \mathcal{A}$ that
  \begin{equation*}
      \wnorm{A_\alpha f}_{2} \le \max\{\eta 2^{k - N -1}, C_{\eta} 2^{\nu_k} \},
  \end{equation*}
  so that if we set $\eta_k = 2^{\ell_k + N -k}$ and take $2^{\nu_k} = C_{2^{\ell_k + N -k }}^{-1} 2^{\ell_k - 1}$ then $\wnorm{A_\alpha f}_{2} < 2^{\ell_k}$ for every $\alpha \in \mathcal{A}$.
  Thus for every $\alpha \in \mathcal{A}$ and $k \ge 1$ we have
  \begin{equation*}
    A_\alpha(2^{\nu_k}B_{\wnorm{\cdot}_1}^\mathrm{o} \cap 2^{k - N - 1}B_{\norm{\cdot}_1}) \subseteq 2^{\ell_k} B_{\wnorm{\cdot}_1}^\mathrm{o} \cap 2^{k - 1} B_{\norm{\cdot}_2}.
  \end{equation*}
  Set
  \begin{equation*}
    V = \bigcup_{n =1}^\infty \left(\sum_{k=N + 1}^{n+N} 2^{\nu_{k}} B_{\wnorm{\cdot}_2}^\mathrm{o} \cap 2^{k-N - 1} B_{\norm{\cdot}_2}\right),
  \end{equation*}
  and note that $V$ is open per \eqref{eq:ss_neighbourhood_basis}.
  Moreover, for every $\alpha \in \mathcal{A}$ we have
  \begin{equation*}\begin{split}
    A_\alpha (V) &= \bigcup_{n =1}^\infty \left(\sum_{k=N + 1}^{n+N} A_\alpha(2^{\nu_{k}} B_{\wnorm{\cdot}_2}^\mathrm{o} \cap 2^{k-N - 1} B_{\norm{\cdot}_2})\right) \\
    &\subseteq \bigcup_{n =1}^\infty \left(\sum_{k=N + 1}^{n+N} 2^{\ell_k} B_{\wnorm{\cdot}_1}^\mathrm{o} \cap 2^{k - 1} B_{\norm{\cdot}_2}\right) \subseteq U.
  \end{split}\end{equation*}
  Thus $\{A_\alpha\}_{\alpha \in \mathcal{A}}$ is equicontinuous in $\LL_S(X_1, X_2)$.

  We will now prove the opposite implication.
  Let $U$ be any $\gamma[\norm{\cdot}_2, \wnorm{\cdot}_2]$-neighbourhood of $0$.
  Since $\{A_\alpha\}_{\alpha \in \mathcal{A}}$ is equicontinuous in $\LL_S(X_1, X_2)$, by \cite[4.1]{schaefer1986topological} the set $V = \cap_{\alpha \in \mathcal{A}} A_\alpha^{-1}(U)$ is $\gamma[\norm{\cdot}_1, \wnorm{\cdot}_1]$-open.
  Since $V$ is $\gamma[\norm{\cdot}_1, \wnorm{\cdot}_1]$-open and $B_{\norm{\cdot}_1}$ is $\gamma[\norm{\cdot}_1, \wnorm{\cdot}_1]$-bounded, there exists $\lambda > 0$ such that $B_{\norm{\cdot}_1} \subseteq \lambda V$.
  Hence $A_\alpha(B_{\norm{\cdot}_1}) \subseteq \lambda U$ for every $\alpha \in \mathcal{A}$ and so $\bigcup_{\alpha \in \mathcal{A}} A_\alpha(B_{\norm{\cdot}_1})$ is $\gamma[\norm{\cdot}_2, \wnorm{\cdot}_2]$-bounded.
  By Proposition \ref{prop:ss_bounded} it follows that $\bigcup_{\alpha \in \mathcal{A}} A_\alpha(B_{\norm{\cdot}_1})$ is $\norm{\cdot}_2$-bounded i.e. there exists $M > 0$ such that $A_\alpha(B_{\norm{\cdot}_1}) \subseteq M B_{\norm{\cdot}_2}$ for every $\alpha \in \mathcal{A}$.
  Hence $\{A_\alpha\}_{\alpha \in \mathcal{A}}$ is bounded in $\LL(X_1,X_2)$, which implies that $\{A_\alpha\}_{\alpha \in \mathcal{A}}$ is equicontinuous as a subset of $\LL(X_1,X_2)$.

  It remains to prove \eqref{eq:saks_equicont_0}.
  Fix $\eta > 0$. Since $\{A_\alpha\}_{\alpha \in \mathcal{A}}$ is equicontinuous in $\LL_S(X_1, X_2)$ and $\wnorm{\cdot}_2$ is $\gamma[\norm{\cdot}_2, \wnorm{\cdot}_2]$-continuous there exists a $\gamma[\norm{\cdot}_1, \wnorm{\cdot}_1]$-neighbourhood of 0, say $U_\eta$, such that if $f \in U_\eta$ and $\alpha \in \mathcal{A}$ then $\wnorm{A_\alpha(f)}_2 < \eta$.
  By construction $U_\eta$ contains a set of the form $2^{\ell_\eta} B_{\wnorm{\cdot}_1}^\mathrm{o} \cap B_{\norm{\cdot}_1}$ for some $\ell_\eta \in \R$. For any non-zero $f \in X_1$ we have
  \begin{equation*}
    \frac{f}{\max\{2^{-\ell_\eta +1} \wnorm{f}_1, \norm{f}_1\}} \in 2^{\ell_\eta} B_{\wnorm{\cdot}_1}^\mathrm{o} \cap B_{\norm{\cdot}_1},
  \end{equation*}
  and so for every $\alpha \in \mathcal{A}$ we have
  \begin{equation*}
    \wnorm{A_\alpha \left(\frac{f}{\max\{2^{-\ell_\eta+1} \wnorm{f}_1, \norm{f}_1\}}\right)}_2 < \eta.
  \end{equation*}
  In particular, if we set $C_\eta = \eta 2^{-\ell_\eta+1}$ then for every $f \in X_1$ and $\alpha \in \mathcal{A}$ we have
  \begin{equation*}
    \wnorm{A_\alpha(f)}_2 < \max\{\eta \norm{f}_1, C_\eta \wnorm{f}_1\},
  \end{equation*}
  as required.
\end{proof}

\begin{proof}[{The proof of Proposition \ref{prop:operator_saks_space_complete}}]
  Suppose $\{A_n \}_{n \in \Z^+} \subseteq \LL_S(X_1,X_2)$ is $\norm{\cdot}$-bounded and $\tnorm{\cdot}$-Cauchy.
  Then for every $f \in X_1$ the sequence $\{A_n f\}_{n \in \Z^+}$ is $\norm{\cdot}_{2}$-bounded and $\wnorm{\cdot}_2$-Cauchy, and so there exists $g \in X_2$ such that $A_n f \to g$ in $\gamma[\norm{\cdot}_2, \wnorm{\cdot}_2]$.
  Define $A: X_1 \to X_2$ by $A f = \lim_{n \to \infty} A_n f$, and  note that $\tnorm{A - A_n} \to 0$ due to the fact that $\{A_n \}_{n \in \Z^+} \subseteq \LL_S(X_1,X_2)$ is $\tnorm{\cdot}$-Cauchy.
  We will use Proposition \ref{prop:saks_equicont} to prove that $A \in \LL_S(X_1, X_2)$.
  Since $\norm{\cdot}$ is lower semicontinuous for $\tnorm{\cdot}$ we have
  \begin{equation*}
    \sup_{\substack{f \in X_1 \\ \norm{f} = 1}} \norm{A f} \le \sup_{\substack{f \in X_1 \\ \norm{f} = 1}} \liminf_{n \to \infty}\norm{A_n f} \le \sup_{n \in \Z^+} \norm{A_n},
  \end{equation*}
  and so $A \in \LL(X_1, X_2)$.
  Fix $\eta > 0$. For every $f \in X_1$ and $n \in \Z^+$ we have
  \begin{equation*}
    \wnorm{A f}_2 \le \wnorm{A_n f}_2 + \tnorm{A_n - A} \norm{f}_1.
  \end{equation*}
  By Proposition \ref{prop:saks_equicont} and as $\{A_n\}_{n \in \Z^+} \subseteq \LL_{S}(X_1, X_2)$, for each $n \in \Z^+$ and $\kappa > 0$ there exists $D_{\kappa,n}$ such that for every $f \in X_1$ we have $\wnorm{A_n f}_2 \le \kappa \norm{f}_1 + D_{\kappa, n} \wnorm{f}_1$.
  Thus
  \begin{equation*}
    \wnorm{A f}_2 \le (\kappa + \tnorm{A_n - A}) \norm{f}_1 + D_{\kappa, n} \wnorm{f}_1.
  \end{equation*}
  Suppose that $n$ is large enough so that $\tnorm{A_n - A} \le \eta /4 $. Set $\kappa = \eta / 4$ and $C_\eta = 2D_{\kappa, n}$. Then
  \begin{equation*}
    \wnorm{A f}_2 \le \frac{\eta}{2} \norm{f}_1 + D_{\kappa, n} \wnorm{f}_1 \le 2\max\left\{\frac{\eta}{2} \norm{f}_1, D_{\kappa,n} \wnorm{f}_1\right\} = \max\{ \eta \norm{f}_1, C_{\eta} \wnorm{f}_1\}.
  \end{equation*}
  Hence $A \in \LL_S(X_1, X_2)$ by Proposition \ref{prop:saks_equicont}.
\end{proof}

\begin{proof}[{The proof of Theorem \ref{thm:compact_saks_unique}}]
  We will prove that the set of $\gamma[\norm{\cdot}, \tau]$-convergent nets coincides with the set of $\gamma[\norm{\cdot}, \tau']$-convergent nets.
  Suppose that $\{f_\alpha\}_{\alpha \in \mathcal{A}}$ is a $\gamma[\norm{\cdot}, \tau]$-convergent net and, for a contradiction, that $\{f_\alpha\}_{\alpha \in \mathcal{A}}$ is not $\gamma[\norm{\cdot}, \tau']$-convergent.
  By continuity, $\{f_\alpha\}_{\alpha \in \mathcal{A}}$ must be convergent in $D$.
  Since $(X, \norm{\cdot}, \tau')$ is a compact Saks space and $\sup_{\alpha \in \mathcal{A}} \norm{f_\alpha} < \infty$ there exists a $\tau'$-convergent sub-net $\{f_\alpha\}_{\alpha \in \mathcal{A}'}$, which accumulates away from the $\tau$-limit of $\{f_\alpha\}_{\alpha \in \mathcal{A}}$.
  But $\{f_\alpha\}_{\alpha \in \mathcal{A}'}$ can only accumulate in $D$ at the accumulation point of $\{f_\alpha\}_{\alpha \in \mathcal{A}}$, and so the two accumulation points must be the same.
\end{proof}

\begin{proof}[{The proof of Theorem \ref{thm:saks_space_norm_formula}}]
  By Proposition \ref{prop:saks_space_finite} there exists a separable Banach space $F$ such that $(X, \norm{\cdot}, \wnorm{\cdot}) = (F^*, \norm{\cdot}, \sigma(F^*, F))$. Specifically, $\sigma(F^*, F)$ and $\wnorm{\cdot}$ are equivalent on $B_{\norm{\cdot}}$.
  Since both $\norm{\cdot}_{F^*}$ and $\norm{\cdot}$ are stronger than $\sigma(F^*, F)$, which is a Hausdorff topology on $F^*$, it follows from the closed graph theorem that $\norm{\cdot}_{F^*}$ and $\norm{\cdot}$ are equivalent (see \cite[Proposition 1]{terry_tao_lemma}).
  Since $F$ is separable, there exists a $\norm{\cdot}_{F}$-bounded family $\Phi = \{\varphi_n\}_{n \in \Z^+} \subseteq F^* \cap F$ that is $\norm{\cdot}_F$-dense in $\partial B_{\norm{\cdot}_F}$ and so that $\norm{\cdot}_{F^*} = \norm{\cdot}_{\Phi}$.
  Thus $\norm{\cdot}_{\Phi}$ and $\norm{\cdot}$ are equivalent, so it only remains to prove that $\wnorm{\cdot}$ is equivalent to $\wnorm{\cdot}_{\Phi}$ on $B_{\norm{\cdot}}$.
  Suppose that $\{f_k\}_{k \in \Z^+} \subseteq B_{\norm{\cdot}}$ is a $\wnorm{\cdot}$-convergent sequence with limit $f \in B_{\norm{\cdot}}$. Fix $\epsilon > 0$. Since $\{f_k\}_{k \in \Z^+} \subseteq B_{\norm{\cdot}}$ there exists $N$ such that
  \begin{equation*}
    \sum_{n = N + 1}^\infty 2^{-n} \abs{\varphi_n(f-f_k)} \le \epsilon.
  \end{equation*}
  Hence
  \begin{equation*}
    \wnorm{f - f_k}_\Phi = \sum_{n =1 }^\infty 2^{-n}  \abs{\varphi_n(f-f_k)} \le \sum_{n =1 }^N 2^{-n} \abs{\varphi_n(f-f_k)} + \epsilon.
  \end{equation*}
  Viewing $f-f_k$ as an element of $F^*$, each $\varphi_n$ as an element of $F$, and using the equivalence of $\sigma(F^*,F)$ and $\wnorm{\cdot}$ on $B_{\norm{\cdot}}$, we observe that $\limsup_{k \to \infty}\wnorm{f - f_k}_\Phi \le \epsilon$. Thus, as $\epsilon$ is arbitrary, we have $\lim_{k \to \infty} \wnorm{f_k - f}_\Phi = 0$.
  On the other hand, if $\{f_k\}_{k \in \Z^+} \subseteq B_{\norm{\cdot}}$ is instead $\wnorm{\cdot}_\Phi$-convergent then we must have $\varphi_n(f-f_k) \to 0$ for every $n \in \Z^+$. Since $\Phi$ is $\norm{\cdot}_F$-dense in $\partial B_{\norm{\cdot}_F}$, we have $(f-f_k)(x) \to 0$ for every $x \in F$ i.e. $f \to f_k$ in $\sigma(F^*, F)$.
  Since $\sigma(F^*,F)$ and $\wnorm{\cdot}$ are equivalent on $B_{\norm{\cdot}}$, and $\sigma(F^*,F)$ is metrisable on $B_{\norm{\cdot}}$, it follows that $f \to f_k$ in $\wnorm{\cdot}$.
  Hence $\wnorm{\cdot}$ and $\wnorm{\cdot}_\Phi$ are equivalent on $B_{\norm{\cdot}}$, as required.
  Thus, $\gamma[\norm{\cdot}, \wnorm{\cdot}]$ and  $\gamma[\norm{\cdot}_\Phi, \wnorm{\cdot}_\Phi]$ are equivalent by Proposition \ref{prop:saks_space_weak_top}.
\end{proof}

\section{Proofs for Section \ref{sec:stability_lyapunov}}\label{app:stability_lyapunov_proofs}

In this section we prove some technical results from Section \ref{sec:stability_lyapunov} on the existence, continuity and measurability of certain maps.

Throughout this appendix $(X, \norm{\cdot})$ will denote a separable Banach space.
Let $\Delta_d = \{ A \in \LL(X) : A^2 = A \text{ and } \rank(A) = d\}$ denote the space of bounded $d$-dimensional projections on $X$, and set
\begin{equation*}
  \Lambda_d = \left\{ (A, \Pi_1, \Pi_2) \in  \LL(X) \times \Delta_d \times \Delta_d \,\mid\, \restr{\Pi_{2} A}{\Pi_1(X)} : \Pi_1(X) \to \Pi_2(X) \text{ is invertible } \right\}.
\end{equation*}
Let $\Gamma^* : \Lambda_d \to \Delta_d$ be defined by
\begin{equation*}
  \Gamma^*(A, \Pi_1, \Pi_2) = (\Id + A^*(0))\Pi_2,
\end{equation*}
where $A^*$ is understood as the forward graph transform from $\LL(\Pi_1(X), \ker(\Pi_1))$ to $\LL(\Pi_2(X), \ker(\Pi_2))$,
and let $\Gamma_* : \Lambda_d \to \Delta_d$ be defined by
\begin{equation*}
    \Gamma_*(A, \Pi_1, \Pi_2) = \Pi_1 - A_*(0)(\Id - \Pi_1),
\end{equation*}
where $A_*$ is understood as the backward graph transform from $\LL(\ker(\Pi_2), \Pi_2(X))$ to $\LL(\ker(\Pi_1), \Pi_1(X))$.

The first result we will focus on proving is the following.
\begin{proposition}\label{prop:measurable_graph_tranform}
  Suppose that $(\Omega, \mathcal{F}, \mathbb{P})$ is a Lebesgue space, $X$ is a separable Banach space, and $Y : \Omega \mapsto \LL(X)^3$ is a $(\mathcal{F}, \mathcal{S}^3)$-measurable map with $Y(\Omega) \subseteq \Lambda_d$.
  Then $\Gamma^* \circ Y$ and $\Gamma_* \circ Y$ are $(\mathcal{F},\mathcal{S})$-measurable.
\end{proposition}

\begin{lemma}\label{lemma:strong_cont_proj_grassmannian}
  The map $\Psi : \Delta_d \to \mathcal{G}_d(X)$ defined by $\Psi(\Pi) = \Pi(X)$ is continuous with respect to the strong operator topology on $\Delta_d$.
  \begin{proof}
    Fix a normalised Auerbach basis $\{v_i\}_{i = 1}^d$ for $\Psi(\Pi)$ i.e. a basis such that
    \begin{equation*}
      \forall i \in \{1 ,\dots, d\} \quad \dist(v_i, \vspan\{ v_j : j \ne i\}) = 1.
    \end{equation*}
    For $\eta > 0$ set
    \begin{equation*}
      S_\eta := \Delta_d \cap \left( \bigcap_{i=1}^d \left\{ \Pi' \in \LL(X) : \norm{(\Pi - \Pi')v_i} < \eta \right\} \right),
    \end{equation*}
    and note that each $S_\eta$ is open in the strong operator topology.
    Set $\epsilon = 2^{-d-2}$ and let $NB_d^\epsilon(X)$ denote the set of $\epsilon$-nice bases for $d$-dimensional subspaces of $X$ (see \cite[Definition 2]{GTQuas1} for the relevant definition). We note that $\{v_i\}_{i = 1}^d \in NB_d^\epsilon(X)$.
    By \cite[Lemma B.8]{GTQuas1} there exists $\eta > 0$ so that if $\{w_i\}_{i = 1}^d$ satisfies $\sup_{i} \norm{v_i-w_i} < \eta'$ then $\{w_i\}_{i = 1}^d \in NB_d^\epsilon(X)$ too.
    Hence if $\Pi' \in S_\eta$ then $\{\Pi' v_i\}_{i=1}^d$ is a $\epsilon$-nice basis for $\Pi'(X)$.
    Moreover, the map $\Pi' \mapsto \{\Pi'v_i\}_{i=1}^d$ is continuous from $S_\eta$ to $NB_d^\epsilon(X)$.
    Thus by \cite[Corollary B.6]{GTQuas1} the map $\Pi' \mapsto \vspan\{\Pi' v_i \}_{i=1}^d = \Pi(X)$ is continuous from $S_\eta$ to $\mathcal{G}_d(X)$.
  \end{proof}
\end{lemma}

A fact we will use frequently throughout the following lemmas is that, as $X$ is separable, the strong operator topology is metrisable on $sB_{\LL(X)}$ for every $s > 0$. In particular, if $\{x_i\}_{i \in \Z^+}$ is a dense subset of $X$ then the topology induced by the metric
\begin{equation*}
  d(S,T) := \sum_{i=1}^\infty 2^{-i} \frac{\norm{(S-T)x_i}}{1 + \norm{(S-T)x_i}}
\end{equation*}
is equivalent to the strong operator topology on $\LL(X)$-bounded sets. Another nice fact about the restriction of the strong operator topology to $\LL(X)$-bounded sets, which we will use frequently, is that the composition map $(S,T) \mapsto S \circ T$ is continuous.

\begin{lemma}\label{lemma:open_sot}
  For every $s > 0$ the set $\Lambda_d \cap (s B_{\LL(X)})^3$ is open in the restriction of the strong operator topology to $(s B_{\LL(X)})^3$.
  \begin{proof}
    It suffices to prove that $(A, \Pi_1, \Pi_2) \mapsto \norm{\left(\restr{\Pi_2 A}{\Pi_1(X)}\right)^{-1}}$ is locally finite and upper semicontinuous on $\Lambda_d \cap (s B_{\LL(X)})^3$ i.e. that if $\{(A_\epsilon, \Pi_{1,\epsilon}, \Pi_{2,\epsilon}) \}_{\epsilon > 0} \subseteq ( sB_{\LL(X)})^3$ converges to $(A_0, \Pi_{1,0}, \Pi_{2,0}) \in \Lambda_d \cap ( sB_{\LL(X)})^3$ then
    \begin{equation}\label{eq:open_sot_0}
      \limsup_{\epsilon \to 0} \norm{\left(\restr{\Pi_{2,\epsilon} A_\epsilon}{\Pi_{1,\epsilon}(X)}\right)^{-1}} \le \norm{\left(\restr{\Pi_{2,0} A_0}{\Pi_{1,0}(X)}\right)^{-1}}.
    \end{equation}
    By definition, for each $v \in \Pi_{1,\epsilon}(X)$ with $\norm{v} = 1$ there exists $w \in \Pi_{1,\epsilon}(X)$ with $\norm{w} = 1$ such that $\norm{v-w} \le d_H( \Pi_{1,\epsilon}(X), \Pi_{1,0}(X))$. Hence, for every such $v$ and $w$,
    \begin{equation}\label{eq:open_sot_1}
      \norm{\Pi_{2,\epsilon} A_{\epsilon} v} \ge \norm{\Pi_{2,\epsilon} A_{\epsilon} w} - \norm{\Pi_{2,\epsilon} A_{\epsilon} (v-w)}
      \ge \norm{\Pi_{2,\epsilon} A_{\epsilon} w} - s^2 d_H( \Pi_{1,\epsilon}(X), \Pi_{1,0}(X)).
    \end{equation}
    Focusing on the first term yields
    \begin{equation}\begin{split}\label{eq:open_sot_2}
      \norm{\Pi_{2,\epsilon} A_{\epsilon} w} &\ge \norm{\Pi_{2,0} A_0 w } - \norm{\left(\Pi_{2,0} A_0 - \Pi_{2,\epsilon} A_{\epsilon}\right) w }\\
      &\ge \norm{\left(\restr{\Pi_{2,0} A_0 }{\Pi_{1,0}(X)}\right)^{-1}}^{-1}\left(1 - d_H( \Pi_{1,\epsilon}(X), \Pi_{1,0}(X))\right) - \norm{\left(\Pi_{2,0} A_0 - \Pi_{2,\epsilon} A_{\epsilon}\right) \Pi_{1,0} w }.
    \end{split}\end{equation}
    Fix a normalised Auerbach basis $\{w_i\}_{i=1}^d$ for $\Pi_{1,0}(X)$ and write $\Pi_{1,0}w = \sum_{i=1}^d a_i w_i$. Since $\{w_i\}_{i=1}^d$ is Auerbach, by \cite[Corollary A.7]{quas2019explicit} we have
    \begin{equation*}
      \left( \sum_{i=1}^d \abs{a_i}^2 \right)^{1/2} \le \sqrt{d} \norm{\Pi_{1,0}w} \le s \sqrt{d}.
    \end{equation*}
    Hence
    \begin{equation}\begin{split}\label{eq:open_sot_3}
      \norm{\left(\Pi_{2,0} A_0 - \Pi_{2,\epsilon} A_{\epsilon}\right) \Pi_{1,0} w } &\le \max_{i \in \{1, \dots, d\}} \norm{\left(\Pi_{2,0} A_0 - \Pi_{2,\epsilon} A_{\epsilon}\right) \Pi_{1,0} w_i} \left(\sum_{i=1}^d \abs{a_i}\right)\\
      &\le s d\max_{i \in \{1, \dots, d\}} \norm{\left(\Pi_{2,0} A_0 - \Pi_{2,\epsilon} A_{\epsilon}\right) \Pi_{1,0} w_i}.
    \end{split}\end{equation}
    Combining \eqref{eq:open_sot_1}, \eqref{eq:open_sot_2} and \eqref{eq:open_sot_3} yields a lower bound for $\norm{\Pi_{2,\epsilon} A_{\epsilon} v}$ that is uniform in $v \in \Pi_{1,\epsilon}(X)$ with $\norm{v}=1$.
    Moreover, using the facts that the right hand side of \eqref{eq:open_sot_3} vanishes as $\epsilon \to 0$, and that, due to Lemma \ref{lemma:strong_cont_proj_grassmannian}, $\lim_{\epsilon \to 0} \Pi_{1,\epsilon}(X) = \Pi_{1,0}(X)$ in $\mathcal{G}_d(X)$, we see that this lower bound converges to $\norm{\left(\restr{\Pi_{2,0} A_0}{\Pi_{1,0}(X)}\right)^{-1}}^{-1}$ as $\epsilon \to 0$, which completes the proof.
    \end{proof}
\end{lemma}

\begin{lemma}\label{lemma:inverse_sot_cont}
  Let $\Xi : \Lambda_d \to \LL(X)$ be defined by
  \begin{equation*}
    \Xi(A,\Pi_1, \Pi_2) = \left(\restr{\Pi_{2} A}{\Pi_{1}(X)} \right)^{-1}\Pi_{2}.
  \end{equation*}
  For each $s > 0$ the restriction of $\Xi$ to $\Lambda_d \cap (s B_{\LL(X)})^3$ is continuous in the strong operator topology.
  \begin{proof}
    Fix $(A_0, \Pi_{1,0}, \Pi_{2,0}) \in  \Lambda_d \cap ( sB_{\LL(X)})^3$. By \eqref{eq:open_sot_0} there exists a neighbourhood $U \subseteq \Lambda_d \cap (sB_{\LL(X)})^3$ of $(A_0, \Pi_{1,0}, \Pi_{2,0})$ that is open in the strong operator topology and such that
    \begin{equation}
    \sup\left\{ \norm{\Xi(A,\Pi_1, \Pi_2)} : (A,\Pi_1, \Pi_2) \in U \right\} < \infty.
    \end{equation}
    Thus $\Xi(U)$ is bounded in $\LL(X)$, and $\LL(X)$ is therefore metrisable on $\Xi(U)$.
    Hence, to prove that $\Xi$ is continuous at $(A_0, \Pi_{1,0}, \Pi_{2,0})$ it suffices to show that if $\{(A_\epsilon, \Pi_{1,\epsilon}, \Pi_{2,\epsilon}) \}_{\epsilon > 0} \subseteq \Lambda_d \cap ( sB_{\LL(X)})^3$ converges to $(A_0, \Pi_{1,0}, \Pi_{2,0}) \in  \Lambda_d \cap ( sB_{\LL(X)})^3$ then $\Xi(A_\epsilon,\Pi_{1,\epsilon}, \Pi_{2,\epsilon}) \to \Xi(A_0,\Pi_{1,0}, \Pi_{2,0})$.
    Fix $v \in X$ with $\norm{v} = 1$. We have
    \begin{equation}\begin{split}\label{eq:inverse_sot_cont_1}
      \norm{\Xi(A_0,\Pi_{1,0}, \Pi_{2,0})v - \Xi(A_\epsilon,\Pi_{1,\epsilon}, \Pi_{2,\epsilon})v}
      &\le \norm{ \Pi_{1,\epsilon}\left(\restr{\Pi_{2,0} A_0}{\Pi_{1,0}(X)} \right)^{-1}\Pi_{2,0}v -  \left(\restr{\Pi_{2,\epsilon} A_\epsilon}{\Pi_{1,\epsilon}(X)} \right)^{-1}\Pi_{2,\epsilon}v} \\
      &+ \norm{ (\Id - \Pi_{1,\epsilon})\Pi_{1,0}\left(\restr{\Pi_{2,0} A_0}{\Pi_{1,0}(X)} \right)^{-1}\Pi_{2,0}v}.
    \end{split}\end{equation}
    We of course have
    \begin{equation}\label{eq:inverse_sot_cont_2}
      \lim_{\epsilon \to 0} \norm{ (\Id - \Pi_{1,\epsilon})\Pi_{1,0}\left(\restr{\Pi_{2,0} A_0}{\Pi_{1,0}(X)} \right)^{-1}\Pi_{2,0} v} = \norm{ (\Id - \Pi_{1,0})\Pi_{1,0}\left(\restr{\Pi_{2,0} A_0}{\Pi_{1,0}(X)} \right)^{-1}\Pi_{2,0} v} = 0.
    \end{equation}
    On the other hand we have
    \begin{equation*}\begin{split}
      \Pi_{1,\epsilon}&\left(\restr{\Pi_{2,0} A_0}{\Pi_{1,0}(X)} \right)^{-1}\Pi_{2,0}v - \left(\restr{\Pi_{2,\epsilon} A_\epsilon}{\Pi_{1,\epsilon}(X)} \right)^{-1}\Pi_{2,\epsilon}v\\
      &= \Xi(A_\epsilon,\Pi_{1,\epsilon}, \Pi_{2,\epsilon})\left(\left(\Pi_{2,\epsilon}A_\epsilon \Pi_{1,\epsilon} - \Pi_{2,0}A_0\Pi_{1,0} \right)\Xi(A_0,\Pi_{1,0}, \Pi_{2,0}) - \Pi_{2,\epsilon}(\Id - \Pi_{2,0})\right)v.
    \end{split}\end{equation*}
    Applying \eqref{eq:open_sot_0} from the proof of Lemma \ref{lemma:open_sot}, we have
    \begin{equation}\begin{split}\label{eq:inverse_sot_cont_3}
      &\limsup_{\epsilon \to 0} \bigg\lVert\Pi_{1,\epsilon}\left(\restr{\Pi_{2,0} A_0}{\Pi_{1,0}(X)} \right)^{-1}\Pi_{2,0}v - \left(\restr{\Pi_{2,\epsilon} A_\epsilon}{\Pi_{1,\epsilon}(X)} \right)^{-1}\Pi_{2,\epsilon}v\bigg\rVert \\
      &\le
      \norm{\Xi(A_0,\Pi_{1,0}, \Pi_{2,0})} \limsup_{\epsilon \to 0} \norm{\left(\left(\Pi_{2,\epsilon}A_\epsilon \Pi_{1,\epsilon} - \Pi_{2,0}A_0\Pi_{1,0}  \right)\Xi(A_0,\Pi_{1,0}, \Pi_{2,0}) - \Pi_{2,\epsilon}(\Id - \Pi_{2,0})\right)v} =0.
    \end{split}\end{equation}
    We obtain the required claim by applying \eqref{eq:inverse_sot_cont_2} and \eqref{eq:inverse_sot_cont_3} to \eqref{eq:inverse_sot_cont_1}.
  \end{proof}
\end{lemma}

\begin{lemma}\label{lemma:forward_transform_cont}
  For every $s > 0$ the maps $\restr{\Gamma^*}{(sB_{\LL(X)})^3}$ and $\restr{\Gamma_*}{(sB_{\LL(X)})^3}$ are continuous with respect to the strong operator topology.
  \begin{proof}
    We will just prove that $\restr{\Gamma^*}{(sB_{\LL(X)})^3}$ is continuous, since essentially the same proof applies to $\restr{\Gamma_*}{(sB_{\LL(X)})^3}$.
    Fix $(A_0,\Pi_{1,0}, \Pi_{2,0}) \in (sB_{\LL(X)})^3$. An argument similar to that at the beginning of Lemma \ref{lemma:inverse_sot_cont} shows that there is a neighbourhood $U \subseteq \Lambda_d \cap (sB_{\LL(X)})^3$ of $(A_0, \Pi_{1,0}, \Pi_{2,0})$ that is open in the strong operator topology and such that $\Gamma^*(U)$ is bounded. Therefore, to show that $\Gamma^*$ is continuous at $(A_0,\Pi_{1,0}, \Pi_{2,0})$ it suffices to prove that if $\{(A_\epsilon, \Pi_{1,\epsilon}, \Pi_{2,\epsilon}) \}_{\epsilon > 0} \subseteq \Lambda_d \cap ( sB_{\LL(X)})^3$ converges to $(A_0, \Pi_{1,0}, \Pi_{2,0})$ then $\Gamma^*(A_\epsilon,\Pi_{1,\epsilon}, \Pi_{2,\epsilon}) \to \Gamma^*(A_0,\Pi_{1,0}, \Pi_{2,0})$.
    By the definition of the forward graph transform we have
    \begin{equation*}\begin{split}
      \Gamma^*(A_\epsilon, \Pi_{1,\epsilon}, \Pi_{2,\epsilon}) - \Gamma^*(A_0, \Pi_{1,0}, \Pi_{2,0}) &= \Pi_{2,\epsilon} + (\Id -\Pi_{2,\epsilon}) A_\epsilon\left(\restr{\Pi_{2,\epsilon} A_{\epsilon}}{\Pi_{1,\epsilon}(X)} \right)^{-1}\Pi_{2,\epsilon} \\
      &-\Pi_{2,0} - (\Id -\Pi_{2,0}) A_0\left(\restr{\Pi_{2,0} A_{0}}{\Pi_{1,0}(X)} \right)^{-1}\Pi_{2,0}.
    \end{split}\end{equation*}
    Applying Lemmas \ref{lemma:open_sot} and \ref{lemma:inverse_sot_cont}, and \eqref{eq:open_sot_0} yields
    \begin{equation*}
      (\Id -\Pi_{2,\epsilon}) A_\epsilon\left(\restr{\Pi_{2,\epsilon} A_{\epsilon}}{\Pi_{1,\epsilon}(X)} \right)^{-1}\Pi_{2,\epsilon} \to (\Id -\Pi_{2,0}) A_0\left(\restr{\Pi_{2,0} A_{0}}{\Pi_{1,0}(X)} \right)^{-1}\Pi_{2,0}
    \end{equation*}
    in the strong operator topology, which completes the proof.
  \end{proof}
\end{lemma}

\begin{proof}[{The proof of Proposition \ref{prop:measurable_graph_tranform}}]
  We will just prove that $\Gamma^* \circ Y$ is $(\mathcal{F}, \mathcal{S})$-measurable, since the same proof works for $\Gamma_* \circ Y$.
  Let $U \subseteq \LL(X)$ be open in the strong operator topology. Then
  \begin{equation}\label{eq:measurable_graph_tranform_1}
    (\Gamma^* \circ Y)^{-1}(U) = Y^{-1}\left(\bigcup_{n \in \Z^+} (n B_{\LL(X)} )^3 \cap (\Gamma^*)^{-1}(U) \right)  = Y^{-1}\left(\bigcup_{n \in \Z^+} \left(\restr{\Gamma^*}{(n B_{\LL(X)} )^3}\right)^{-1}(U) \right).
  \end{equation}
  By Lemma \ref{lemma:forward_transform_cont}, the map $\restr{\Gamma^*}{(n B_{\LL(X)} )^3}$ is continuous in the strong operator topology for every $n \in \Z^+$, and so $\left(\restr{\Gamma^*}{(n B_{\LL(X)} )^3}\right)^{-1}(U) = U_n \cap \Lambda_d \cap (n B_{\LL(X)})^3$ for some $U_n \in B_{\LL(X)}^3$ that is open in the strong operator topology.
  Since $Y(\Omega) \subseteq \Lambda_d$ we have
  \begin{equation}\label{eq:measurable_graph_tranform_2}
    Y^{-1}(U_n \cap \Lambda_d \cap (n B_{\LL(X)} )^3) = Y^{-1}(U_n \cap (n B_{\LL(X)} )^3).
  \end{equation}
  Since $(n B_{\LL(X)} )^3$ is a separable metric space, for each $n \in \Z^+$ there exists countably many rectangles $\{R_{i,n} \times P_{i,n} \times Q_{i,n}\}_{i \in \Z^+} \subseteq (B_{\LL(X)} )^3$ such that $R_{i,n} ,P_{i,n}$, and $Q_{i,n}$ are open in the strong operator topology on $B_{\LL(X)}$ and so that
  \begin{equation}\label{eq:measurable_graph_tranform_3}
    U_n \cap (n B_{\LL(X)} )^3 = \bigcup_{i\in \Z^+} (R_{i,n} \cap n B_{\LL(X)}) \times (P_{i,n} n B_{\LL(X)}) \times (Q_{i,n} \cap n B_{\LL(X)}).
  \end{equation}
  By \cite[Lemma A.2]{GTQuas1} we have $n B_{\LL(X)} \in \mathcal{S}$, and so $U_n \cap (n B_{\LL(X)} )^3 \in \mathcal{S}^3$, being the countable union of sets in $\mathcal{S}^3$ by \eqref{eq:measurable_graph_tranform_3}.
  Since $Y$ is $(\mathcal{F}, \mathcal{S}^3)$-measurable, by \eqref{eq:measurable_graph_tranform_1}, \eqref{eq:measurable_graph_tranform_2} and \eqref{eq:measurable_graph_tranform_3} we may conclude that $(\Gamma^* \circ Y)^{-1}(U) \in \mathcal{F}$ i.e $\Gamma^* \circ Y$ is $(\mathcal{F}, \mathcal{S})$-measurable.
\end{proof}

\begin{proposition}\label{prop:measurable_inverse}
  Suppose that $(\Omega, \mathcal{F}, \mathbb{P})$ is a Lebesgue space, $X$ is a separable Banach space, and that $Y : \Omega \mapsto \LL(X)^3$ is $(\mathcal{F}, \mathcal{S}^3)$-measurable with $Y(\Omega) \subseteq \Lambda_d$. Then $\Xi \circ Y$ is $(\mathcal{F}, \mathcal{S})$-measurable.
  \begin{proof}
  The proof is identical to that of Proposition \ref{prop:measurable_graph_tranform}, but with Lemma \ref{lemma:inverse_sot_cont} used in place of Lemma \ref{lemma:forward_transform_cont}.
  \end{proof}
\end{proposition}

\begin{lemma}\label{lemma:limit_of_strong_measurable}
    Suppose that $(\Omega, \mathcal{F}, \mathbb{P})$ is a Lebesgue space, $X$ is a separable Banach space, that $\{f_n \}_{n \in \Z^+}$ is a sequence of strongly $(\mathcal{F}, \mathcal{S})$-measurable functions, and that $f : \Omega \to \LL(X)$ $f_n \to f$ almost uniformly and  $\esssup_{\omega \in \Omega} \norm{f} < \infty$.
    Then $f$ is $(\mathcal{F}, \mathcal{S})$-measurable.
    \begin{proof}
        Let $r > \esssup_{\omega \in \Omega} \norm{f}$.
        By changing each $f_n$ on a set of measure $0$ we may assume that $\limsup_{n \to \infty} \sup_{\omega \in \Omega} \norm{f_n(\omega)} \le r$ and that there exists $g : \Omega \to r B_{\LL(X)}$ with $f = g$ a.e. and such that $f_n \to g$ uniformly.
        Since $f_n \to g$ uniformly there exists $N > 0$ such that $f_n(\Omega) \subseteq rB_{\LL(X)}$ for every $n > N$.
        By \cite[Lemma A.2]{GTQuas1} we have $rB_{\LL(X)} \in \mathcal{S}$ and so $f_n$ is $(\mathcal{F}, \mathcal{S}_{r})$-measurable for $n > N$, where $\mathcal{S}_{r}$ denotes the Borel $\sigma$-algebra associated to the restriction of the strong operator topology to $rB_{\LL(X)}$.
        Since $X$ is separable, the strong operator topology on $rB_{\LL(X)}$ is metrisable.
        Thus $g$ is $(\mathcal{F}, \mathcal{S}_{r})$-measurable, being the pointwise limit of measurable functions with values in a metric space.
        For $U \subseteq \LL(X)$ that is open in the strong operator topology we have $U \cap rB_{\LL(X)} \in \mathcal{S}_{r}$, and so $g^{-1}(U) = g^{-1}(U \cap rB_{\LL(X)}) \in \mathcal{F}$.
        Thus $g$ is $(\mathcal{F}, \mathcal{S})$-measurable. Since $f = g$ a.e., we have that $f$ is $(\mathcal{F}, \mathcal{S})$-measurable too.
    \end{proof}
\end{lemma}

\begin{proof}[{The proof of Lemma \ref{lemma:measurable_change_of_basis}}]
  By Lemma \ref{lemma:strong_cont_proj_grassmannian} the map $\omega \mapsto \Pi_\omega(X)$ is measurable.
  By \cite[Corollary 39]{lian2010lyapunov} for every $\epsilon > 0$ there exists measurable maps $e_i : \Omega \to X$, $1 \le i \le d$ such that $\vspan\{e_1(\omega), \dots, e_d(\omega)\} = \Pi_{\omega}(X)$ and for each $1 \le i \le d-1$ we have $\norm{e_i(\omega)} = 1$ and
  \begin{equation*}
    \dist(e_i(\omega), \vspan\{ e_{j}(\omega) : j > i\}) \ge 1 - \epsilon.
  \end{equation*}
  Let $\{\nu_i(\omega)\}_{i=1}^d$ denote the dual basis to $\{e_i(\omega)\}_{i=1}^d$ in $\Pi_{\omega}(X)$. For every $\omega$ we have $\nu_i(\omega) e_j(\omega) = \delta_{ij}$ and so each $\omega \mapsto \nu_i(\omega)e_j(\omega)$ is measurable.
  For $(a_1, \dots, a_d) \in \Q^d + i\Q^d$ set $\psi(a_1, \dots, a_d)(\omega) = \sum_{i=1}^d a_i e_i(\omega)$. We of course have
  \begin{equation*}
    \norm{\nu_i(\omega)} = \sup_{\substack{(a_1, \dots, a_d) \in \Q^d + i\Q^d \\ (a_1, \dots, a_d) \ne (0, \dots, 0)}} \frac{\abs{\nu_i(\omega)(\psi(a_1, \dots, a_d)(\omega))}}{\abs{\psi(a_1, \dots a_d)(\omega)}},
  \end{equation*}
  and that each of the maps
  \begin{equation*}
    \omega \mapsto \frac{\abs{\nu_i(\omega)(\psi(a_1, \dots, a_d)(\omega))}}{\abs{\psi(a_1, \dots, a_d)(\omega)}}
  \end{equation*}
  is measurable.
  Hence $\omega \mapsto \norm{\nu_i(\omega)}$ is measurable, being the supremum of countably many measurable maps.
  By \cite[Proposition 40]{lian2010lyapunov}, each $\nu_i$ may be extended to a strongly measurable map $\nu_i : \Omega \to \LL(X, \C)$ without increasing $\norm{\nu_i(\omega)}$.
  Define $\phi_\omega : X \to \C^d$ by
  \begin{equation*}
    \phi_\omega v = (\nu_1(\omega)(v), \dots, \nu_d(\omega)(v)),
  \end{equation*}
  and set $A_\omega = \phi_\omega \Pi_{\omega}$.
  We clearly have $\ker(A_\omega) = \ker(\Pi_\omega)$, and that $\restr{A_{\omega}}{\Pi_\omega(X)}$ is a bijection. The map $\omega \mapsto \phi_\omega$ is strongly measurable as each of component maps $\omega \mapsto \nu_i(\omega)$ is strongly measurable, and so $\omega \mapsto A_\omega$ is strongly measurable, due to being the composition of strongly measurable maps \cite[Lemma A.5]{GTQuas1}.
  Moreover, we have
  \begin{equation*}
    \left(\restr{A_\omega}{\Pi_\omega(X)}\right)^{-1}(a_1, \dots, a_d) = \sum_{i=1}^d a_i e_i(\omega),
  \end{equation*}
  which implies that $\left(\restr{A_\omega}{\Pi_\omega(X)}\right)^{-1}$ is strongly measurable.

  We may now prove the estimates \eqref{eq:measurable_change_of_basis_0} and \eqref{eq:measurable_change_of_basis_1}.
  For \eqref{eq:measurable_change_of_basis_1} we simply note that if $v \in \Pi_{\omega}$ then
  \begin{equation*}
    \norm{v} \le \sum_{i=1}^d \abs{\nu_i(\omega) v} \le \sqrt{d} \left(\sum_{i=1}^d \abs{\nu_i(\omega) v}^{2}\right)^{1/2} = \sqrt{d} \norm{A_\omega v}.
  \end{equation*}
  Obtaining \eqref{eq:measurable_change_of_basis_0} is more involved.
  For every $v \in \Pi_{\omega}(X)$ one has
  \begin{equation}\label{eq:measurable_change_of_basis_2}
    \norm{v} \ge \max\left\{ \norm{\Pi_{\vspan\{e_i(\omega)\} || \vspan\{e_j(\omega) : j \ne i \}}v} \norm{\Pi_{\vspan\{e_i(\omega)\} || \vspan\{e_j(\omega) : j \ne i \}}}^{-1} : 1 \le i \le d\right\}.
  \end{equation}
  For each $i \in \{1, \dots, d\}$ set $\Pi_{i,\omega} = \Pi_{\vspan\{e_i(\omega)\} || \vspan\{e_j(\omega) : j > i \}}$ and $\Gamma_{i,\omega} = \Id - \Pi_{i,\omega}$.
  Note that
  \begin{equation*}
    \Pi_{\vspan\{e_i(\omega)\} || \vspan\{e_j(\omega) : j \ne i \}} = \Pi_{i, \omega} \left(\prod_{j=1}^{i-1}\Gamma_{i-j, \omega}\right),
  \end{equation*}
  and so $\norm{\Pi_{\vspan\{e_i(\omega)\} || \vspan\{e_j(\omega) : j \ne i \}}} \le 2^{i-1}\prod_{j=1}^i\norm{\Pi_{j, \omega}}$.
  In addition for $i \in \{1, \dots, d-1\}$ we have
  \begin{equation*}\begin{split}
    \norm{\Pi_{i,\omega}} = \sup_{v \in \vspan\{ e_j(\omega) : j \ge i\}} \frac{\norm{\Pi_{i,\omega}(v)}}{\norm{v}} &= \sup_{v' \in \vspan\{ e_j(\omega) : j > i\}} \frac{\norm{e_{i}(\omega)}}{\norm{e_{i}(\omega) - v'}} \\
    &=  \dist(e_i(\omega), \vspan\{ e_{j}(\omega) : j > i\})^{-1} \le (1- \epsilon)^{-1},
  \end{split}\end{equation*}
  while it is clear that $\norm{\Pi_{d,\omega}} = 1$.
  Thus, for every $i \in \{1, \dots, d\}$ we have $\norm{\Pi_{\vspan\{e_i(\omega)\} || \vspan\{e_j(\omega) : j \ne i \}}} \le 2^{d-1} (1- \epsilon)^{-d+1}$.
  Since $\Pi_{\vspan\{e_i(\omega)\} || \vspan\{e_j(\omega) : j \ne i \}}v = \nu_i(\omega)(v)$, from \eqref{eq:measurable_change_of_basis_2} we obtain
  \begin{equation*}
    \norm{v} \ge \left(\frac{2}{1-\epsilon}\right)^{d-1}\max\left\{ \nu_i(\omega)(v) : 1 \le i \le d\right\} \ge \left(\frac{2}{1-\epsilon}\right)^{d-1} \norm{A_\omega v}.
  \end{equation*}
\end{proof}

Our final main result for this appendix concerns the measurability of the determinant map, which is crucial for the proof of stability of Lyapunov exponents in Section \ref{sec:Lyapunov_converge}.
We refer the reader to \cite[Section 2.2]{blumenthal2016volume} for an overview of the basic properties of the determinant.

\begin{proposition}\label{prop:measurable_determinant}
  Suppose that $(\Omega, \mathcal{F}, \mathbb{P})$ is a Lebesgue space, $X$ is a separable Banach space, and that $Y : \Omega \mapsto \LL(X) \times \mathcal{G}_d(X)$ is a $(\mathcal{F}, \mathcal{S} \times \mathcal{B}_{\mathcal{G}_d(X)})$-measurable map. Then $\omega \mapsto \det(Y(\omega))$ is $(\mathcal{F}, \mathcal{B}_{\R})$-measurable.
\end{proposition}

\begin{lemma}\label{lemma:measurable_determinant_bounded}
  Suppose that $X$ is a separable Banach space. For every $d \in \Z^+$ and $s > 0$ the map $\det : \LL(X) \times \mathcal{G}_d(X) \to \R$ is continuous with respect to strong operator topology and the usual Grassmannian topology on $\mathcal{G}_d(X)$ when restricted to $sB_{\LL(X)} \times \mathcal{G}_d(X)$.
  \begin{proof}
    It suffices to prove that if $\{(A_n, E_n) \}_{n \in \Z^+} \subseteq sB_{\LL(X)} \times \mathcal{G}_d(X)$ converges to $(A, E)$ then $\det(A_n | E_n) \to \det(A | E)$.
    Let $F \in \mathcal{G}^d(X)$ be such that $E \oplus F = X$. Since $E_n \to E$ and $\mathcal{N}(F)$ is open in $\mathcal{G}_d(X)$, without loss of generality we may assume that $E_n \oplus F = X$ for every $n$.
    Moreover, by Proposition \ref{prop:graph_chart} we have $\Pi_{E_n || F} \to \Pi_{E || F}$ in the operator norm topology.
    Since $\{A_n\}_{n \in \Z^+}$ is bounded in $\LL(X)$ it follows that $A_n\Pi_{E_n || F} \to A$ in the strong operator topology.

    \paragraph{The case where $\det(A | E) = 0$.}
    If $\det(A_n | E) = 0$ eventually holds for all large $n \in \Z^+$ then we are done. Otherwise we may pass to a subsequence such that $\limsup_{n \to \infty} \det(A_n | E_n)$ is unchanged and $\det(A_n | E) \ne 0$ for every $n \in \Z^+$.
    In particular, we may assume that $\restr{A_n}{E_n}$ is injective for every $n$.
    Since $\det(A | E) = 0$ there exists $f \in \ker(\restr{A}{E}) \setminus \{0\}$.
    Let $G$ be a complementary subspace for $\vspan\{f\}$ in $E$.
    Let $f_n = \Pi_{E_n || F}f$ and $G_n = \Pi_{E_n || F} G$. As $E_n,E \in \mathcal{F}$ by Lemma \ref{lemma:inv_proj} we have that $\restr{\Pi_{E_n || F}}{E}$ is invertible. Thus $E_n = \vspan\{ f_n\} \oplus G_n$ and
    \begin{equation*}
      \Pi_{\vspan\{ f_n\} || G_n} = \Pi_{E_n || F} \Pi_{\vspan\{f\} || G} \left( \restr{\Pi_{E_n || F}}{E} \right)^{-1}.
    \end{equation*}
    Hence as $E_n \to E$ we have $\limsup_{n \to \infty} \norm{\Pi_{\vspan\{ f_n\} || G_n}} \le \norm{\Pi_{\vspan\{f\} || G}}< \infty$. Since each $A_n$ is injective, by \cite[Lemma 2.15]{blumenthal2016volume} there exists $C_d > 0$ such that
    \begin{equation*}
      \det(A_n | E_n) \le C_d \det(A_n | \vspan\{f_n\}) \det(A_n | G_n) \norm{\Pi_{\vspan\{ f_n\} || G_n}}.
    \end{equation*}
    On one hand we have $\det(A_n | \vspan\{f_n\}) \le \norm{A_n \Pi_{E_n || F}f} \to 0$,
    while on the other we have $\det(A_n | G_n) \le \norm{A_n}^{d-1}$.
    Thus
    \begin{equation*}
      \det(A_n | E_n) \le C_d \norm{A_n}^{d-1} \det(A_n | \vspan\{f_n\}) \norm{\Pi_{\vspan\{ f_n\} || G_n}} \to 0,
    \end{equation*}
    as required.

    \paragraph{The convergence of $A_n E_n$ to $A E$.}
    Henceforth we shall assume that $\det(A | E) \ne 0$, and so  $\restr{A}{E}$ has trivial kernel.
    Fix a basis $\{v_i\}_{i = 1}^d$ for $E$ such that $\{A v_i \}_{i=1}^d$ is a normalised Auerbach basis for $AE$. Then by \cite[Corollary A.7]{quas2019explicit} for every set of scalars $\{a_i \}_{i=1}^d$ we have
    \begin{equation}\label{eq:auerbach}
       \frac{1}{\sqrt{d}} \left(\sum_{i=1}^d \abs{a_i}^2\right)^{1/2} \le \norm{\sum_{i=1}^d a_i Av_i} \le \sqrt{d} \left(\sum_{i=1}^d \abs{a_i}^2\right)^{1/2}.
    \end{equation}
    For each $f \in E$ with $\norm{Af} = 1$ we write $f = \sum_{i=1}^d a_i v_i$. Then
    \begin{equation*}\begin{split}
      \dist(Af,A_n E_n) \le\norm{\sum_{i=1}^d a_i Av_i - \sum_{i=1}^d a_i A_n\Pi_{E_n || F}v_i} &\le \left(\sup_{1 \le i\le d} \norm{ (A- A_n\Pi_{E_n || F})v_i}\right) \sum_{i=1}^d \abs{a_i} \\
      &\le d \left(\sup_{1 \le i\le d} \norm{ (A - A_n\Pi_{E_n || F} )v_i}\right),
    \end{split}\end{equation*}
    where we used \eqref{eq:auerbach} and the fact that $\norm{Af} = 1$ to obtain the last inequality.
    By taking the supremum over $f \in E$ with $\norm{Af} = 1$ and letting $n \to \infty$ we observe that $\gap(A E, A_n E_n) \to 0$.
    By \cite[IV \S 2, Corollary 2.6]{kato1966perturbation} it follows that $\dim(A E) \le \dim(A_n E_n)$ for sufficiently large $n$. Since $\restr{A}{E}$ has trivial kernel we have $\dim(A E) = \dim(E) = \dim(E_n) \ge \dim(A_n E_n)$ and so $\dim(A E) = \dim(A_n E_n)$.
    By \cite[Lemma 2.6]{blumenthal2016volume} we therefore have $\gap(A_n E_n, A E) \to 0$, and so $A_n E_n \to A E$ in $\mathcal{G}_d(X)$ by \eqref{eq:gap_equiv}.

    \paragraph{The case where $\det(A | E) \ne 0$.}
    Let $F' \in \mathcal{G}^d(X)$ be such that $AE \oplus F' = X$. Since $A_n E_n \to AE$, without loss of generality we may assume that $A_n E_n \oplus F' = X$ for every $n$ and that $\Pi_{A_n E_n || F'} \to \Pi_{AE || F'}$.
    By Lemma \ref{lemma:inv_proj} the map $\restr{\Pi_{A_nE_n || F'}}{A E}$ is invertible, and so the pushforward of $m_{A E}$ under $\Pi_{A_n E_n || F'}$ is a well defined, translation invariant measure on $A_n E_n$.
    Since the Haar is unique up to scaling we get
    \begin{equation*}
      m_{A_n E_n} = \frac{m_{A_n E_n}( B_{A_n E_n})}{m_{AE}( \left(\restr{\Pi_{A_nE_n || F'}}{A E}\right)^{-1}(B_{A_n E_n}))} (m_{AE} \circ \left(\restr{\Pi_{A_nE_n || F'}}{A E}\right)^{-1}).
    \end{equation*}
    For notational convenience we set $\Gamma_n = \left(\restr{\Pi_{A_nE_n || F'}}{A E}\right)^{-1}$.
    By the definition of the determinant \eqref{eq:det}, one has
    \begin{equation}\begin{split}\label{eq:measurable_determinant_bounded_1}
      \abs{m_{E}(B_E)}&\abs{\det(A | E) - \det(A_n | E_n)} \\
      &=  \abs{m_{AE}(A(B_E)) - \frac{m_{A_n E_n}( B_{A_n E_n})}{m_{AE}( \Gamma_n(B_{A_n E_n}))} \frac{m_{AE}( \Gamma_nA_n(B_{E_n}))}{m_{AE}(\Gamma_n A_n\Pi_{E_n || F }B_E)} m_{AE}(\Gamma_n A_n\Pi_{E_n || F }B_E) }.
    \end{split}\end{equation}
    As $\norm{\restr{\Pi_{A_nE_n || F'}}{A E}}^{-1} B_{AE} \subseteq \Gamma_n(B_{A_n E_n}) \subseteq \norm{\Gamma_n} B_{AE}$ we have
    \begin{equation}\label{eq:measurable_determinant_bounded_2}
      \norm{\restr{\Pi_{A_nE_n || F'}}{A E}}^{-d} m_{AE}(B_{AE})\le m_{AE}(\Gamma_n(B_{A_n E_n}))) \le \norm{\Gamma_n}^d m_{AE}(B_{AE}).
    \end{equation}
    Since $A_nE_n \to AE$ we have $\norm{\Gamma_n - \Id} \to 0$ by Proposition \ref{prop:graph_chart} and the definition of the graph representation of $\mathcal{N}(F')$.
    Applying the facts that $m_{AE}(B_{AE}) = m_{A_n E_n}( B_{A_n E_n})$, $\norm{\restr{\Pi_{A_nE_n || F'}}{AE}}\to 1$, and $\norm{\Gamma_n}\to 1$ to \eqref{eq:measurable_determinant_bounded_2} yields
    \begin{equation}\label{eq:measurable_determinant_bounded_3}
      \lim_{n \to \infty} \frac{m_{A_n E_n}( B_{A_n E_n})}{m_{AE}( \Gamma_n(B_{A_n E_n}))} = 1.
    \end{equation}
    Similarly we have $\norm{\restr{\Pi_{E_n || F }}{E}}^{-1} \Pi_{E_n || F }B_E\subseteq B_{E_n} \subseteq \norm{\left(\restr{\Pi_{E_n || F }}{E}\right)^{-1}} \Pi_{E_n || F }B_E$ and so
    \begin{equation*}
      \norm{\restr{\Pi_{E_n || F }}{E}}^{-d} m_{AE}(\Gamma_n A_n\Pi_{E_n || F }B_E) \le m_{AE}(\Gamma_n A_nB_{E_n}) \le \norm{\left(\restr{\Pi_{E_n || F }}{E}\right)^{-1}}^{d} m_{AE}(\Gamma_n A_n\Pi_{E_n || F }B_E).
    \end{equation*}
    Arguing as before, it therefore follows that
    \begin{equation}\label{eq:measurable_determinant_bounded_4}
      \lim_{n \to \infty} \frac{m_{AE}(\Gamma_n A_n\Pi_{E_n || F }B_E)}{m_{AE}(\Gamma_n A_nB_{E_n})} = 1.
    \end{equation}
    Note that
    \begin{equation*}
      \abs{m_{AE}(A(B_E)) - m_{AE}( \Gamma_nA_n\Pi_{E_n || F }(B_{E}))} \abs{m_{E}(B_E)}^{-1} = \abs{\det(A | E) - \det(\Gamma_nA_n\Pi_{E_n || F } | E)}.
    \end{equation*}
    Since $\Gamma_nA_n\restr{\Pi_{E_n || F }}{E} \in \LL(E, AE)$ for every $n$, we have $\Gamma_nA_n\restr{\Pi_{E_n || F }}{E} \to A$ in the operator norm on $\LL(E, AE)$. Hence by \cite[Lemma 2.20]{blumenthal2016volume} we have
    $\lim_{n \to \infty} \det(\Gamma_nA_n\Pi_{E_n || F } | E) = \det(A | E)$. Combining this with \eqref{eq:measurable_determinant_bounded_1}, \eqref{eq:measurable_determinant_bounded_3} and \eqref{eq:measurable_determinant_bounded_4} completes the proof.
  \end{proof}
\end{lemma}

\begin{proof}[{The proof of Proposition \ref{prop:measurable_determinant}}]
  The proof is similar to that of Proposition \ref{prop:measurable_graph_tranform}, but we include it for completeness.
  For every open $U \in \R$ we have
  \begin{equation}\label{eq:measurable_determinant_1}
     \{ \omega : \det(Y(\omega)) \in U \} = Y^{-1}\left( \bigcup_{n \in \Z^+} \left\{ (A, E) \in n B_{\LL(X)} \times \mathcal{G}_d(X) : \det(A | E) \in U \right\} \right).
  \end{equation}
  By Lemma \ref{lemma:measurable_determinant_bounded} for each $n \in \Z^+$ there exists a set $U_n \in nB_{\LL(X)} \times \mathcal{G}_d(X)$ that is open in the product of the relative strong operator topology and the Grassmannian topology, and such that
  \begin{equation*}
   U_n = \left\{ (A, E) \in n B_{\LL(X)} \times \mathcal{G}_d(X) : \det(A | E) \in U \right\}.
  \end{equation*}
  Since $n B_{\LL(X)}$ and $\mathcal{G}_d(X)$ are both are separable metric spaces (for the later claim see \cite[Lemma B.11]{GTQuas1}) we may write $U_n \cap (n B_{\LL(X)} \times \mathcal{G}_d(X))$ as the union of countably many rectangles $\{(R_{n,i} \cap n B_{\LL(X)}) \times Q_{n,i}\}_{i \in \Z^+}$, where $R_{n,i}$ is open in the strong operator topology on $\LL(X)$ and $Q_{n,i}$ is open in the Grassmannian topology.
  We have $n B_{\LL(X)} \in \mathcal{S}$ by \cite[Lemma A.2]{GTQuas1}, and so $R_{n,i} \in \mathcal{S}$ for every $i, n \in \Z^+$.
  It follows that each $U_n$ is a countable union of sets in $\mathcal{S} \times \mathcal{B}_{\mathcal{G}_d(X)}$, and so $U_n \in \mathcal{S} \times \mathcal{B}_{\mathcal{G}_d(X)}$ for every $n$. Thus $\bigcup_{n \in \Z^+} U_n \in \mathcal{S} \times \mathcal{B}_{\mathcal{G}_d(X)}$.
  Since $Y$ is $(\mathcal{F},\mathcal{S}\times\mathcal{B}_{\mathcal{G}_d(X)})$-measurable, it follows that the left hand side \eqref{eq:measurable_determinant_1} must be in $\mathcal{F}$.
  Thus $\omega \mapsto \det(Y(\omega))$ is $(\mathcal{F}, \mathcal{B}_{\R})$-measurable, as required.
\end{proof}

\bibliographystyle{siam}
\bibliography{bibliography}

\begin{thebibliography}{10}

\bibitem{alexiewicz1953two}
{\sc A.~Alexiewicz}, {\em On the two-norm convergence}, Studia Mathematica, 1
  (1953), pp.~49--56.

\bibitem{alexiewicz1958linear}
{\sc A.~Alexiewicz and Z.~Semadeni}, {\em Linear functionals on two-norm
  spaces}, Studia Mathematica, 17 (1958), pp.~121--140.

\bibitem{alexiewicz1959two}
\leavevmode\vrule height 2pt depth -1.6pt width 23pt, {\em The two-norm spaces
  and their conjugate spaces}, Studia Mathematica, 3 (1959), pp.~275--293.

\bibitem{baladi1997correlation}
{\sc V.~Baladi}, {\em Correlation spectrum of quenched and annealed equilibrium
  states for random expanding maps}, Communications in mathematical physics,
  186 (1997), pp.~671--700.

\bibitem{viviane2000positive}
\leavevmode\vrule height 2pt depth -1.6pt width 23pt, {\em {Positive transfer
  operators and decay of correlations}}, vol.~16, World scientific, 2000.

\bibitem{baladi2018dynamical}
\leavevmode\vrule height 2pt depth -1.6pt width 23pt, {\em {Dynamical zeta
  functions and dynamical determinants for hyperbolic maps}}, Springer, 2018.

\bibitem{baladi1996random}
{\sc V.~Baladi, A.~Kondah, and B.~Schmitt}, {\em {Random correlations for small
  perturbations of expanding maps}}, Random and Computational Dynamics, 4
  (1996), pp.~179--204.

\bibitem{blank2002ruelle}
{\sc M.~Blank, G.~Keller, and C.~Liverani}, {\em {Ruelle--Perron--Frobenius
  spectrum for Anosov maps}}, Nonlinearity, 15 (2002), p.~1905.

\bibitem{blumenthal2016volume}
{\sc A.~Blumenthal}, {\em {A volume-based approach to the multiplicative
  ergodic theorem on {B}anach spaces}}, Discrete \& Continuous Dynamical
  Systems-A, 36 (2016), pp.~2377--2403.

\bibitem{bogenschutz2000stochastic}
{\sc T.~Bogensch{\"u}tz}, {\em {Stochastic stability of invariant subspaces}},
  Ergodic Theory and Dynamical Systems, 20 (2000), pp.~663--680.

\bibitem{lawsOfChaos}
{\sc A.~Boyarsky and P.~G{\'o}ra}, {\em Laws of Chaos: invariant measures and
  dynamical systems in one dimension}, Birkh{\"a}user Boston, 1997.

\bibitem{brin2002introduction}
{\sc M.~Brin and G.~Stuck}, {\em Introduction to dynamical systems}, Cambridge
  university press, 2002.

\bibitem{conze2007limit}
{\sc J.-P. Conze and A.~Raugi}, {\em {Limit theorems for sequential expanding
  dynamical systems}}, in Ergodic Theory and Related Fields: 2004-2006 Chapel
  Hill Workshops on Probability and Ergodic Theory, University of North
  Carolina Chapel Hill, North Carolina, vol.~430, American Mathematical Soc.,
  2007, p.~89.

\bibitem{cooper1987saks}
{\sc J.~Cooper}, {\em {Saks Spaces and Applications to Functional Analysis}},
  no.~139 in North-Holland Mathematics Studies, North-Holland, 1987.

\bibitem{cooper2011saks}
\leavevmode\vrule height 2pt depth -1.6pt width 23pt, {\em {Saks Spaces and
  Applications to Functional Analysis}}, North-Holland Mathematics Studies,
  Elsevier Science, 2011.

\bibitem{dellnitz2000isolated}
{\sc M.~Dellnitz, G.~Froyland, and S.~Sertl}, {\em On the isolated spectrum of
  the perron-frobenius operator}, Nonlinearity, 13 (2000), p.~1171.

\bibitem{froyland2014detecting}
{\sc G.~Froyland, C.~Gonz{\'a}lez-Tokman, and A.~Quas}, {\em Detecting isolated
  spectrum of transfer and {K}oopman operators with {F}ourier analytic tools},
  J. Comput. Dyn, 1 (2014), pp.~249--278.

\bibitem{froyland2014stability}
\leavevmode\vrule height 2pt depth -1.6pt width 23pt, {\em {Stability and
  approximation of random invariant densities for {L}asota--{Y}orke map
  cocycles}}, Nonlinearity, 27 (2014), p.~647.

\bibitem{froyland2010coherentpf}
{\sc G.~Froyland, S.~Lloyd, and A.~Quas}, {\em Coherent structures and isolated
  spectrum for perron--frobenius cocycles}, Ergodic Theory and Dynamical
  Systems, 30 (2010), pp.~729--756.

\bibitem{froyland2013semi}
\leavevmode\vrule height 2pt depth -1.6pt width 23pt, {\em {A semi-invertible
  {O}seledets Theorem with applications to transfer operator cocycles}},
  Discrete \& Continuous Dynamical Systems-A, 33 (2013), pp.~3835--3860.

\bibitem{froyland2010coherent}
{\sc G.~Froyland, S.~Lloyd, and N.~Santitissadeekorn}, {\em Coherent sets for
  nonautonomous dynamical systems}, Physica D: Nonlinear Phenomena, 239 (2010),
  pp.~1527--1541.

\bibitem{galatolo2015statistical}
{\sc S.~Galatolo}, {\em {Statistical properties of dynamics. Introduction to
  the functional analytic approach}}, arXiv preprint arXiv:1510.02615,  (2015).

\bibitem{gonzalez2018multiplicative}
{\sc C.~Gonz{\'a}lez-Tokman}, {\em Multiplicative ergodic theorems for transfer
  operators: towards the identification and analysis of coherent structures in
  non-autonomous dynamical systems}, Contemp. Math, 709 (2018), pp.~31--52.

\bibitem{GTQuas1}
{\sc C.~Gonz{\'a}lez-Tokman and A.~Quas}, {\em {A semi-invertible operator
  Oseledets theorem}}, Ergodic Theory and Dynamical Systems, 34 (2014),
  pp.~1230--1272.

\bibitem{gonzalez2015concise}
\leavevmode\vrule height 2pt depth -1.6pt width 23pt, {\em A concise proof of
  the multiplicative ergodic theorem on banach spaces}, Journal of Modern
  Dynamics, 9 (2015), pp.~237--255.

\bibitem{gonzalez2018stability}
\leavevmode\vrule height 2pt depth -1.6pt width 23pt, {\em {Stability and
  collapse of the {L}yapunov spectrum for {P}erron-{F}robenius operator
  cocycles}}, arXiv preprint arXiv:1806.08873,  (2018).

\bibitem{gouezel2006banach}
{\sc S.~Gou{\"e}zel and C.~Liverani}, {\em {Banach spaces adapted to Anosov
  systems}}, Ergodic Theory and dynamical systems, 26 (2006), pp.~189--217.

\bibitem{hennion1993theoreme}
{\sc H.~Hennion}, {\em {Sur un th{\'e}oreme spectral et son application aux
  noyaux lipchitziens}}, Proceedings of the American Mathematical Society, 118
  (1993), pp.~627--634.

\bibitem{kato1966perturbation}
{\sc T.~Kato}, {\em {Perturbation theory for linear operators}}, Grundlehren
  der mathematischen Wissenschaften, Springer Berlin Heidelberg, 1966.

\bibitem{katznelson2002introduction}
{\sc Y.~Katznelson}, {\em {An introduction to harmonic analysis}}, Cambridge
  University Press, 2002.

\bibitem{keller1982stochastic}
{\sc G.~Keller}, {\em {Stochastic stability in some chaotic dynamical
  systems}}, Monatshefte f{\"u}r Mathematik, 94 (1982), pp.~313--333.

\bibitem{keller1985generalized}
\leavevmode\vrule height 2pt depth -1.6pt width 23pt, {\em Generalized bounded
  variation and applications to piecewise monotonic transformations},
  Zeitschrift f{\"u}r Wahrscheinlichkeitstheorie und Verwandte Gebiete, 69
  (1985), pp.~461--478.

\bibitem{keller1999stability}
{\sc G.~Keller and C.~Liverani}, {\em {Stability of the spectrum for transfer
  operators}}, Annali della Scuola Normale Superiore di Pisa-Classe di Scienze,
  28 (1999), pp.~141--152.

\bibitem{keller2004eigenfunctions}
{\sc G.~Keller and H.~H. Rugh}, {\em {Eigenfunctions for smooth expanding
  circle maps}}, Nonlinearity, 17 (2004), p.~1723.

\bibitem{lian2010lyapunov}
{\sc Z.~Lian and K.~Lu}, {\em {{L}yapunov exponents and invariant manifolds for
  random dynamical systems in a Banach space}}, American Mathematical Soc.,
  2010.

\bibitem{liverani2004invariant}
{\sc C.~Liverani}, {\em Invariant measures and their properties. a functional
  analytic point of view}, Dynamical systems. Part II,  (2004), pp.~185--237.

\bibitem{novel2017p}
{\sc M.~Novel}, {\em p-dimensional cones and applications}, arXiv preprint
  arXiv:1712.00762,  (2017).

\bibitem{oseledets1968multiplicative}
{\sc V.~I. Oseledets}, {\em {A multiplicative ergodic theorem. Characteristic
  Ljapunov, exponents of dynamical systems}}, Trudy Moskovskogo
  Matematicheskogo Obshchestva, 19 (1968), pp.~179--210.

\bibitem{quas2019explicit}
{\sc A.~Quas, P.~Thieullen, and M.~Zarrabi}, {\em {Explicit bounds for
  separation between Oseledets subspaces}}, Dynamical Systems,  (2019),
  pp.~1--44.

\bibitem{schaefer1986topological}
{\sc H.~Schaefer}, {\em {Topological Vector Spaces}}, Graduate texts in
  mathematics, Springer, 1986.

\bibitem{sedro2018etude}
{\sc J.~Sedro}, {\em {\'E}tude de syst{\`e}mes dynamiques avec perte de
  r{\'e}gularit{\'e}}, PhD thesis, 2018.

\bibitem{slipantschuk2013analytic}
{\sc J.~Slipantschuk, O.~F. Bandtlow, and W.~Just}, {\em Analytic expanding
  circle maps with explicit spectra}, Nonlinearity, 26 (2013), p.~3231.

\bibitem{terry_tao_lemma}
{\sc T.~Tao}, {\em A quick application of the closed graph theorem}.
\newblock
  \url{https://terrytao.wordpress.com/2016/04/22/a-quick-application-of-the-closed-graph-theorem/#comments}.
\newblock Accessed: 2019-11-21.

\bibitem{thieullen1987fibres}
{\sc P.~Thieullen}, {\em Fibr{\'e}s dynamiques asymptotiquement compacts
  exposants de {L}yapounov. {E}ntropie. {D}imension}, in Annales de l'Institut
  Henri Poincare (C) Non Linear Analysis, vol.~4, Elsevier, 1987, pp.~49--97.

\bibitem{varzaneh2019dynamical}
{\sc M.~G. Varzaneh and S.~Riedel}, {\em A dynamical theory for singular
  stochastic delay differential equations with a multiplicative ergodic theorem
  on fields of banach spaces}, arXiv preprint arXiv:1903.01172,  (2019).

\end{thebibliography}

\end{document}